\documentclass[ba]{imsart}
\pubyear{2023}
\volume{TBA}
\issue{TBA}
\firstpage{1}
\lastpage{1}

\usepackage{amsthm}
\usepackage{amsmath}
\usepackage{amsfonts}
\usepackage{amssymb}
\usepackage{natbib}
\usepackage[colorlinks,citecolor=blue,urlcolor=blue,filecolor=blue,backref=page]{hyperref}
\usepackage[english]{babel}
\usepackage{bbm}
\usepackage[dvipsnames]{xcolor}
\usepackage{graphicx}
\usepackage{rotating}
\usepackage{array}

\startlocaldefs

\numberwithin{equation}{section}

\theoremstyle{plain}
\newtheorem{theorem}{Theorem}[section]
\newtheorem{lemma}{Lemma}[section]

\newtheorem{corollary}{Corollary}[section]

\theoremstyle{remark}

\newtheorem*{assumption*}{Assumption}
\newtheorem{example}{Example}[section]
\newtheorem{remark}{Remark}[section]

\newenvironment{mycenter}[1][\topsep]
{\setlength{\topsep}{#1}\par\kern\topsep\centering}
	{\par\kern\topsep}

\def\cov{\mathop{\rm cov}\nolimits}
\def\var{\mathop{\rm var}\nolimits}
\def\Law{{\cal L}}
\def\HH{\mathbb{H}}
\def\PP{\mathbb{P}}
\def\RR{\mathbb{R}}
\def\BB{\mathbb{B}}
\def\E{\mathbb{E}}
\def\F{\mathcal{F}}
\def\L{\mathcal{L}}
\def\mX{\mathfrak{X}}
\def\X{\mathcal{X}}
\def\mY{\mathfrak{Y}}
\def\Y{\mathcal{Y}}
\def\ub{\underline b}
\let\weak=\rightsquigarrow
\mathchardef\given="626A

\def\BL{\text{BL}}
\def\DP{\text{DP}}
\def\ra{\rightarrow}
\def\sumin{\sum_{i=1}^n}
\def\prodin{\prod_{i=1}^n}
\def\l{\lambda}

\def\a{\alpha}
\newcommand{\indep}{\protect\mathpalette{\protect\independenT}{\perp}}\def\independenT#1#2{\mathrel{\rlap{$#1#2$}\mkern2mu{#1#2}}}
\newcommand{\ind}{\, \raise-2pt\hbox{$\stackrel{\makebox{\scriptsize ind}}{\sim}$}\, }
\newcommand{\iid}{\, \raise-2pt\hbox{$\stackrel{\makebox{\scriptsize iid}}{\sim}$}\,}

\endlocaldefs

\begin{document}


\begin{frontmatter}
\title{Bayesian sensitivity analysis for a missing data model}
\runtitle{Bayesian sensitivity analysis}

\begin{aug}

\author{\fnms{B.} \snm{Eggen}\thanksref{t2}\corref{}\ead[label=e1]{b.eggen@tudelft.nl}}
\address{Delft Institute of Applied Mathematics, Delft University of Technology \\ \printead{e1}}

\author{\fnms{S. L.} \snm{van der Pas}\thanksref{t1}\ead[label=e2]{s.l.vanderpas@amsterdamumc.nl}}
\address{Amsterdam UMC location Vrije Universiteit Amsterdam, Epidemiology and Data Science, and
	Amsterdam Public Health, Methodology\\ \printead{e2} }


\author{\fnms{A. W.} \snm{van der Vaart}\thanksref{t2}\ead[label=e3]{a.w.vandervaart@tudelft.nl}}
\address{Delft Institute of Applied Mathematics, Delft University of Technology \\ \printead{e3} }

\thankstext{t1}{The research leading to these results is partly financed by  grant VI.Veni.192.087
	by the Netherlands Organisation for Scientific Research (NWO).}
\thankstext{t2}{The research leading to these results is partly financed by a Spinoza prize awarded
	by the Netherlands Organisation for Scientific Research (NWO).}

\runauthor{B. Eggen et al.}


\end{aug}

\begin{abstract}
	In causal inference, sensitivity analysis is important to assess the robustness of study conclusions to key assumptions. We perform sensitivity analysis of the assumption that missing outcomes are missing completely at random. We follow a Bayesian approach, which is nonparametric for the outcome distribution and can be combined with an informative prior on the sensitivity parameter. We give insight in the posterior and
	provide theoretical guarantees in the form of
	Bernstein-von Mises theorems for estimating the mean outcome. We study different parametrisations of the model involving Dirichlet process priors on the distribution of the outcome and on the distribution of the outcome conditional on the subject being treated. We show that these parametrisations incorporate a prior on
	the sensitivity parameter in different ways and discuss the relative merits.
	We also present a simulation study, showing the performance of the methods in finite sample scenarios.
\end{abstract}



\end{frontmatter}


\section{Introduction}
Drawing causal conclusions from missing data usually rests on
unverifiable assumptions. Sensitivity analysis allows to assess the
robustness of study conclusions to these assumptions. A Bayesian
approach allows to incorporate prior beliefs on parameters
that govern the sensitivity, and can lead to a single summary through the
posterior distribution. However, theoretical support for many Bayesian
methods has been lacking.

In this paper we follow \cite{Scharfstein2003} in performing sensitivity analysis for the assumption
that missing outcomes are missing completely at random (MCAR),
i.e.\ the assumption that observing an individual's outcome is independent of that outcome \citep{LittleRubin}.
Under MCAR the observed outcomes can be considered representative for the missing outcomes, and an ordinary analysis of the observed
outcomes yields unbiased conclusions for the full population.
Yet in many applications, MCAR is not plausible. For example, in HIV research, patients who drop out tend to have worse outcomes
than those who stay on the study (see e.g.\ \cite{Hammeretal96}, \cite{Balzeretal2020}, \cite{Scharfstein2003}).
A naive analysis of only the observed outcomes would tend to result in  biased, overly optimistic conclusions.


With sensitivity analysis, we aim to answer the question: how will our conclusions change under
plausible deviations from MCAR? For example, how does our estimate of drug effectiveness change, if data is missing not at random? To get a mathematical grip on this question, we need to quantify a `deviation from MCAR'. We do so by introducing a sensitivity parameter. Ideally, the sensitivity parameter has a clear interpretation, so that we can be precise about what deviations from MCAR can be considered `plausible'. Moreover, we wish to impose as little structure as possible on the distribution of the observed data.


In this paper we denote the distributions of the observed and unobserved outcomes by $P_1$ and $P_0$, respectively,
and are interested in the distribution $P$ of the outcome in the full population, which is the mixture $P=(1-p)P_0+pP_1$, for
$1-p$ the probability that an observation is missing. Under MCAR the distributions $P_0$ and $P_1$ coincide, and hence
the distribution of interest $P$ agrees with the distribution $P_1$ of the observed outcomes.
However, without further assumptions, the distribution $P$ is not identifiable from the observed data. Following
\cite{Scharfstein2003}, we adopt the working model
\begin{equation*}
	dP_0(y) \propto e^{q(y)}\,dP_1(y),
\end{equation*}
where $q$ is a given function, referred to as the sensitivity function. Any given function $q$
allows to identify the distribution $P$ from the observed data. The MCAR assumption corresponds to
the special case that the function $q$ is constant. Sensitivity analysis may consist of
performing the statistical analysis for every function $q$ in a given parameterised collection
$\{q_\a: \a\in A\}$ of functions, and compare the results for different values of the parameter $\a$,
referred to as the sensitivity parameter. Here it is not assumed that the working model is correct,
or that there exists a `true' sensitivity function or parameter.

Typically, and certainly if we adopt a non-parametric model for $P$, the sensitivity parameter will not
be identifiable from the data.  Further inference must then
rely on subjective  knowledge of the substantive scientist. He or she might deem particular values of $\a$
plausible for their particular study and conduct their data analysis accordingly.
In this subjective approach, the Bayesian paradigm is natural, as this allows a  range of possible values
accompanied by weights, as an alternative to assigning a single, fixed value to the sensitivity parameter.

In \cite{Scharfstein2003} the authors followed a non-parametric Bayesian approach,
putting the classical Dirichlet process
prior on $P$. They conjectured a Bernstein-von Mises theorem for estimating the mean of $P$, conditional on the sensitivity parameter.
One aim of the present paper is to prove the validity of their conjecture.
For this purpose, we first show that the Dirichlet prior on $P$
leads to an extended Gamma type prior on the observed data distribution. Next we develop novel theory
for posterior distributions corresponding to such priors.

We also introduce an alternative non-parametric Bayesian modelling
strategy and compare the results of the approaches. It turns out that the two approaches
are asymptotically equivalent given the sensitivity function, but differ when the sensitivity
function is also equipped with a prior distribution. In both approaches the final inference
is a mixture of the inferences for known sensitivity functions, but the two approaches
differ in the mixing distribution. In one approach this is just the prior over the
sensitivity function, but in the other approach the mixing distribution is dependent
on the estimated data distribution. This is true, even though in both cases the sensitivity parameter is
modelled a priori independent from the other parameters. We discuss this
somewhat surprising finding in Section~\ref{SectionComparison}, when comparing
the two approaches.

Our theoretical analysis is `frequentist Bayesian' in nature in the sense
that we adopt prior modelling as a method to derive posterior distributions, but perform theoretical analysis
under the assumption that the observations are sampled from a fixed distribution $H_0$.
This observational distribution is arbitrary and independent of the sensitivity parameter.

The paper is organised as follows. After presenting
the model in Section~\ref{sec:Model}, we investigate two different ways of putting a prior
on the model, and derive  asymptotic distributions  relating to these priors
in Sections~\ref{sec:paramH} and~\ref{SectionEtaP}. In Section~\ref{SectionComparison}
we compare the results of the two approaches and discuss the merits of the two priors.
In Section~\ref{SectionSimulation}, we present a simulation study and in Section~\ref{SectionDiscussion}
we discuss open questions and further work. In particular, we discuss extension of
our results to general missing at random (MAR) models and causal inference,
which both require that covariates are included in the analysis. There we also briefly
review other Bayesian approaches to causal inference.

Section~\ref{SectionExtendedGammaPosterior} contains the main technical part of the
paper, which consists of an analysis of the posterior process based on the
extended gamma normalised completely random measure, that arises when using
one of our two priors. Part of the simulations and proofs can be found in the supplementary material (\cite{Supplement}), along with the derivation of the asymptotic information bound for the model.

\subsection{Notation}
We use operator notation $Pg=\int g\,dP$ for the expectation of a function $g$ under a probability measure $P$.
The distribution of a random variable $X$ or its conditional distribution given a random variable $Y$
are denoted by $\L(X)$  or $\L(X\given Y)$.
We write $X\given Y\sim P$ if $\L(X\given Y)=P$.
We also write $X\indep Y\given Z$ if the variables
$X$ and $Y$ are conditionally independent given $Z$.
Furthermore we write $\PP_n$, $\HH_n$ or $\PP_{1,n}$ for the empirical distribution of a sample of observations,
while $\BB_P$ is a $P$-Brownian bridge process (defined below).
We write $X_n\given Y_n\weak Z$ if the conditional distribution of $X_n$ given $Y_n$ converges in distribution
to $Z$, in probability or almost surely. The latter concept is defined formally as convergence to zero,
in probability or almost surely,
of the bounded Lipschitz distance between $\L(X_n\given Y_n)$ and $\L(Z)$, as explained in the
beginning of Section~\ref{SectionExtendedGammaPosterior}.

\section{Model specification}
\label{sec:Model}
In this section we define the missing outcomes model and introduce sensitivity functions.

Let the `full data' consist of a random variable $Y$ with values in a measurable space $(\mY,\Y)$, and a Bernoulli variable
$R$, with values in $\{0,1\}$, which registers whether the full data is observed or not.
Instead of $(Y,R)$, we only observe the pair $(X,R)$, for $X$ defined by,
for some  arbitrary symbol $\ast$, not contained in $\mY$,
\begin{align}
	\label{EqDefX}
	X = \begin{cases}
		Y, & R=1, \\
		\ast, & R=0.
	\end{cases}
\end{align}
Thus
$(X,R)$ takes values in $\mX\times\{0,1\}$, for $\mX = \mY \cup \{\ast\}$. In the case
that $\mY=(0,\infty)$, it is customary to choose $\ast$ equal to $0\in\RR$,
and we then have the identity $X = RY$, but this special structure
is irrelevant for the following. In fact, the results in the following are valid for $(\mY,\Y)$
a general Polish space with its Borel $\sigma$-field, and $\mX$ equipped
with the $\sigma$-field generated by $\mY$ and the point $\ast$. (Since $X=\ast$ if and only if
$R=0$, the variable $R$ in the observation $(X,R)$ is redundant, but we shall keep it for ease of notation.)

We assume that we observe a sample $(X_1,R_1),\ldots, (X_n,R_n)$ of $n$ i.i.d.\ copies of $(R,X)$, and are
interested in estimating characteristics of the distribution of $Y$. In particular, we consider estimating
the expectation
$\E_P g(Y)$, for a given measurable function $g : \mY \to \RR$, for instance
the mean outcome $\E_P Y$ or the distribution function $\E_P 1_{Y\le y}$,
in the case that $\mY=(0,\infty)$.

Denote the distribution of $Y$ by $P$, and denote the conditional distributions  of $Y$
given $R=0$ and $R=1$ by $P_0$ and $P_1$.
Then  $P = p P_1 + (1-p)P_0$, for $p=\Pr(R=1)$ the success probability of $R$.
The distribution $P_1$ is identifiable from $X$, but $P_0$ and hence $P$ are not.  Following \cite{Scharfstein2003},
we assume as a working hypothesis, that for a given measurable function $q : \mY \to\RR$ such that $\int e^q\, dP_1 < \infty$,
\begin{align}
	\label{EqRelationP0P1}
	P_0(A) = \frac{\int_A e^q\, dP_1}{\int e^q\,dP_1}, \qquad A \in \Y.
\end{align}
In particular, the measures  $P_0$ and $P_1$ possess the same support.
The identity $P=(1-p)P_0+pP_1$ now implies the relationship
\begin{align}
	\label{EqRelationPP1}
	P(A)=p\int_A \left(1 + \frac{1-p}{p \int e^q dP_1} e^q \right)dP_1   , \qquad A \in \Y.
\end{align}
Since $p$ and $P_1$ are identifiable from the distribution of the data, it is then clear that,
for a given function $q$, the distribution $P$ is identifiable as well.

Another way to interpret model \eqref{EqRelationP0P1}
is through the propensity score, the conditional probability of (not) being observed given the outcome.
An application of Bayes's theorem gives
\begin{align}
	\label{eq.selectionbias}
	\text{logit } \text{Pr}\left(R=0 \given Y\right) = \eta + q(Y),
\end{align}
where
\begin{align}
	\label{eq:prelation}
	e^\eta = \frac{1-p}{p}\frac{1}{\int e^q \,dP_1}.
\end{align}
Because of this relationship, the function $q$ will be referred to as the \emph{selection bias} or the \emph{sensitivity function}.
When $q = 0$ (or constant), the propensity score is independent of the outcome and therefore the model is MCAR.
In the other case the missingness depends on the outcomes, and $q$ models deviations from the MCAR assumption.
The function $q$ can be modelled parametrically, as in the following example, or non-parametrically.

\begin{example}[\cite{Scharfstein2003}]
	Let $\mY = \RR^+$ and let $q$ range over all functions of the form
	$q_\a(y)=\a \log y$, where $\a$ ranges over $\RR$.
	Think of a clinical trial, with $R=0$ indicating a dropout of the study. For $\a > 0$,
	the function   $y\mapsto P(R = 0 \given Y = y)$ increases
	monotonically from 0 to 1 as $y$ increases from 0 to infinity. Thus for $\a > 0$,
	subjects with higher outcomes are more likely to drop out of the study.
	For $\a < 0$, this is reversed. For $\a=0$, the dropout is MCAR.
\end{example}

We do not assume that the model \eqref{EqRelationP0P1} or \eqref{eq.selectionbias} is correct,
or that there is a `true' frequentist parameter $q_0$ that generates the observations.
The main purpose is to assess robustness of study conclusions when the function $q$ varies,
or is equipped with a prior.

We are interested in estimating, for a given measurable function $g : \mY \to \RR$,
\begin{align}
	\label{eq:functional}
	\E_P g(Y)
	= \int g\, dP = p\int g(y)\left(1 + \frac{1-p}{p} \frac{e^{q(y)}}{\int e^q \,dP_1}\right)\,dP_1(y).
\end{align}
The right side of the equation shows that this functional is estimable from the observed data, for any given function
$q$. We might for instance be interested in all functions $q$ such that this functional
is above or below a certain threshold, or
to give confidence or credible intervals on this functional for varying $q$. A prior on
$q$ in addition to a prior on the data model will lead to a posterior distribution on the
functional. It is reasonable to combine non-informative priors on the data distribution
with an informative prior on $q$.

The full model can be parameterised by the triplet $(p, P_1,q)$, but also by the triplet
$(\eta,P, q)$. The first triple determines the marginal distribution of $R$ through $p$, and next
the conditional distribution of $Y$ given $R=1$ through $P_1$, and given $R=0$ through $q$ in view of
equation \eqref{EqRelationP0P1}. The second triple gives the marginal distribution of $Y$ through $P$ and next
the conditional distribution of $R$ given $Y$ through $(\eta, q)$, in view of equation \eqref{eq.selectionbias}.
This suggests various possibilities for prior modelling and inference, which we discuss in the next sections.

\section{Parametrisation by \texorpdfstring{$p$}{p}, \texorpdfstring{$P_1$}{P1}}
\label{sec:paramH}
The distribution of the observation $X$ in \eqref{EqDefX} is given by
$$H=(1-p)\delta_\ast+pP_1,$$
for $\delta_\ast$ the Dirac measure at the missingness indicator $\ast$. Since there is
a bijective relation between $H$ and the pair $(p,P_1)$, the model can also be parameterised
by the pair $(H,q)$. The standard non-parametric
prior on the distribution $H$ is the Dirichlet process prior $\DP(a)$, which is given by
a base measure $a$, a finite Borel measure on $\mX$
(see \cite{Ferguson1973} or Chapter~4 in \cite{Ghosal2017} for a review).
It is well known that the Dirichlet process prior
is conjugate: if $H \sim \DP(a)$ and $X_1\ldots,X_n \given H \sim H$,
then a version of the posterior distribution is given by $H\given X_1,\ldots, X_n\sim \DP(a + n\HH_n)$,
where $\HH_n = \sumin \delta_{X_i} /n$ is the empirical measure of the observations.
Furthermore, the functional Bernstein-von Mises theorem for the Dirichlet process (\cite{Lo1983, Lo1986}
or \cite[p.~364]{Ghosal2017}) gives that
\begin{align}
	\label{eq:BvMH}
	\sqrt{n}\left(H - \HH_n\right) \given X_1,\ldots,X_n \weak\BB_{H_0},\qquad H_0^\infty-\text{a.s.}
\end{align}
Here $\BB_{H_0}$ is an $H_0$-Brownian bridge process, and the convergence can be understood
to mean that, for every finite set
of measurable functions $g_j: \mX\to\RR$ with $(H_0 + a)^\ast g_j^2 < \infty$,
and for almost every realisation of the i.i.d.\ sequence $X_1,X_2,\ldots$ with law $H_0$,
the posterior distribution of
$$\bigl(\sqrt{n}\left(H - \HH_n\right)g_1,\ldots, \sqrt{n}\left(H - \HH_n\right)g_k\bigr)
\given X_1,\ldots,X_n $$
converges to the distribution of the random vector $(\BB_{H_0}g_1,\ldots, \BB_{H_0}g_k)$:
a multivariate normal distribution with mean zero and
covariances $\cov(\BB_{H_0}g_i,\BB_{H_0}g_j)=H_0g_ig_j-H_0g_iH_0g_j$.
The convergence is also true in the space $\ell^\infty(\mathcal{G})$, for every
$H_0$-Donsker class $\mathcal{G}$ with measurable envelope function $G$
such that $(H_0 + a)G^2 < \infty$. (See e.g.\ \cite{Dudley} or \cite{Vandervaart1996} for the
definition of Donsker classes.)

The posterior distribution of $H$ induces a posterior distribution on the pair $(p,P_1)$, and hence for
a fixed function $q$, a posterior distribution on the parameter of interest \eqref{eq:functional}.
In fact, when the domain of $q$ is extended  to $\mathfrak{X}$ such that $\delta_* e^q = 0$, the parameter of interest can be expressed as a function of $H$ as
\begin{align}
	\label{eq:functionalH}
	\E_Pg(Y)  =  \frac{H\left(ge^q \right) H\{\ast\}}{H\left(e^q \right)} + Hg =: \chi(H,q).
\end{align}
We may use the (conditional) delta-method to deduce the limit behaviour of the posterior distribution
of this quantity.

\begin{theorem}[Bernstein-von Mises $H$]
	\label{TheoremBvMH}
	If $H\given X_1,\ldots, X_n\sim \DP(a+n\HH_n)$ and \\
	$\left(H_0 + a \right)(g^2e^{2q}+g^2+e^{2q}) < \infty$,
	then, conditional on $q$, $H_0^\infty-\text{a.s.}$,
	\begin{align*}
		&\sqrt{n}\bigl(\chi(H,q) - \chi(\HH_n,q) \bigr) \given X_1,\ldots, X_n, q\\
		\noalign{\vskip3pt}
		& \weak \frac{\BB_{H_0}(1_{\{\ast\}}) H_0(ge^q)}{H_0e^q}
		+\frac{H_0\{\ast\}\BB_{H_0}(ge^q)}{H_0e^q}-
		\frac{H_0\{\ast\}H_0(ge^q)\BB_{H_0}(e^q)}{(H_0e^q)^2}+\BB_{H_0}g.
	\end{align*}
\end{theorem}

\begin{proof}
	The functional can be written $\chi(H,q)=\phi\left( H\{\ast\}, H(ge^q), He^q, Hg \right)$, for
	the function $\phi : \RR^4 \to \RR$, defined by
	\begin{align*}
		\phi(x_1,x_2,x_3,x_4) = \frac{x_1 x_2}{x_3} + x_4.
	\end{align*}
	The theorem therefore follows from
	the delta-method for conditional distributions (\cite{Vandervaart1996}, Section~3.9.3,
	or more precisely \cite{Vandervaart2023}, Theorem~3.10.13)
	combined with the functional Bernstein-von Mises theorem \eqref{eq:BvMH}. The
	limit variable arises as the inner product
	$$\nabla\phi(\theta_0)\cdot\bigl(\BB_{H_0}(1_{\{\ast\}}), \BB_{H_0}(ge^q), \BB_{H_0}e^q, \BB_{H_0}g \bigr),$$
	where $\nabla\phi(\theta_0)$ is the gradient of the function $\phi$ evaluated
	at the vector $\theta_0 = (H_0\{\ast\}, \allowbreak H_0(ge^q), H_0e^q, H_0g )$.
\end{proof}

By the multivariate central limit theorem (or Donsker's theorem),
the empirical process $\sqrt n(\HH_n-H_0)$ converges
(unconditionally) to the same $H_0$-Brownian bridge process $\BB_{H_0}$.
Therefore, the (ordinary) delta-method gives that the sequence
$\sqrt{n}\bigl(\chi(\HH_n,q) - \chi(H_0,q) \bigr)$ tends
in distribution to the same limit variable as in the theorem. Thus the Bayesian and
non-Bayesian estimation procedures of the parameter merge asymptotically.
In the non-parametric situation that the distribution $H$ of the observations is
completely unknown, the empirical distribution $\HH_n$ is asymptotically efficient
for estimating $H_0$. Because efficiency is retained under the delta-method
(\cite{vdVHadamard}), the estimator $\chi(\HH_n,q)$ is asymptotically
efficient for estimating the parameter of interest $\chi(H,q)= Pg$. This implies that
the posterior distribution of $\chi(H,q)$ has minimal concentration as well.
We give an explicit derivation of the lower bound theory in the supplementary material (\cite{Supplement}).

However, these results refer to the situation that the function $q$ is known,
while our purpose here is sensitivity analysis. The unknown function $q$
can be equipped with a prior as well and be incorporated in a joint
Bayesian analysis. This prior would ordinarily be elicited from experts in the field of
research. In the present setup this will typically lead to mixing the limit distribution
over the prior, as we now argue.

If the sensitivity parameter is specified to be a priori independent from $H$,
and we maintain that $X\given (H,q)\sim H$, then it follows that
$(X_1,\ldots, X_n)\indep q\given H$, and hence $q\given X_1,\ldots, X_n,H\sim q\given H\sim q$.
Thus the data provide no information about $q$ and by conditioning on $q$,
the posterior distribution of the parameter can be seen to satisfy
\begin{equation}
	\label{EqLawFunctionalPriorH}
	\L\bigl(\chi(H,q)\given X_1,\ldots, X_n\bigr)=
	\int \Law\bigl(\chi(H,q)\given X_1,\ldots, X_n,q\bigr)\,d\Pi(q).
\end{equation}
Here the integrand is the distribution considered in the preceding theorem and $\Pi$ is
the prior of $q$. We conclude that putting a prior on the sensitivity function
leads simply to averaging the inferences for fixed sensitivity function with respect
to the prior. At first sight this seems natural, but it is special to the present
prior modelling, as we shall see in the next section.

In any case, given a prior on the sensitivity function, the distribution
of the functional will not be asymptotically normal, but a mixture of normal
distributions. This is a consequence of the fact that the sensitivity function
is not identifiable from the data, so that its posterior does not contract to
a Dirac measure.

We finish this section with two more technical observations on non-parametric
modelling with the Dirichlet process.

\begin{remark}
	\label{RemarkParametersBeta}
	Although the Dirichlet process is the canonical non-parametric
	prior for the unknown distribution $H$, in the present case it induces a
	possibly undesirable relation between the priors on the probability $1-p=H\{\ast\}$
	of a missing observation and the prior on $P_1=H_{|\mY}/p$. If $H\sim \DP(a)$,
	then $1-p\sim B\bigl(a\{\ast\},a(\mY)\bigr)$ and is independent of
	$P_1\sim \DP(a_{|\mY})$, where $B(\a,\beta)$ denotes the Beta distribution and is to be interpreted
	as a point mass at 0 if $0=\a<\beta$ (see \cite{Ghosal2017}, Theorem~4.5). Thus the
	specification $H\sim \DP(a)$ with a single base measure $a$ links the parameters of
	the Beta prior for $p$ through the `prior precision' $a(\mY\cup\{\ast\})$ of the base measure.
	An alternative would be to combine a general beta prior on $p$ with a general Dirichlet process prior
	on $P_1$. It can be verified that this does not change the asymptotic distribution
	of the posterior distribution, although it does change the finite sample behaviour.
\end{remark}

\begin{remark}
	\label{RemarkAtom}
	Choosing a base measure $a$ in $H\sim \DP(a)$ without an atom at
	the point $\ast$ leads to the Dirac distribution at 0 as a prior on $H\{\ast\}=1-p$. In the preceding
	theorem the negative effect of this extreme prior is not seen, because the theorem considers the
	standard $\DP(a+n\HH_n)$-posterior distribution.
	However, this is only a \emph{version} of the posterior distribution, while
	a posterior distribution is only unique up to null sets under the Bayesian marginal distribution
	of the data. If $X\given P\sim P$ and $P\sim DP(a)$, then $X\sim a$.
	Therefore, in the present case non-uniqueness of the posterior arises for a set of
	positive frequentist probability if $a\{\ast\}=0<H_0\{\ast\}$.  Since it is likely that $H_0\{\ast\}>0$, it is
	reasonable to put a prior $H\sim \DP(a)$ with $a\{\ast\}>0$.
\end{remark}

\section{Parametrisation by \texorpdfstring{$\eta$}{n}, \texorpdfstring{$P$}{P}}
\label{SectionEtaP}
The distribution $P$ of the full data variable $Y$ is the parameter of prime interest, and is
intuitively  more fundamental than the distribution $H$ of the observed data variable $X$.
Therefore, as is also argued in \cite{Scharfstein2003},
it is reasonable to put a Dirichlet process prior on $P$ rather than $H$. To obtain a full prior on the
observation model, we combine this with an independent prior on the parameter $(\eta,q)$. The resulting
posterior distribution is more complicated than the one in the preceding section, but can, for fixed $(\eta,q)$,
be studied using the theory of completely random measures.

Also in the parametrisation by the triplet $(\eta, P,q)$, the sensitivity parameter $q$ is unidentifiable
from the observed data. Nevertheless, due to the functional relations between the parameters,
the posterior distribution
of $q$ under the present prior modelling will be dependent on the data, even if $q$ is modelled
a priori independent of $(\eta,P)$. We discuss this intriguing situation further at the end of
the section, after first analysing the posterior distribution for fixed $q$.

In view of \eqref{eq:prelation} and \eqref{EqRelationPP1}, we have $dP=p(1+e^{\eta+q})\,dP_1$.
This can be inverted to give
\begin{align}
	\label{eq.PP1relation}
	P_1(A) = \frac{1}{p}\int_A  \frac{1}{1 + e^{\eta + q}}\,dP,\qquad A \in \Y.
\end{align}
Here the parameter $p=\Pr(R=1)$ acts as the normalising constant to the probability
distribution on the right and can be expressed in $(\eta,P,q)$ as
\begin{align}
	\label{eq.prelation}
	p  = \int \frac{1}{1+e^{\eta + q}}\,dP.
\end{align}
The distribution $P_1$ is the distribution of the observed outcome $X$. In the following lemma we show that
a Dirichlet process prior on $P$ leads, for fixed $(\eta, q)$, to a normalised
extended gamma process prior on $P_1$. We briefly review this terminology, referring
to \cite{Kingman(1967),Kingman(1975)} or Appendix~J of \cite{Ghosal2017} for extended discussion.

A completely random measure is a random element $U$ taking values in the set of Borel
measures on a Polish space $\mY$ such that $U(A_1),\ldots, U(A_k)$ are
independent random variables, for every finite measurable partition $A_1,\ldots, A_k$ of $\mY$.
Every such measure can be represented as the sum of a deterministic measure and a measure
of the form
\begin{equation}
	\label{EqRepNCRM}
	U(A)=\int_A\int_0^\infty s\,N^c(dy,ds)+\sum_{j=1}^\infty U_j\delta_{y_j}(A),
\end{equation}
where $N^c$ is a Poisson process on $\mY\times(0,\infty]$, the variables
$U_1,U_2,\ldots$ are arbitrary independent, nonnegative  variables independent of $N^c$, and
$y_1,y_2,\ldots$ is an arbitrary sequence of elements of $\mY$. The weights $U_i$ and locations $y_i$
together form the `fixed atoms' of $U$.
The distribution of $U$ is determined by the intensity measure of $N^c$ and the distributions
and locations of the fixed atoms, which we denote by
\begin{align*}
	\nu^c(A\times B)&=\E N^c(A\times B),\\
	\nu^d(\{y_j\}\times B)&=\Pr(U_j\in B).
\end{align*}
We can think of $\nu^d$ as another measure on $\mY\times(0,\infty]$ concentrated
on the union $\cup_j\{y_j\}\times [0,\infty]$. The pair $(\nu^c, \nu^d)$, or equivalently the sum $\nu^c+\nu^d$,
is known as the intensity measure of $U$.

If $0<U(\mY)<\infty$ almost surely, then
$U$ can be normalised to the random probability measure $U/U(\mY)$, which can serve
as a prior for a probability measure. The posterior distribution of this measure
given an i.i.d.\ sample $X_1,\ldots, X_n$ from the prior $U/U(\mY)$, can be characterised
as a mixture of normalised, completely random measures (see
\cite{James2005,JamesLijoiandPrunster(2009)} or Theorem~14.56 in \cite{Ghosal2017}).

The Dirichlet process is a normalised completely random measure derived from the
Gamma process $U$. If the process $\bigl(U(A): A \in \Y\bigr)$ consists of
random variables such that
$U(A_1),\ldots, U(A_k)$ are independent variables with $U(A_i)\sim \text{Ga}\left(a(A_i),1\right)$,
for every $i$ and every finite collection of disjoint sets $A_1,\ldots,A_k \in \Y$,
then $P= U/U(\mY)$ follows the Dirichlet process $\DP(a)$.
In this case, in view of \eqref{eq.PP1relation},
\begin{align}
	\label{EqRepresentationP1U}
	P_1(A) = \frac{\int_A \frac{1}{1+e^{\eta + q}}\,dU}{\int \frac{1}{1+e^{\eta + q}}\,dU},
	\qquad A \in \Y.
\end{align}
This exhibits $P_1$ also as a normalised completely random measure. Its
intensity measure can be obtained from the intensity measure of the gamma process.
This process is a generalisation of the extended gamma process on $\RR^+$.

\begin{lemma}
	If $P\sim \DP(a)$ for an atomless finite Borel measure $a$ on $\mY$,
	then the measure $P_1$ given in \eqref{eq.PP1relation} is a normalised
	completely random measure with intensity measure given by $\nu^d=0$ and
	\begin{equation}
		\label{EqIntensityMeasure}
		\nu^c(dy,ds) = s^{-1}e^{-s\left(1+e^{\eta + q(y)}\right)}\,ds\,da(y).
	\end{equation}
\end{lemma}

\begin{proof}
	The gamma process $U$ has no fixed atoms and continuous
	intensity measure $\nu^c(dy,ds) = s^{-1}e^{-s}\,ds\,da(y)$. Consider the completely random measure
	defined by $P_1'(A) = \int_A 1/\left(1 + e^{\eta + q}\right)\,dU$.
	For any non-negative measurable function $f : \mY \to\RR$,
	we have that $\int f\, dP_1' = \int f/\left(1+e^{\eta + q}\right)\, dU$,
	as $P_1' \ll U$ with Radon-Nikodym derivative $1/\left(1 + e^{\eta + q}\right)$.
	As in (J.6) on p. 599 of \cite{Ghosal2017}, the minus log Laplace transform of $P_1'$ is then given by
	\begin{align*}
		-\log \E\left(e^{-\theta \int f dP_1'}\right)
		&= -\log \E\left(e^{-\theta \int \frac{f}{1 + e^{\eta + q}} dU}\right) \\
		&= \int\!\!\int_0^\infty \left(1-e^{-\theta s \frac{ f(y)}{1+e^{\eta + q(y)}}}\right) s^{-1}e^{-s}\,ds\,da(y).
	\end{align*}
	To identify the intensity measure $\nu'$ of $P_1'$, the right side must be written in the form
	$\int\!\int_0^\infty \left(1 - e^{-\theta s f(y)}\right)\, \nu'(dy,ds)$.
	By the substitution $s' = s/(1+e^{\eta + q(y)})$, we obtain
	\begin{align*}
		\int\!\!\int_0^\infty &\left(1-e^{-\theta s \frac{ f(y)}{1+e^{\eta + q(y)}}}\right) s^{-1}e^{-s}\,ds\,da(y) \\
		&= \int\!\!\int_0^\infty \left(1-e^{-\theta s' f(y)}\right) (s')^{-1}e^{-s'\left(1 + e^{\eta + q(y)}\right)}\,ds'\,da(y).
	\end{align*}
	As the Laplace functional determines the completely random measure, the proof is complete.
\end{proof}

\begin{remark}
	Unlike in the preceding section, the base measure $a$ of the $\DP(a)$-prior is now a measure on $\Y$, and not
	on $\X=\Y\cup\{\ast\}$. To apply the theory on normalised completely random measures,
	it will be convenient to take it without atoms. This is natural in this situation, unlike
	in the preceding section, where an atom at the special point $\{\ast\}$ is natural, as noted
	in Remark~\ref{RemarkAtom}.
\end{remark}

The non-missing observations  $Y_i$ (the values $X_i$ with $R_i=1$, or equivalently $X_i\not=\ast$) form a
random sample from the distribution $P_1$, of random size
\begin{align*}
	N_n = \sumin R_i.
\end{align*}
The variable $N_n$ follows a Binomial$(n,p)$ distribution, and will tend to infinity with $n$.
In the Bayesian setup with a Dirichlet $\DP(a)$ prior on $P$,
the data may be thought of as having been generated by the Bayesian hierarchy (with every step conditioned
on the outcomes of all preceding steps):
\begin{itemize}
	\item[] {\bf prior:}
	\item generate $(\eta,q)$ from a prior.
	\item generate $P_1$ as a normalised completely random measure
	with intensity measure \eqref{EqIntensityMeasure}.
	\item compute $p=1/\int(1+ e^{\eta+q})\,dP_1$.
	\item[] {\bf data:}
	\item generate $N_n\sim \text{binom}(n,p)$.
	\item generate $\bar X_1,\ldots,\bar X_{N_n}$ i.i.d.\ from $P_1$.
	\item \scalebox{0.94}{randomly permute $\bar X_1,\ldots,\bar X_{N_n}$ and $n-N_n$ copies of $\ast$ to define $X_1,\ldots, X_n$.}
\end{itemize}
If we also compute $P$ by  inverting \eqref{eq.PP1relation}, then the variables $(\eta, q, p, P, X_1,\ldots, X_n)$ will have the same
joint distribution as in the scheme with prior $(\eta,q)\indep P\sim \DP(a)$.
Therefore, the resulting posterior distributions of $P$ or $P_1$ given $(\eta,q,p)$ and $(X_1,\ldots, X_n)$
are the same.

In the Bayesian hierarchy the conditional distribution of $P_1$ given $N_n=k, \bar X_1,\ldots, \allowbreak \bar X_k, \eta, q, p$ is the same as
the posterior distribution of $P_1$ based on a sample of size $k$ from the normalised
completely random measure with intensity measure \eqref{EqIntensityMeasure}.
The asymptotics of this posterior distribution are obtained in Theorem~\ref{thm.BvMegp}
and imply the following theorem.

The theorem gives the limit of the posterior distribution under the assumption
that the observations are an i.i.d.\ sample from the distribution of $(R,X)$ given by
$\Pr(R=1)=p_0$ and $X\given R=1\sim P_{1,0}$, for given $(p_0,P_{1,0})$, or equivalently
$X_1,\ldots, X_n$ are i.i.d.\ copies of $X\sim H_0$, where
$H_0=(1-p_0)\delta_\ast+p_0P_{1,0}$, as in Theorem~\ref{TheoremBvMH}.
Denote the empirical distribution of the observed outcomes by $\PP_{n,1}=N_n^{-1}\sumin \delta_{X_i}1_{R_i=1}.$

\begin{theorem}[Functional Bernstein-von Mises $P_1$]
	\label{thm:BvMP1eta}
	The posterior distribution of $P_1$ under the prior $P\sim \DP(a)$ independent of
	$(\eta,q)$  satisfies, $H_0^\infty$-almost surely,
	$$\sqrt n (P_1-\PP_{n,1})\given X_1,\ldots, X_n,\eta,q,p \weak \sqrt{\frac 1{p_0}}\, \BB_{P_{1,0}},$$
	uniformly in $\eta$ and $q$ that are uniformly bounded above.
\end{theorem}

\begin{proof}
	Given $X_1,\ldots, X_n$, define $N_n$ as the number of $X_i$ unequal to $\ast$ and let
	$\bar X_1,\ldots, \allowbreak \bar X_{N_n}$ be the $X_i$ that are not equal to $\ast$. The variables
	$N_n, \bar X_1,\ldots, \bar X_{N_n}$ represent the same information as $X_1,\ldots, X_n$,
	apart from ordering. Thus for any measurable function $h$,
	\begin{align*}
		&\E h\bigl(\sqrt n (P_1-\PP_{n,1})\given X_1,\ldots, X_n,\eta,q,p \bigr)\\
		&\qquad\qquad\qquad=\E h\bigl(\sqrt n (P_1-\PP_{n,1})\given N_n=k, \bar X_1,\ldots, \bar X_k,\eta,q,p \bigr).
	\end{align*}
	The Bayesian hierarchy shows that given $\eta,q,p$,  the variable $N_n$ follows a binomial distribution with
	probability $p$, while given $N_k=k, \eta,q,p$ the variables
	$\bar X_1,\ldots, \bar X_k$ are an i.i.d.\ sample from the distribution $P_1$, which follows
	an extended gamma normalised completely random measure without fixed atoms and
	continuous intensity given by \eqref{EqIntensityMeasure}.
	Thus the distribution of $P_1$ given $N_n=k, \bar X_1,\ldots, \bar X_k,\eta,q,p$ in the
	present notation is the same as the distribution in the notation
	of Section~\ref{SectionExtendedGammaPosterior} of $P_k$ given $X_1,\ldots, X_k$,
	with $b$ in the latter section taken equal to $b=1+e^{\eta+q}$ and $P_0$ in
	the latter section taken equal to the present $P_{1,0}$.
	
	By the law of large numbers $N_n/n\ra p_0$, almost surely. Hence the theorem
	follows from Theorem~\ref{thm.BvMegp}.
\end{proof}

For posterior inference on the functional of interest, we need
the posterior distribution of $p$ next to the one of $P_1$. The preceding theorem shows that
$P_1$ is asymptotically independent of $p$ under the posterior distribution.
Because the parameters $p$ and $P_1$ are related under the prior, this is
not true for finite $n$, and it appears that the posterior distribution of
$p$ may depend on the values $\bar X_1,\ldots,\bar X_{N_n}$ generated later
in the Bayesian hierarchy, besides on the variable $N_n$. The asymptotic
independence suggests that little is lost by ignoring $\bar X_1,\ldots,\bar X_{N_n}$
in posterior inference on $p$ and base this on the observational model
$N_n\sim \text{binom}(n,p)$ only. This so-called \emph{cut-posterior} \cite{Jacob2017, Moss2022} (which cuts
out the later data) will satisfy the usual Bernstein-von Mises theorem for the
binomial distribution.

\begin{lemma}
	\label{LemmaBvMp}
	Suppose that the prior of $\eta$ possesses a bounded, continuous density on a compact
	interval in $\RR$ and that $|q|$ is uniformly bounded.
	Then the posterior distribution of $p$ based on the observation $N_n\sim \text{binom}(n,p)$
	and prior on $p$ induced by $P\sim \DP(a)\indep \eta$, almost surely
	under $p_0$,
	$$\sqrt n \Bigl(p-\frac{N_n}n\Bigr)\given N_n,q \weak N\bigl(0, p_0(1-p_0)\bigr).$$
\end{lemma}

\begin{proof}
	The lemma follows from the classical Bernstein-von Mises theorem (e.g.\ \cite{Vandervaart1998}, Chapter~10),
	provided that the prior density of $p$ is continuous.
	
	By \eqref{eq.prelation} $p=\int \Psi(\eta+q)\,dP=:\phi_P(\eta)$,
	where $\Psi(v)=1/(1+e^v)$ is the logistic function, and $P\sim\DP(a)$ is independent of $\eta$.
	Because the function $\eta\mapsto \phi_P(\eta)$ is strictly increasing, it follows that the cumulative
	distribution function of $p$ can be written
	$$\Pr\bigl(\phi_P(\eta)\le t\bigr)=\E_P \int_0^{\phi_P^{-1}(t)}h(\eta)\,d\eta,$$
	for $h$ the density of $\eta$.
	The function $\Psi$ is differentiable with derivative $\phi_P'(\eta)=\int \psi(\eta+q)\,dP$,
	where $\psi(v)=\Psi(v)\bigl(1-\Psi(v)\bigr)$ is
	bounded away from zero if $v$ is restricted to a compact set in $\RR$. It follows that
	the cumulative distribution function in the preceding display is continuously differentiable
	with derivative $t\mapsto\E_P h\circ \phi_P^{-1}(t)/\phi_P'\circ \phi_P^{-1}(t)$.
\end{proof}

\begin{remark}
	The Bayesian hierarchy shows that $\bar X_1,\ldots, \bar X_{N_n}\indep p\given N_n,P_1,q$, whence
	$p\given (\bar X_1,\ldots, \bar X_{N_n},N_n,P_1,q)\sim p\given N_n,P_1,q$. Thus the posterior
	distribution of $p$ given the full data and $P_1,q$ is obtainable from the observational model
	$N_n\given p,P_1,q\sim\text{binom}(n,p)$ with prior $p\given P_1,q$. If the latter conditional prior
	is sufficiently smooth, then the posterior distribution of $p$ given the full data and $P_1,q$ will
	satisfy the ordinary binomial Bernstein-von Mises theorem, where the Gaussian approximation
	will be independent of the prior and hence coincide with the cut-posterior in the preceding lemma.
	We conjecture that this is true, but verifying the required smoothness seems not trivial due to the
	infinite-dimensional nature of $P_1$.
\end{remark}

The parameter of interest can be expressed in the parameters $(P_1,\eta,p,q)$ as
$$\E_Pg(Y) = (1-p)\frac{P_1(ge^q)}{P_1e^q}+pP_1 g=:\kappa(p,P_1,q).$$
Therefore, combining the preceding theorem and lemma gives the following Bernstein-von Mises
theorem for the parameter of interest.

\begin{theorem}
	[Bernstein-von Mises $\eta, P$]
	\label{TheoremBvMFunctionalPEta}
	If $P_1$ is as in Theorem~\ref{thm:BvMP1eta} and $p$ is as in Lemma~\ref{LemmaBvMp},
	then, conditional on $q$, $H_0^\infty$-a.s.,
	$$\sqrt n \bigl(\kappa(p,P_1,q)-\kappa(N_n/n,\PP_{1,n},q)\bigr)\given X_1,\ldots, X_n,q\weak Z_{H_0},$$
	where the limit variable $Z_{H_0}$ has the same distribution as the limit in Theorem~\ref{TheoremBvMH}.
\end{theorem}

\begin{proof}
	The functional can be written $\kappa(p,P_1,q)=\phi\bigl(p,P_1(ge^q),P_1e^q,P_1g\bigr)$, for the function
	$\phi: \RR^4\to\RR$, defined by $\phi(x_1,x_2,x_3,x_4)=(1-x_1)x_2/x_3 +x_1x_4$.
	The theorem therefore can be derived with the help of
	the delta-method for conditional distributions (see \cite{Vandervaart1996}, Section~3.9.3,
	or more precisely \cite{Vandervaart2023}, Theorem~3.10.13)
	combined with the results of Theorem~\ref{thm:BvMP1eta} and Lemma~\ref{LemmaBvMp}.
	The limit variable $Z_{H_0}$ arises as the inner product
	$$Z_{H_0}=\nabla\phi(\theta_0)\cdot\Bigl(Z_0,\sqrt{\frac1{p_0}}\BB_{P_{1,0}}(ge^q),
	\sqrt{\frac1{p_0}}\BB_{P_{1,0}}e^q, \sqrt{\frac1{p_0}}\BB_{P_{1,0}}g \Bigr),$$
	for $Z_0\sim N\bigl(0,p_0(1-p_0)\bigr)$ independent of the Brownian bridge $\BB_{P_{1,0}}$
	and $\nabla\phi(\theta_0)$ the gradient of $\phi$ evaluated
	at the vector
	$\theta_0 = \bigl(p_0, P_{1,0}(ge^q), P_{1,0}e^q, P_{1,0}g \bigr)$. We can check by explicit calculation
	of the variance that $Z_{H_0}$ possesses the same centred normal distribution
	as the limit variable in Theorem~\ref{TheoremBvMH}.
	(Alternatively, we can apply Corollary~\ref{CorollaryBvMHPrioretaP} below to see that
	the posterior distributions of $\bigl(H\{\ast\}, H_{|\Y}/H(\mY)\bigr)$ and $(1-p,P_1)$
	in the two theorems are the same, together with the
	identity $(1-N_n/n)\delta_\ast+(N_n/n)\PP_{1,n}=\HH_n$.)
\end{proof}

\begin{corollary}
	\label{CorollaryBvMHPrioretaP}
	If $P_1$ is as in Theorem~\ref{thm:BvMP1eta} and $p$ is as in Lemma~\ref{LemmaBvMp}
	and $H=(1-p)\delta_{\ast}+pP_1$,
	then the sequence $\sqrt n(H-\HH_n\bigr)$
	has the same asymptotic conditional distribution given $X_1,\ldots, X_n$
	as the sequence $\sqrt n(H-\HH_n)$ in Theorem~\ref{TheoremBvMH}.
\end{corollary}

\begin{proof}
	Since $\HH_n=(1-N_n/n)\delta_\ast+(N_n/n)\PP_{n,1}$, it follows that
	$H=\phi(p,P_1)$ and $\HH_n=\phi(N_n/n,\PP_{1,n})$, for the map
	$\phi: [0,1]\times\ell^\infty(\F)\to \ell^\infty(\F)$ given by $\phi(s,Q)=(1-s)\delta_\ast+sQ$.
	By  Lemma~\ref{LemmaBvMp} and Theorem~\ref{thm:BvMP1eta}, the sequence
	$\sqrt n(p-N_n/n, P_1-\PP_{1,n})$ converges in distribution given $X_1,\ldots, X_n$
	in the space $[0,1]\times \ell^\infty(\F)$ to the pair $(Z_0,\BB_{P_{1,0}})$ of
	a normal variable $Z_0\sim N\bigl(0,p_0(1-p_0)\bigr)$ and an independent $P_{1,0}$-Brownian bridge
	process, almost surely.
	Thus the delta-method for conditional distributions \cite[p.~511]{Vandervaart1996}
	gives that, almost surely,
	$$\sqrt n (H-\HH_n)\given X_1,\ldots, X_n\weak \phi_{p_0,P_{1,0}}'\Bigl(Z_0,\sqrt{\frac1{p_0}}\BB_{P_{1,0}}\Bigr),$$
	where the derivative is given by $\phi_{p_0,P_{1,0}}'(g,G)=-g\delta_{\ast}+gP_{1,0}+p_0G$.
	The limit process $f\mapsto Z_0\bigl(-f(\ast)+P_{1,0}f\bigr)+\sqrt{p_0}\BB_{P_{1,0}}f$
	is centred Gaussian and linear in $f$ with variance
	function $p_0(1-p_0) \bigl(P_{1,0}f-f(\ast)\bigr)^2+p_0\bigl(P_{1,0}f^2-(P_{1,0}f)^2\bigr)$.
	It can be verified that this is equal to $\var \BB_{H_0}f$ for
	a $\BB_{H_0}$-Brownian bridge process, which is the limit process in Theorem~\ref{TheoremBvMH}.
\end{proof}

Theorem~\ref{TheoremBvMFunctionalPEta} gives the posterior distribution of the parameter
of interest given the sensitivity function $q$. Suppose that we also equip the
sensitivity function with a prior. In practical situations this might
be an informative prior based on substantive knowledge. If the parameter  $q$ is independent of $(\eta,P)$ under the prior, then $q$ will \emph{not} be independent of
the distribution  $H=(1-p)\delta_\ast+p P_1$ of the observations. We can write
\begin{align*}
	&\L\bigl(\kappa(p,P_1,q)\given X_1,\ldots,X_n\bigr)
	=\int \L\bigl(\kappa(p,P_1,q)\given X_1,\ldots,X_n,q\bigr)\,d\Pi(q\given X_1,\ldots, X_n)\\
	&\qquad\qquad=\iint \L\bigl(\kappa(p,P_1,q)\given X_1,\ldots,X_n,q\bigr)\,d\Pi(q\given H)\,d\Pi(H\given X_1,\ldots, X_n).
\end{align*}
In the second step we use that $q\given H, X_1,\ldots, X_n\sim q\given H$, since
$X_1,\ldots, X_n\indep q\given H$, as follows from the fact that $X_1,\ldots, X_n\given H,q,\eta\iid H$.
By Corollary~\ref{CorollaryBvMHPrioretaP} the posterior distribution $H\given X_1\ldots, X_n$ satisfies
the same Bernstein-von Mises theorem as in Theorem~\ref{TheoremBvMH}.
In particular, the posterior distribution of $H$ will concentrate near $\HH_n$ (equivalently
near $H_0$). Thus informally the preceding display approximates to
\begin{equation}
	\label{EqLawFunctionalPrioretaP}
	\int \L\bigl(\kappa(p,P_1,q)\given X_1,\ldots,X_n,q\bigr)\,d\Pi(q\given \HH_n).
\end{equation}
Thus the observations influence the posterior distribution of $q$ through the prior
dependence between $q$ and $H$. In the next section we compare this to the result of
the preceding section.

\section{Comparison of the two parametrisations}
\label{SectionComparison}
While the posterior distributions of the parameter of interest given $q$ is the same
for the two prior setups in Sections~\ref{sec:paramH} and~\ref{SectionEtaP}, the
distributions differ when $q$ is also (independently) equipped with a prior.

The asymptotic posterior distributions of the parameter for the two priors are
given in \eqref{EqLawFunctionalPriorH}  and \eqref{EqLawFunctionalPrioretaP}.
Both expressions can be informally written in the form
$$\chi(H,q)\given X_1,\ldots, X_n\sim W_n+\frac 1{\sqrt n} Z_n,$$
where the variables on the right satisfy informally, asymptotically, for
a latent `posterior variable' $q$,
$$W_n\indep Z_n\given q,\quad W_n=\chi(\HH_n,q)\quad
Z_n\given  q\sim N(0,\sigma_{H_0}^2),$$
for $\sigma_{H_0}^2$ the variance of the limit variable in Theorem~\ref{TheoremBvMH}.
The two prior settings differ in the distributions of the latent variable $q$:
\begin{itemize}
	\item When using the prior on $H$ as in Section~\ref{sec:paramH}, the sensitivity function $q$ is
	distributed according to its prior.
	\item When using the prior on $(\eta,P)$ as in Section~\ref{SectionEtaP}, the sensitivity function
	is distributed according to $q\given H=\HH_n$, where $(q,H)\sim (q, (1-p)\delta_\ast+pP_1)$.
\end{itemize}
Thus when using the first  prior, the inference on the functional is simply mixed over the
various values of the sensitivity function, weighted by its prior. When using the second
prior, the data will influence our ideas about the sensitivity function.

Note that this
is true even though in both cases the sensitivity function is not identifiable from the
data. The non-identifiability makes that for both priors the sensitivity function remains
a random object, even if the number of observations tends to infinity: its posterior distribution
does not contract to a degenerate distribution. However, the two posterior distributions
differ. In the first case the posterior is equal to the prior, while in the second it is
adapted to the data through the estimation of $H$.

One can debate which of the two situations is more desirable. From a nonparametric Bayesian
view, it seems natural to place the Dirichlet process prior, which in many ways is \emph{the}
nonparametric prior, comparable to the empirical distribution, on the most fundamental
distribution in the problem. This seems to favour the second prior, the one which makes
the posterior distribution learn something about the sensitivity parameter from the data.


From an applied point of view, the choice of prior may be decided based on subject matter experts' beliefs about $q$ in relation to $(p, P_1)$. Only if experts' ideas about the magnitude of the selection bias are unwavering,
no matter the observed values of $X$, does the first prior seem to be preferable. The second prior is the better choice if experts' beliefs about the extent of selection bias will change if observed outcomes turn out to be higher or lower than expected, or if fewer or more subjects drop out than expected. \cite{Scharfstein2003} provide an example of the latter situation in the context of HIV research. The desired dependence between $q$ and $(p, P_1)$ can then be expressed in the second prior. Since $\eta$ is a function of $q, p$ and $P_1$, and $p$ and $P_1$ can be identified from the observed data, a prior belief on $\eta$ induces a posterior belief on $q$ which will depend on the data.
To set the priors on $\eta$ and $q$ in practice, relationship \eqref{eq.selectionbias} forms a useful basis for discussion. For several values of $\eta$ and functions $q$, the prior expectations of the conditional probability of drop-out can be shown for a realistic range of values for $Y$. In this process the priors of $\eta$ and $q$ are intimately linked, as will be fleshed out further in the simulations.

\section{Simulations}
\label{SectionSimulation}
In this section we present the results of a simulation study, illustrating
the theoretical results of the preceding sections, in particular
when a prior is placed on the sensitivity parameter. We focus on the setting where the prior on $\a$ is misspecified, as this is where we expect to see differences between the two parametrisations. The code can be found on \href{https://gitlab.tudelft.nl/beggen/bayesian-sensitivity-analysis}{Gitlab}.

We consider a model with sample space $\mY=(0,\infty)$ and sensitivity function $q(y)=\a \left(\log y - c\right)$,
for $\a, c \in\RR$. When $c=0$, the function is equal to the choice in \cite{Scharfstein2003}. More information on $c$ can be found below in Section \ref{sec:simulsetting}. The functional of interest is the mean outcome $\E_P Y$.
We take $\ast = 0$, and therefore the observations can be written as $X = R Y\in [0,\infty)$.

\subsection{Sampling scheme for the (\texorpdfstring{$P_1$}{P1}, \texorpdfstring{$p$}{p}, \texorpdfstring{$\a$}{a}) parametrisation}
\label{sec:SimulationsH}
A $\DP(a)$ prior on $H=(1-p)\delta_0+p P_1$ is combined with an independent
prior on $\a$. As the observed data $X_1,\ldots X_n$ is a random sample from $H$,
the posterior distribution of $H$ again is a Dirichlet process, but with base measure $a + n\mathbb{H}_n$,
where $\HH_n=\sumin \delta_{X_i}=(1-N_n/n)\delta_0+ \sumin R_i\delta_{X_i}$.
Furthermore, as discussed in Section~\ref{sec:paramH}, see \eqref{EqLawFunctionalPriorH}, the posterior distribution
of $\a$ is equal to its prior, and the two parameters $H$ and $\a$ remain independent.

The functional of interest is written as a function of $H$ and $\a$ by equation \eqref{eq:functionalH},
with $q$ the function $q(y)=\a \left(\log y - c\right)$. The posterior distribution of the functional is obtained
by inserting the $\DP(a + n\mathbb{H}_n)$-distribution for $H$ and the prior distribution for $\a$.
We approximated the Dirichlet process through its stick-breaking representation
(\cite{Sethuraman1994}, or \cite{Ghosal2017}, Theorem~4.12),
where we truncated the infinite sum such that the retained weights summed numerically to one.
We sampled $10000$ times from the posterior of ($H$, $\a$)
and approximated each expectation in $\chi(H,\a)$
by the average of the samples.


\subsection{Sampling scheme for the (\texorpdfstring{$P$}{P}, \texorpdfstring{$\eta$}{n}, \texorpdfstring{$\a$}{a}) parametrisation}
\label{sec:SimulationsetaP}
A $\DP(a)$ prior on $P$ is combined with independent priors on $\eta$ and $\a$. In this case the
posterior distribution given the observed data does not possess a simple form, but the posterior distribution given
the full data does. Following a similar approach to \cite{Scharfstein2003},
we implement a Gibbs sampling scheme, which fills in the missing data.

The full data can be generated through the Bayesian hierarchy, for
$\Psi$ the logistic function (see \eqref{eq.selectionbias}):
\begin{itemize}
	\item $P\sim\DP(a)\indep (\eta,\a)$.
	\item $Y_1\ldots,Y_n\given P,\eta,\a\iid P$.
	\item $R_1,\ldots, R_n\given Y_1,\ldots, Y_n,P,\a,\eta\ind \text{Bernoulli}\bigl(\Psi(-\eta-\a(\log Y_i - c))\bigr)$.
\end{itemize}
This hierarchy shows that
$P\indep (\eta,\a,R_1,\ldots, R_n)\given(Y_1,\ldots, Y_n)$.
This implies that the posterior distribution of $P$ given the full data
$(Y_1,\ldots, Y_n,R_1,\ldots,R_n)$ is the same
as its posterior distribution given $Y_1,\ldots, Y_n$, which is $\DP(a+\sumin\delta_{Y_i})$.
Moreover, this implies $P\indep (\eta,\a)\given(Y_1,\ldots, Y_n,R_1,\ldots,R_n)$, so that the posterior distributions
of $P$ and $(\eta,\a)$ given the full data are independent. Finally, the
second  posterior distribution $(\eta,\a)\given(Y_1,\ldots, Y_n,R_1,\ldots,R_n)$ is recognized
as the posterior distribution of the parameters $(\eta,\a)$
in the logistic linear regression model regressing $R_1,\ldots, R_n$ on the
independent variables $\log Y_i$ with intercept $-\eta$, slope $-\a$.

The observed data consists of $R_1,\ldots, R_n$ and the $N_n=\sumin R_i$ values $Y_i$ with $R_i=1$. Given the
parameters $(P,\eta,\a)$, the $n-N_n$ missing values $(Y_i: R_i=0)$ are distributed according to the
measure $P_0$ given by (combine equations \eqref{EqRelationP0P1} and \eqref{eq.PP1relation}):
\begin{equation}
	\label{eq:relationP0P}
	P_0(A) = \frac{\int_A \frac{e^{ q}}{1 + e^{\eta+q}}\,dP}{\int \frac{e^{ q}}{1 + e^{\eta+q}}\,dP}, \qquad A \in \mathcal{Y}.
\end{equation}
Furthermore, given $(P,\eta,\a)$,
the missing values are conditionally independent of each other and of the observed data.

These observations lead to the following Gibbs sampling scheme, which augments the observed data
to full data in the first step and next updates the parameters $P$ and $(\eta,\a)$ according to the
full data posterior in the second and third steps. In each step we take the observed data as given,
as well as the outcomes of the other two steps.
\begin{enumerate}
	\item Sample  $n-N_n$ values from $P_0$, assign them to the (missing) values $(Y_i: R_i=0)$.
	\item Sample $P\sim \DP(a+\sumin \delta_{Y_i})$.
	\item Sample $(\eta,\a) $ from its posterior distribution given $R_1,\ldots, R_n,Y_1,\ldots, Y_n$
	in logistic regression model $R_1,\ldots, R_n$ on $-\eta-\a \left(\log Y_1 - c \right),\ldots,-\eta-\a \left(\log Y_n - c\right)$.
\end{enumerate}
The three steps are repeated until convergence, giving a sample $P$ in each
iteration. The corresponding mean values $\int y\,dP(y)$, after burn-in, form an approximate
sample from the posterior distribution of the functional of interest. The third step of the algorithm was carried out by the Metropolis-Hastings algorithm  \citep{Metropolis1953,Hastings1970}.
We used a random walk proposal distribution with steps generated from the normal distribution
with mean zero and a given variance. The latter variance was set such that the probability of accepting a new draw
was approximately 0.234 \citep{Gelman1997}. The acceptance probabilities are
computed using the likelihood of the logistic regression model,  given by
\begin{align*}
	&\prodin \left(\frac{1}{1+e^{\eta + \a\left(\log Y_i - c \right)}} \right)^{R_i} \left(\frac{e^{\eta + \a\left(\log Y_i - c\right)}}{1+e^{\eta + \a\left(\log Y_i - c\right)}} \right)^{1-R_i}.
\end{align*}

\subsection{Simulation settings}
\label{sec:simulsetting}
First, we simulate data under the assumption that the true data
distribution belongs to the sensitivity model. Specifically, we fix a true data distribution
through choosing $P^{0}_1=\text{Gamma}(r,s)$,  for some $r,s>0$,  a fraction $1-p_0$ of missing observations
and a sensitivity parameter $\a_0$, where the superscript $0$ indicates the truth. We reserve the subscript for denoting the conditional distributions. With the choice of $q$, the sensitivity model \eqref{EqRelationP0P1}
imposes $dP^0_0(y)\propto e^{\a_0 \left(\log y - c\right)}\,dP^0_1(y)$, resulting in $P^0_0=\text{Gamma}(r+\a_0,s)$. The full data
distribution $P^0=(1-p_0)P^0_0+p_0P^0_1$ is a mixture of the two Gamma distributions, and the functional of interest
is given by $\E_{P^0}Y = p_0 r/s + (1-p_0) (r+\a_0)/s$. Relation \eqref{eq:prelation} yields $\eta_0 = \log\left( (1-p_0)/p_0\right) + \a_0\left(\log s +c\right)- \log\left( \Gamma(r+\a_0) /\Gamma(r)\right)$. 

We want to observe the performance of the model in a natural, but somewhat challenging setting by letting a substantial part of the data be missing and choosing $\a_0$ such that there is a deviation from MCAR, while keeping the groups comparable. 
Specifically, we chose $r=2$, $s=1$, $p_0=0.6$, $\a_0=2$ and $c \approx 0.42$, giving $\eta_0\approx -1.35$ and $\E_{P^0}Y=2.8$. A default choice is $c=0$, but a different interesting choice would be $c = \mathbb{E}_{P^0}[\log Y] = p_0\left(\psi(r)-ln(s)\right) + (1-p_0)\left(\psi(r+\alpha_0) - \ln(s)\right)$, where $\psi$ is the digamma function. Using this choice of $c$ results in a centred $q$, ideally resulting in independent behaviour of $\alpha$ and $\eta$. In practice, data would only be available from $P^0_1$. Therefore another choice for $c$ could be $\mathbb{E}_{P_1^0}\left[\log Y\right]$, which was the choice we made. Note that $P^0_1$, $P^0_0$ and the functional of interest do not depend on the choice of $c$. Choosing an oracle $c$ had similar results. 

The posterior distributions are agnostic of the above choices, and depend only on the priors chosen for the parameters. In the first prior scheme, we let $H \sim \DP(a)$, with $a(\{0\}) = 1-p_0$
and $a_{|\mathbb{R}^+} = \text{Gamma}(2,1)$. For the other parametrisation, we let $P \sim \DP(\tilde{a})$,
with $\tilde{a}$ a mixture between the true $P^0_1$ with probability $p_0$ and the true $P^0_0$ with probability $1-p_0$. In both parametrisations the prior precision was taken to be equal to one, as to minimise the effect of the base measure.
We put a normal prior on $\a$ with mean 1 and standard deviation 0.5 and in the second parametrisation a uniform prior on $\eta$ with minimum -5 and maximum 2. By taking this interval relatively big we hope to investigate the effect of the data on the posterior of $\a$. Our main interest lies in observing this effect in the $P, \eta$ parametrisation, discussed in Section~\ref{SectionComparison}. We found that this effect is visible when the prior on $\a$ is misspecified: the mean of $\a$ is shifted from $\a_0$. Although it should be expected that practitioners have reasonable insight in this choice, usually the $\a$ can't be correctly specified as it is not assumed that the sensitivity model is true. 

\subsection{Results}
\label{sec:simulresult}
The results can be seen in Figure~\ref{fig:functional}. We also calculated the coverage of the $90\%$-credible intervals. These intervals are defined by the 0.05 and 0.95 quantiles of the posteriors. The results for fixed $\a$ can be found in Table~\ref{tab:fixedalpha} and with prior on $\a$ in Table~\ref{tab:alphaprior}. Because of large run times, especially for $n=10000$, the coverage results were obtained using the Delftblue supercomputer \citep{DHPC2022}.
\begin{table}[!ht]
	\begin{tabular}{r|c|c|c}
		& $n = 100$ & $n = 1000$ & $n = 10000$ \\
		& Coverage | Length & Coverage | Length & Coverage | Length \\
		\hline  $H$ \hspace*{8pt}&  \hspace*{9pt}0.784 |  0.877 & \hspace*{9pt}0.857 | 0.361 & \hspace*{9pt}0.892 | 0.122 \\
		$P, \eta$ \hspace*{8pt}& \hspace*{9pt}0.771 | 0.870 & \hspace*{9pt}0.855 | 0.360 & \hspace*{9pt}0.892 | 0.122
	\end{tabular}
	\caption{Coverage of the $90 \%$-credible intervals of the posterior distribution of the functional of interest for the different parametrisations with fixed $\alpha = \alpha_0$.}
	\label{tab:fixedalpha}
\end{table}
\begin{table}[!ht]
	\centering
	\begin{tabular}{r|c|c|c}
		& $n = 100$ & $n = 1000$ & $n = 10000$ \\
		& Coverage | Length & Coverage | Length & Coverage | Length \\
		\hline  $H, \a$ &  \hspace*{9pt}0.560 |  0.946 & \hspace*{9pt}0.328 | 0.700 & \hspace*{9pt}0.026 | 0.662 \\
		$P, \eta, \a$ & \hspace*{9pt}0.531 | 0.932 & \hspace*{9pt}0.304 | 0.687 & \hspace*{9pt}0.131 | 0.635 
	\end{tabular}
	\caption{Coverage of the $90 \%$-credible intervals of the posterior distribution of the functional of interest for the different parametrisations with \emph{misspecified prior} $\alpha \sim N(1,0.25)$.}
	\label{tab:alphaprior}
\end{table}
\begin{figure}[!ht]
	\centering
	\begin{tabular}{|m{0.3\textwidth}|m{0.3\textwidth}|m{0.3\textwidth}|}
		\hline
		\begin{mycenter}[2pt] $\mathbf{n = 100}$ \end{mycenter} & \begin{mycenter}[2pt]
			$\mathbf{n=1000}$	\end{mycenter} &  \begin{mycenter}[2pt]
			$\mathbf{n=10000}$	\end{mycenter}  \\[-15pt] \hline \vspace*{-0.3in}\hspace*{-0.26in} \vspace*{-0.3in} \includegraphics[scale = 0.457]{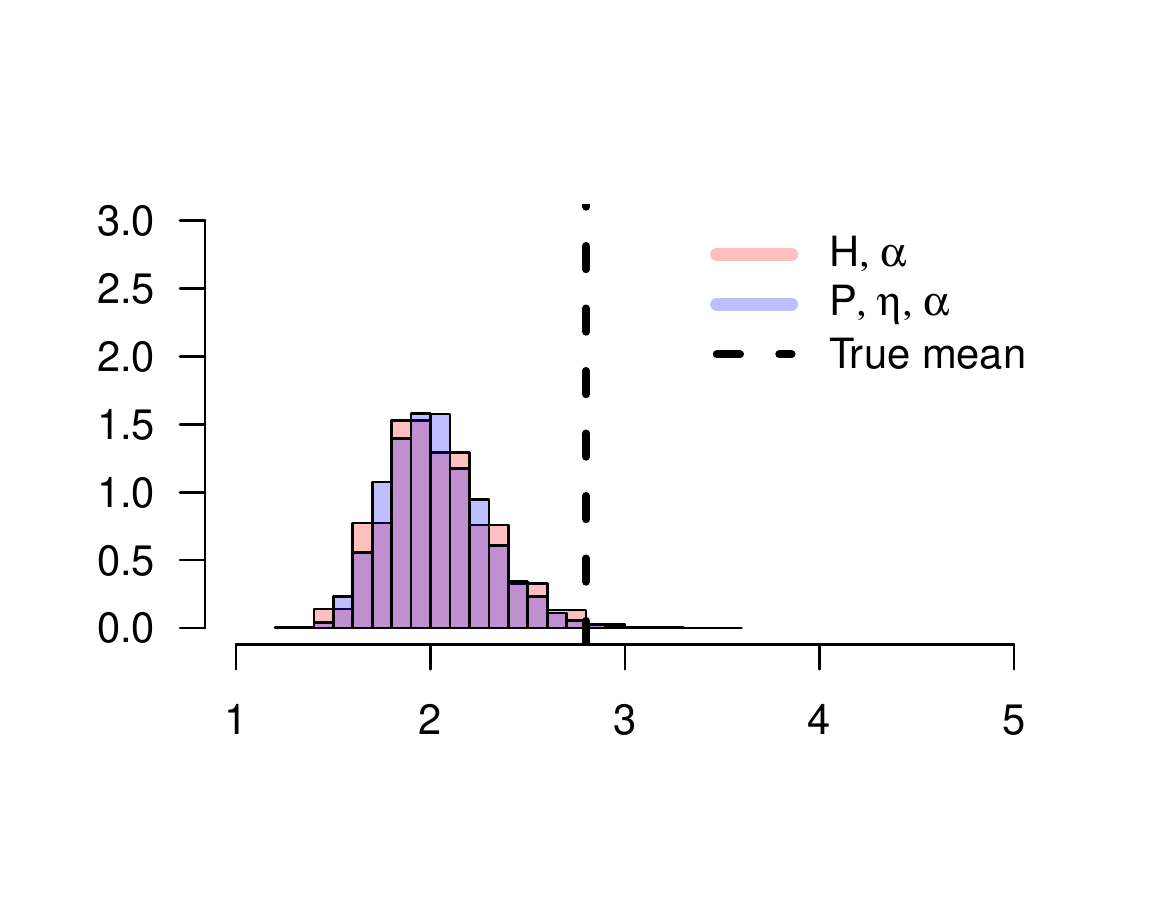}&\vspace*{-0.3in}\hspace*{-0.26in} \vspace*{-0.3in} \includegraphics[scale = 0.457]{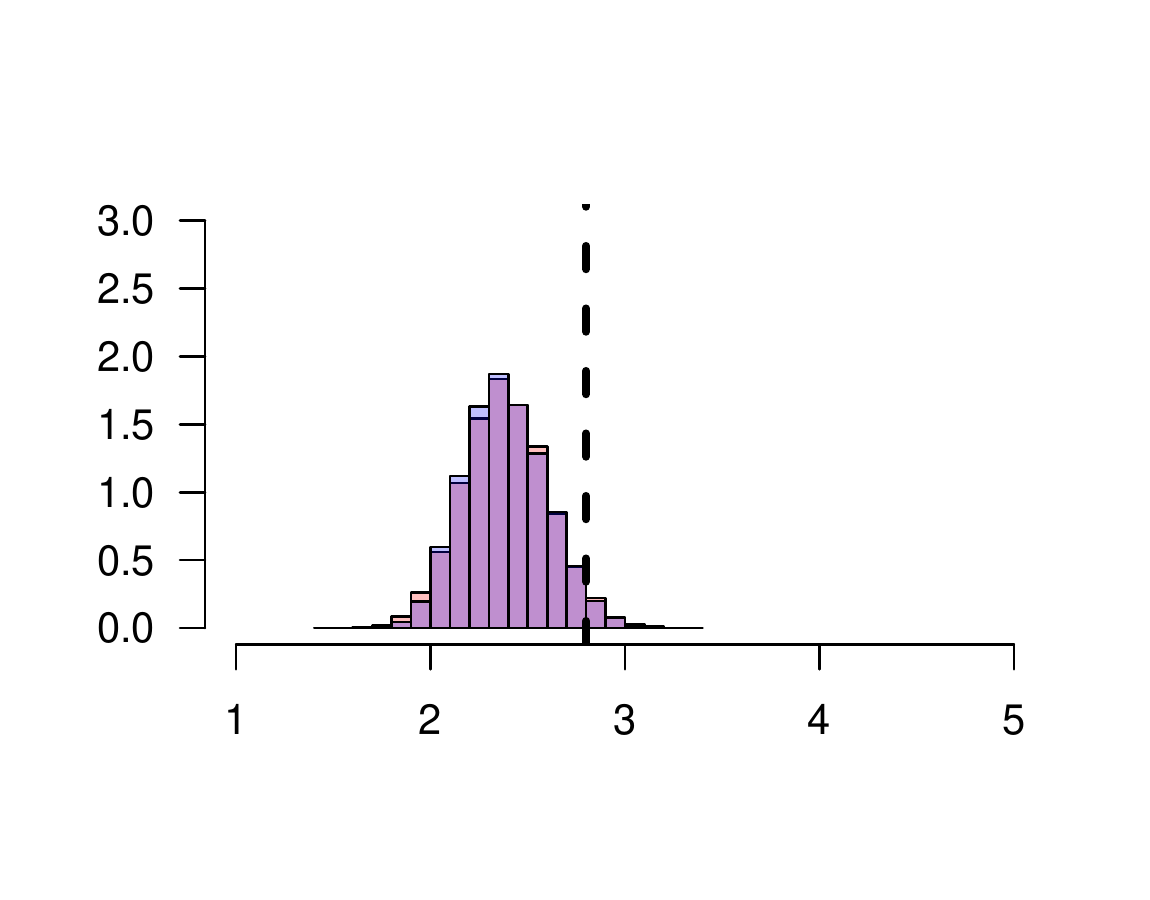}&\vspace*{-0.3in}\hspace*{-0.26in}\vspace*{-0.3in} \includegraphics[scale = 0.457]{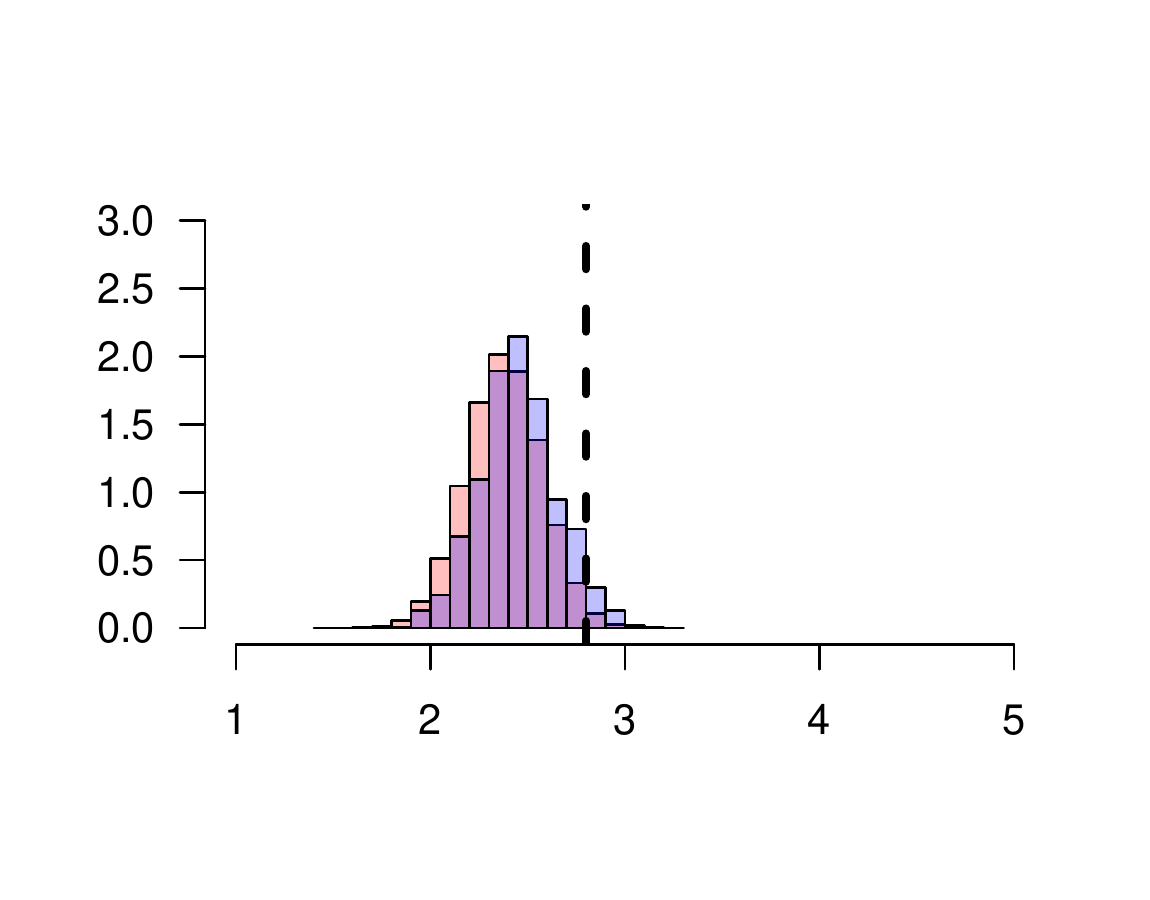}\\
		\hline \vspace*{-0.3in}\hspace*{-0.26in} \vspace*{-0.3in} \includegraphics[scale = 0.457]{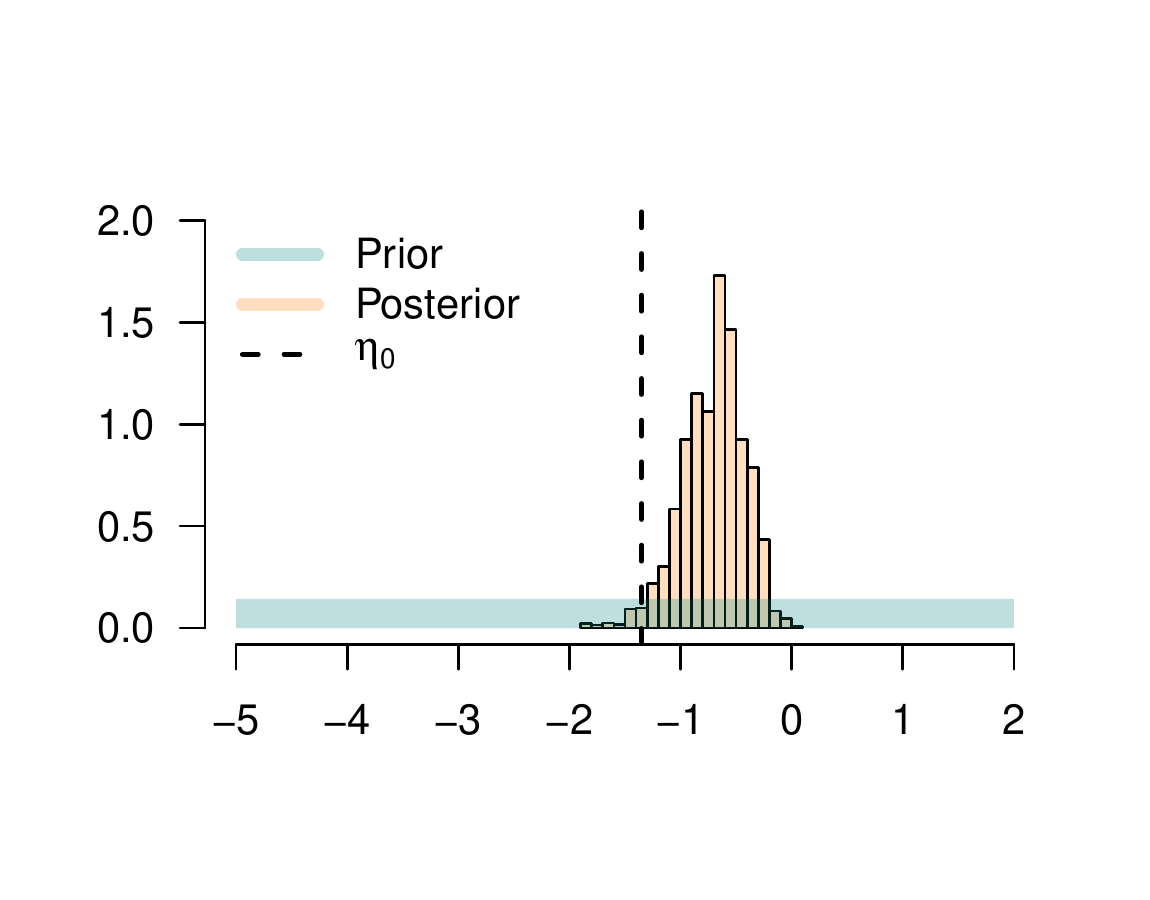}&\vspace*{-0.3in}\hspace*{-0.26in} \vspace*{-0.3in} \includegraphics[scale = 0.457]{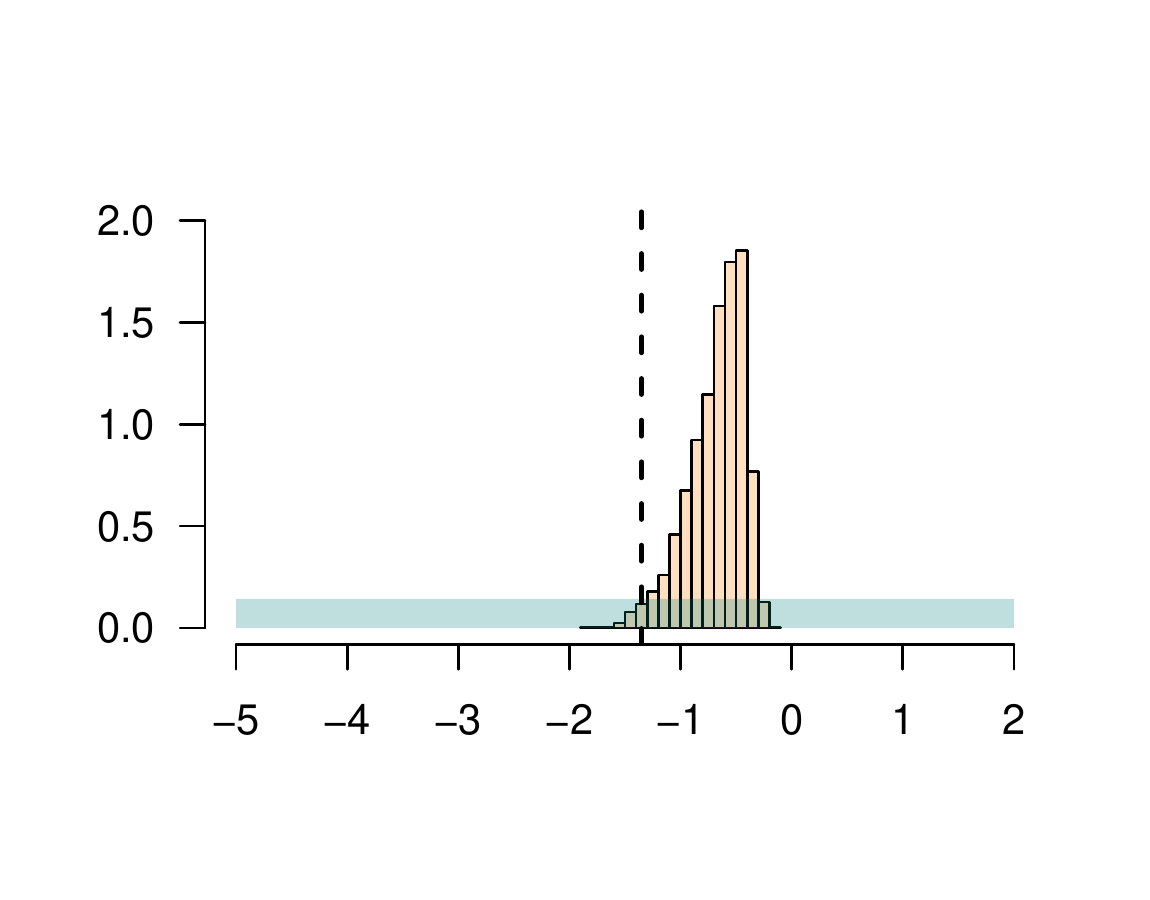}&\vspace*{-0.3in}\hspace*{-0.26in}\vspace*{-0.3in} \includegraphics[scale = 0.457]{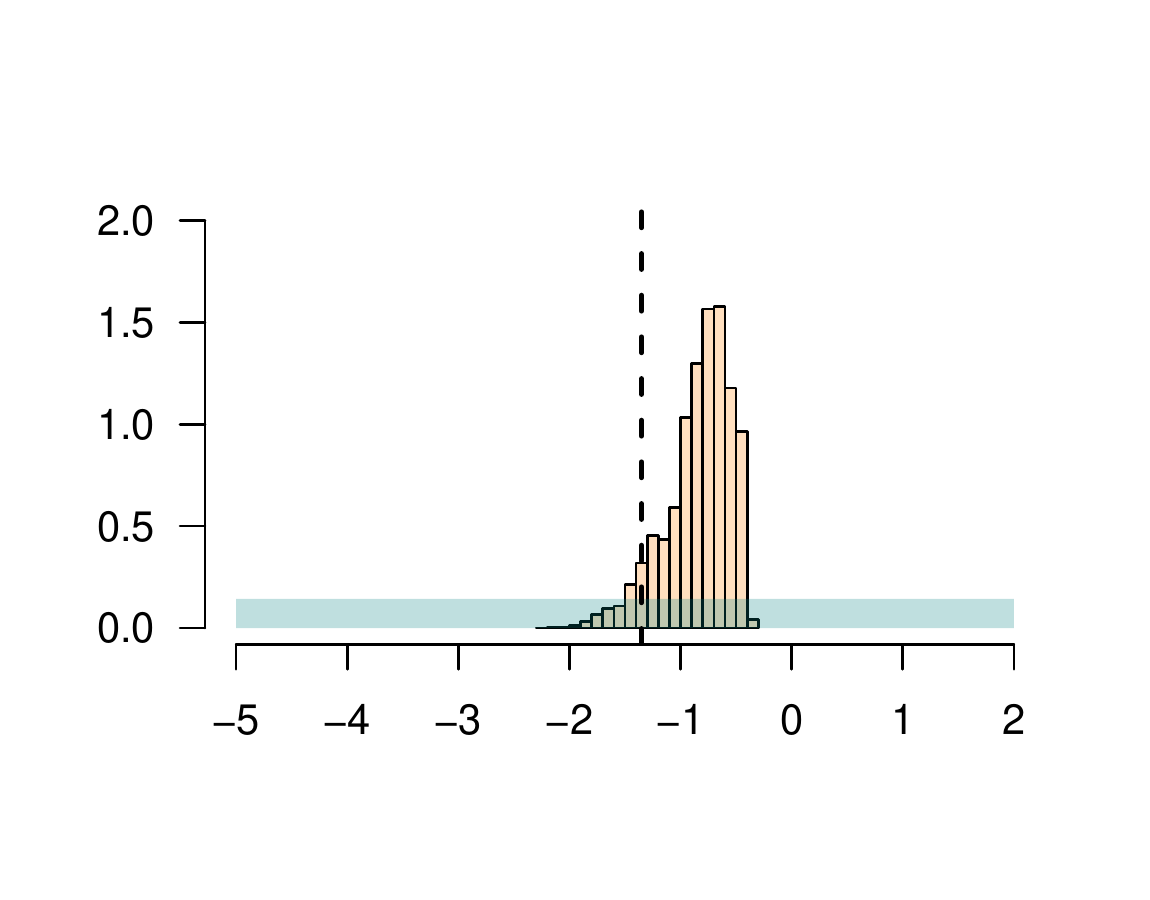}\\
		\hline \vspace*{-0.3in}\hspace*{-0.26in} \vspace*{-0.3in} \includegraphics[scale = 0.457]{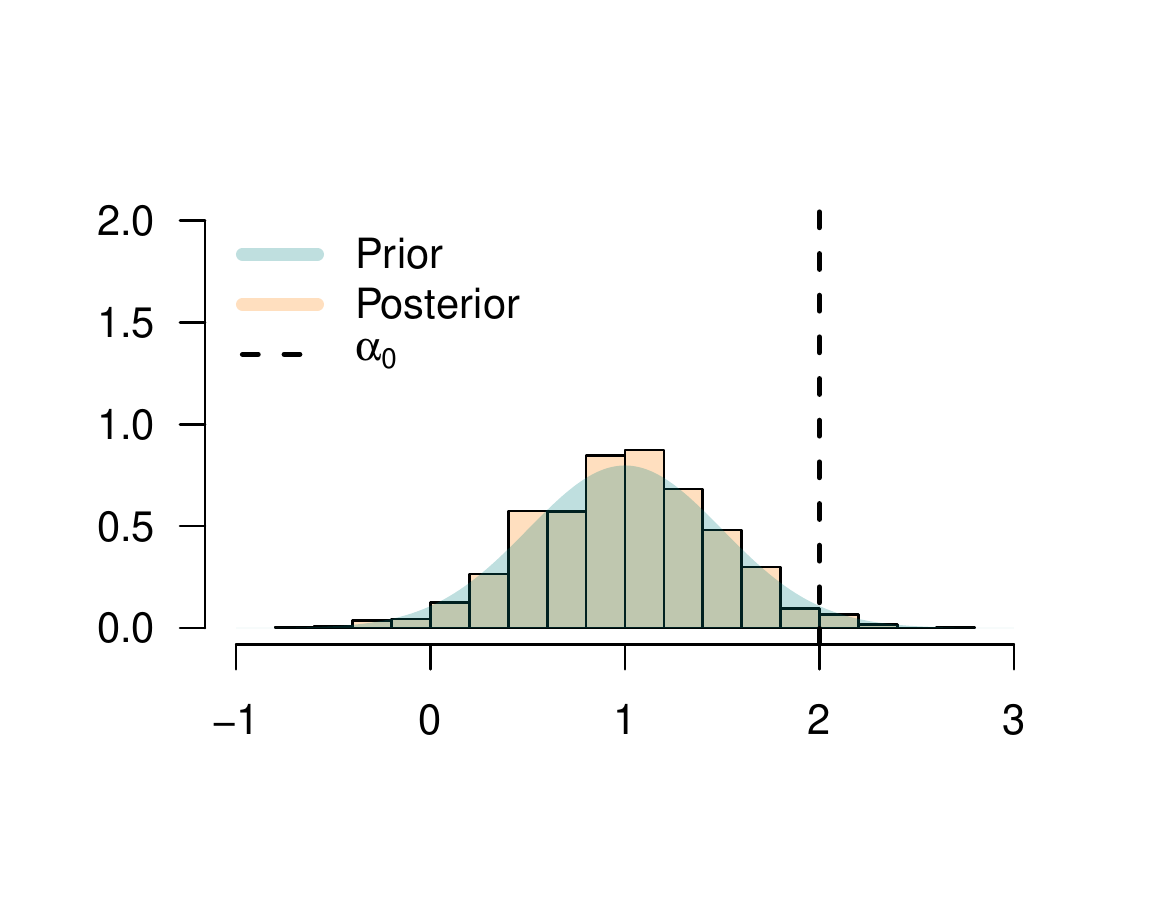}&\vspace*{-0.3in}\hspace*{-0.26in} \vspace*{-0.3in} \includegraphics[scale = 0.457]{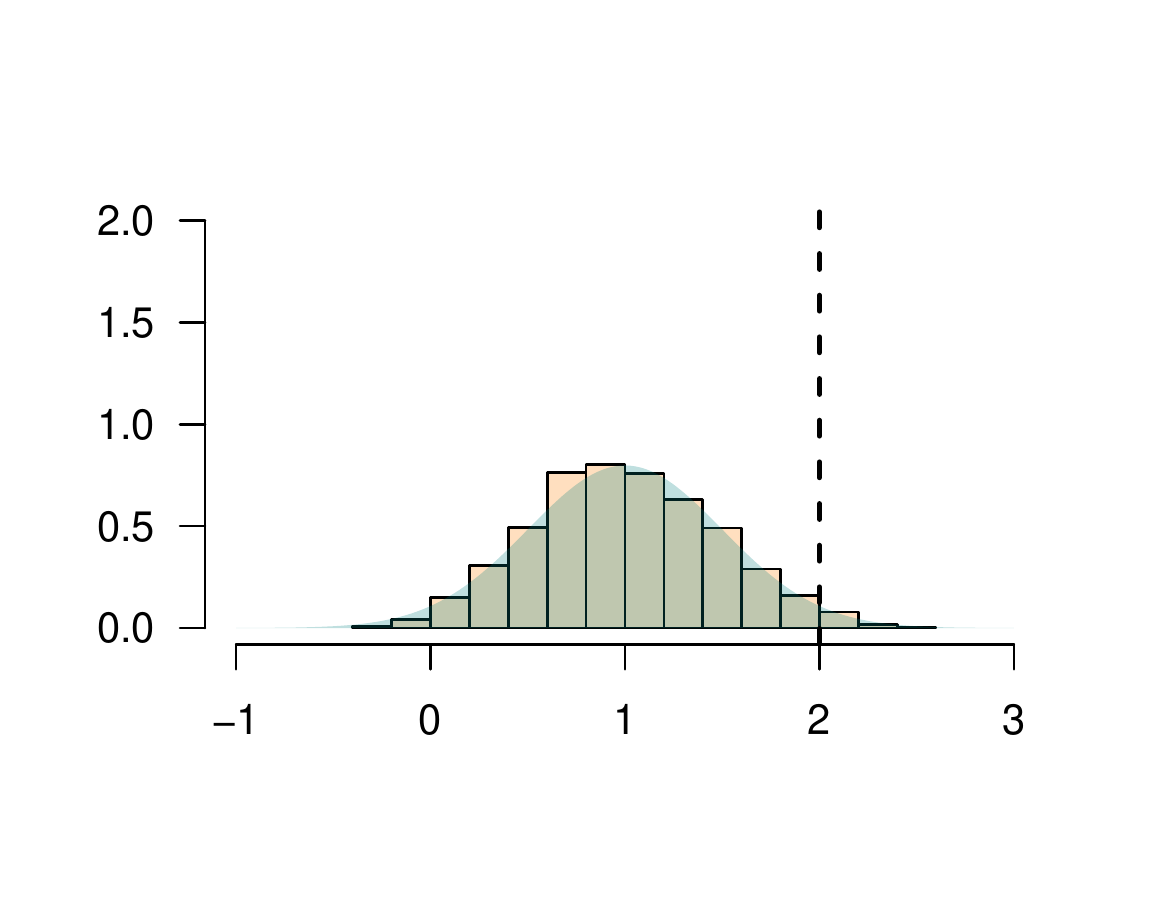}&\vspace*{-0.3in}\hspace*{-0.26in}\vspace*{-0.3in} \includegraphics[scale = 0.457]{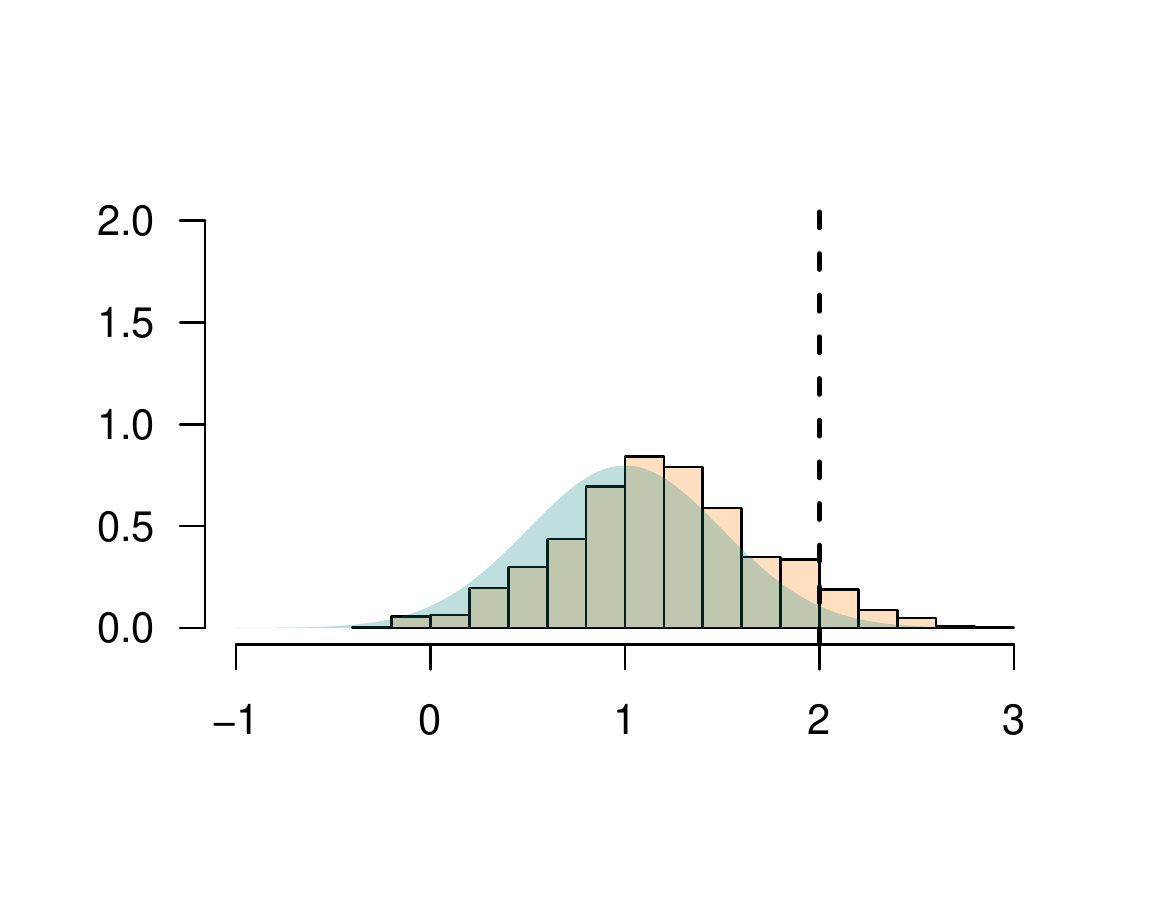}\\\hline
	\end{tabular}
	\caption{Posterior distributions with a prior on $\a$. \emph{Row 1}: Histogram of draws of the posterior distribution of the functional of interest for the different parametrisations. \emph{Row 2}: The prior and posterior of $\eta$. \emph{Row 3}: The prior and posterior of $\alpha$.}
	\label{fig:functional}
\end{figure}

From Figure~\ref{fig:functional} it can be seen that both methods perform similarly. When $n$ increases the variance of the posteriors decrease, but they will never converge to a Dirac measure due to the uncertainty in $\a$. This effect can be seen in Table~\ref{tab:alphaprior}, when compared to Table~\ref{tab:fixedalpha}. The latter illustrates the identifiability of the model when $\a = \a_0$ is fixed. The posteriors converge to a Dirac and have proper uncertainty quantification, as the theory in this paper confirms. From Figure~\ref{fig:functional} and Table~\ref{tab:alphaprior} it can be seen that the posteriors of the parametrisations asymptotically behave differently. The length of the credible intervals for the $P, \eta$ parametrisation are slightly smaller than those for $H$, but the coverage is higher. This means the location of this posterior is slightly different than the posterior of the functional in the other parametrisation, which is illustrated in in row 1 Figure~\ref{fig:functional}. In row 3, the posterior of $\a$ slightly shifts compared to the prior on $\a$ when $n$ increases, indicating the influence of the added data. From the point of view of sensitivity analysis, the $P, \eta$ parametrisation might therefore be preferable: the credible intervals more often contains the truth when the number of observations is large. If experts don't agree on the location of the prior on $\a$, but do on the direction of the relationship between $\a, \eta$ and $P$, the data might point the model in the right direction. When $\a$ is correctly specified, both parametrisations perform well and give similar results, as can be seen in the supplementary material \citep{Supplement}.

\section{Discussion}
\label{SectionDiscussion}

The results in this paper justify the use of the studied Bayesian sensitivity analyses procedures in practice, with the caveat that the utility of the answers will depend on the ability of subject matter experts to express their ideas about the outcomes for those lost to follow-up through a prior on $q$. The need to incorporate expert knowledge originates from the unidentifiability of the distribution $P_0$ of those who dropped out and is thus not special to the Bayesian paradigm. In frequentist sensitivity analysis approaches, expert knowledge is incorporated in a variety of ways, usually through bounds on or specific values for parameters that depend on unknown, $P_0$-derived quantities \citep{Liu2013}. Bayesian approaches offer perhaps an advantage by allowing the user to specify their ideas about the sensitivity parameter through a probability distribution, leading to a single summary of the sensitivity analysis.

In practice, more information may be available than assumed in this paper,
in the form of covariates $Z$ measured for each individual. In that case, the method can be
extended to sensitivity analysis on the Missing At Random (MAR) assumption,
by introducing dependence of $p, P_1$ and $q$ or $\eta, P$ and $q$ on $Z$. The approach
taken here then becomes a Bayesian implementation of the modelling strategy
proposed in \cite{RobinsRotnitzkyScharfstein}. They note that failure
of the MAR assumption $Y\indep R \given Z$ is equivalent to the conditional
distribution of $R\given Z,Y$ not being free of $Y$. They then propose a sensitivity analysis
by fitting a model of the type $\Pr(R=1\given Z,Y)=\Psi\bigl(\eta(Z)+q(Y,Z)\bigr)$,
for a given sensitivity function $q$, which depends on both $Y$ and $Z$. This is equivalent
to \eqref{eq.selectionbias} within levels of  the covariate $Z$.
In a Bayesian approach we need priors for  this extended logistic regression model,
the conditional distribution of the outcome $Y$ given $Z$, and the marginal distribution of $Z$.
For a binary outcome, nonparametric priors for
both conditional models can take the form of logistic Gaussian process models, as considered in \cite{RayvdV}, who
combine this with a Dirichlet process prior on the marginal distribution of the covariates, and
prove a Bernstein-von Mises theorem under the MAR assumption.
For survival outcomes, covariates could be added through a Cox proportional hazards model
\citep{Cox(1972)}, for which several Bernstein-von Mises theorems exist
\citep{Kim(2006),KimandLee(2003b)}. Another option of adding covariates is to use a dependent Dirichlet process prior
\citep{Quintana2020,MacEachern(1999)}. It will be of interest to extend these results
to the sensitivity setup, or from a different angle, extend the results of the present paper to
include covariates.

In this paper we considered estimating a mean in the missing data problem, but this is easily
extended to the estimation of an \emph{average causal effect}.
In the usual causal model (see \cite{RobinsHernan2020}),
there are two potential outcomes $Y^1$ and $Y^0$, corresponding to an individual
being treated $(R=1$) or not ($R=0$), and it is desired to estimate $\E Y^1-\E Y^0$,
based on observing only $Y^1$ if $R=1$ and only $Y^0$ if $R=0$. In other words,
one of the two potential outcomes is missing. The usual assumption, called \emph{conditional exchangeability} (CE)
or \emph{no unmeasured confounders}, is that $Y^r\indep R\given Z$, for $r=0,1$,
which is exactly the MAR assumption considered in the preceding paragraph.
A sensitivity analysis may investigate deviations from (CE) in exactly the manner proposed.

Within the causal setting, a different modelling perspective was considered in
\cite{McCandlessetal2007,McCandlessetal2012}, also see
\cite{Gustafson2010,GustafsonMcCandlessLawrence2018,McCandlessGustafson2017}, and
\cite{Dorieetal2016}. These authors interpret the failure of (CE) as the omission of
a covariate  in the conditioning. Thus they hypothesize the existence of a unmeasured confounder $U$
such that $Y^r\indep R\given Z,U$. Starting from a model for the distribution of
$(Y^r,R,Z,U)$, estimation of the parameter of interest is then achieved by
considering this as a missing data problem with missing data $U$. Typically $U$
is assumed to be a binary variable and its effect on the other variables is modelled
parametrically, although \cite{Dorieetal2016} allow a nonparametric relation between the outcomes
and covariates $Z$.


\section{Extended gamma posterior}
\label{SectionExtendedGammaPosterior}
In this section we study the posterior distribution for sampling from a prior equal to
the extended gamma process. We are interested
in the special case \eqref{eq.PP1relation}, but in this section adopt a general framework
independent of the main paper. We prove a functional Bernstein-von Mises theorem
and consistency of the posterior mean in the total variation norm.

For given measurable functions $b: \mX\to (0,\infty)$
and a finite, atomless Borel measure $a$ on the Polish space $\mX$, let $P_n$
be distributed as the normalised completely random measure without fixed atoms and
continuous part with intensity measure $\nu^c(dx,ds) = s^{-1}e^{-sb(x)}\,ds\,da(x)$.
Suppose that $P_n$ is used as a prior for the distribution of an i.i.d.\ sample of observations; hence
$X_1,\ldots,X_n\given P_n\iid P_n$. In this section we
derive the asymptotics of the conditional distribution of $P_n$ given $X_1,\ldots,X_n$
(i.e.\ the posterior distribution). In the intended application the function $b$ is equal to
$b=(1+e^{\eta+q})$.

We identify probability distributions $P$ on $\mX$ with the stochastic
processes $(Pf: f\in\F)$, for a given collection of measurable maps $f: \mX\to \RR$,
and consider weak convergence of sequences of such processes in $\RR^\F$ in the case of a finite collection, and
more generally in the space $\ell^\infty(\F)$ of bounded functions $z: \F\to\RR$ equipped
with the uniform norm $\|z\|_\F=\sup_{f\in \F}|z(f)|$. Weak convergence of conditional
distributions is understood in terms of the bounded Lipschitz metric: we say that $Z_n\given X_n\weak Z$
in probability or almost surely if
$\sup_{h\in \BL_1} \bigl| \E \bigl( h(Z_n)\given X_n\bigr)-\E h(Z)\bigr|\ra 0,$
in probability or almost surely. Here $\BL_1$ is the set of functions $h: \ell^\infty(\F)\to [-1,1]$ such
that $|h(z_1)-h(z_2)|\le \|z_1-z_2\|_\F$, for every $z_1, z_2\in \ell^\infty(\F)$. Let $\BB_{P_0}$ be a $P_0$-Brownian bridge process: a tight, Borel measurable, zero-mean Gaussian map into
$\ell^\infty(\F)$ with covariance function $\cov(\BB_{P_0}f,\BB_{P_0}g)=P_0(fg)-P_0fP_0g$.

\begin{theorem}[Bernstein-von Mises for extended gamma processes]
	\label{thm.BvMegp}
	Let $\F$ be a $P_0$-Donsker class with measurable envelope function $F$ such that $\int F\, da < \infty$,
	and let $b: \mX\to [\ub,\infty)$ be a measurable function, for some constant $\ub>0$.
	If there exist $q,r > 0$ with $ 1/q +1/r < 1/2$ such that $P_0b^q+P_0F^r< \infty$, then
	for $P_0^\infty$-almost every sequence $X_1,X_2,\ldots$,
	\begin{align*}
		\sqrt{n}\left(P_n-\PP_n\right) \given X_1,\ldots, X_n \weak \BB_{P_0},\qquad \text{in } \ell^\infty(\F).
	\end{align*}
	The convergence is uniform in measurable functions $b: \mX\to [\ub,\infty)$
	such that $\max_{1\le i\le n} \allowbreak b(X_i)=o(n^{1/q})$ and $\PP_nb-P_0b\ra 0$,
	almost surely, uniformly in $b$. If the latter uniform convergence is true in probability, then the
	statement is true in probability, uniformly in $b$.
	
\end{theorem}

\begin{proof}
	Denote  the distinct values in $X^{(n)}:=(X_1,\ldots, X_n)$ by $\tilde{X}_1,\ldots, \tilde{X}_{K_n}$
	and let $N_{1,n},\ldots,N_{K_n,n}$ be their multiplicities.
	It was shown in \cite{JamesLijoiandPrunster(2009)} (alternatively, see Theorem 14.56 of \cite{Ghosal2017}),
	that the posterior distribution of $P_n$ is a mixture over $\l$ of the distributions
	of the normalised completely random measures $\Phi_\l/\Phi_\l(\mX)$ with intensity measures
	\begin{align}
		\label{eq:mixcont}
		\nu^c_{\Phi_\Lambda \given X^{(n)}}(dx,ds \given \l) &= s^{-1}e^{-s\bigl(\l + b(x)\bigr)}\, ds\, a(dx), \\
		\label{eq:mixdisc}
		\nu^d_{\Phi_\Lambda \given X^{(n)}}(\{\tilde{X}_j \},s \given \l) &\propto s^{N_{j,n}-1}e^{-s\bigl(\l + b(\tilde{X}_j)\bigr)}\, ds,
	\end{align}
	and mixing distribution with Lebesgue density proportional to
	\begin{align}
		\label{eq.mixpost}
		\pi(\l \given X^{(n)})
		&\propto \l^{n-1}e^{-\psi(\l)}\prod_{j=1}^{K_n} \frac{1}{(\l + b(\tilde{X}_j))^{N_{j,n}}},
	\end{align}
	where $\psi(\l) := \int\!\int_0^\infty \left(1-e^{-\l s}\right)s^{-1}e^{-sb(x)}\, ds\, da(x)$. For convenience of notation, let $\Lambda$ be a random variable following the mixing density
	\eqref{eq.mixpost}, for given $X^{(n)}$. Then (see \eqref{EqRepNCRM})
	the posterior distribution can be interpreted as the distribution of $(P_n(A): A\in\X)$ for
	\begin{align*}
		P_n(A)= \frac{\Psi(A) + \sum_{j=1}^{K_n} W_{j,K_n}' \delta_{\tilde{X}_j}(A) }{\Psi(\mX) + \sum_{j=1}^{K_n} W_{j,K_n}'},
	\end{align*}
	where the random measure $\Psi$ is distributed
	as $\Psi(A) = \iint \mathbbm{1}_A(x) s \, dN^c(x,s)$ conditional on $\Lambda=\l$, for a Poisson process $N^c$
	on $\mX \times (0,\infty]$ with intensity measure \eqref{eq:mixcont}, and the variables
	$W_{j,K_n}'$ follow gamma distributions with parameters $\bigl(N_{j,n},\l + b(\tilde{X}_j)\bigr)$
	(as given in \eqref{eq:mixdisc}), independent of each other and of $\Psi$.
	By the properties of the gamma distribution, for any $j \in \{1,\ldots,K_n\}$, we can represent $W_{j,K_n}'$ as the sum
	of $N_{j,K_n}$ independent exponential random variables with intensity $\l + b(\tilde{X}_j)$.
	This allows to simplify the preceding representation to
	\begin{align}
		\label{EqDefPn}
		P_n(A)=\frac{\Psi(A) + \sum_{i=1}^{n} W_i \delta_{X_i}(A) }{\Psi(\mX) + \sum_{i=1}^{n} W_i},
	\end{align}
	where conditional on $\Lambda=\l$ the variables $W_{1},\ldots, W_{n}$ are independent exponentially
	distributed with intensity parameters $\l + b(X_i)$, and are independent of $\Psi$. We denote by $\E_\Psi$ and $\E_W$ the expectations with respect to
	$\Psi$ and the $W_i$, for given observations $X^{(n)}$ and given $\Lambda$.
	Furthermore, we denote by $\E_\Lambda$ the expectation relative to the mixture measure with density \eqref{eq.mixpost},
	for given $X^{(n)}$. In this notation we need to show that
	$\bigl|\E_{\Lambda}\E_{W}\E_\Psi h\bigl(\sqrt{n}(P_n - \PP_n)\bigr) - \E h(\BB_{P_0}) \bigr|\ra0$,
	almost surely (or in probability), uniformly in $h\in\BL_1$. Define a weighted empirical distribution by $\PP_n^W=(\sumin W_i \delta_{X_i})/(\sumin W_i)$.
	We shall show that the following two statements are true, almost surely (or in probability), after which an application of the triangle inequality gives the desired result:
	\begin{align}
		\label{eq.thmfirst}
		&\sup_{h \in \BL_1} \Bigl|\E_{\Lambda}\E_W\E_\Psi h\bigl(\sqrt{n}(P_n - \PP_n)\bigr)
		-\E_\Lambda\E_W h\bigl(\sqrt{n}(\PP_n^W - \PP_n)\bigr) \Bigr|\ra 0,\\
		\label{eq.thmsecond}
		&\sup_{h \in \BL_1} \Bigl|\E_\Lambda\E_W h\bigl(\sqrt{n}( \PP_n^W- \PP_n)\bigr) - \E h(\BB_{P_0})\Bigr|\ra0.
	\end{align}
	
	An essential ingredient for proving these statements
	is that the distributions of $\Lambda$ will give probability tending to
	one to the sets $[c n/\log n,\infty)$, for some small $c>0$,
	as is shown in Lemma~\ref{lem.concentration}. This causes that
	the contribution of $\Psi$ to \eqref{EqDefPn} is negligible relative to the contribution of $\PP_n^W$,
	leading to \eqref{eq.thmfirst}. Furthermore, it causes that the $W_i$ will be asymptotically i.i.d.\
	exponential variables (with large but almost equal intensities),
	so that the processes $\PP_n^W$ are asymptotically equivalent to
	the Bayesian bootstrap process, which is known to tend to a Brownian bridge, leading to \eqref{eq.thmsecond}. To handle \eqref{eq.thmfirst}, we start by noting that, by \eqref{EqDefPn} and the definition of $\PP_n^W$,
	$$P_n-\PP_n^W=\frac{\Psi-\Psi(\mX)\PP_n^W}{\Psi(\mX)+\sumin W_i}.$$
	Because $\|\Psi\|_\F\le \Psi F$ and $\|\PP_n^W\|_\F\le M_n:=\max_{1\le i\le n}F(X_i)$,
	it follows that the left side of  \eqref{eq.thmfirst} is bounded above by
	\begin{align*}
		&\E_\Lambda\E_W\E_\Psi \bigl\|\sqrt{n}(P_n - \PP_n^W)\bigr\|_\F
		\le \sqrt n\,\E_\Lambda\E_W\E_\Psi
		\frac{\Psi F+ \Psi(\mX)M_n}{\Psi(\mX) + \sumin W_i}.
	\end{align*}
	The expression on the right side without the leading expectation on $\Lambda$ is of the form as in
	Lemma~\ref{thm.integralexpressions}, with $f=F+1_\mX M_n$ (and with $M_n$ fixed due
	to the conditioning on $X_1,\ldots, X_n$).
	Thus the preceding display can be bounded above by
	\begin{align*}
		& \sqrt{n}\, \E_\Lambda\int_0^\infty \Bigl[e^{\int \log\bigl((\Lambda + b)/(t + \Lambda + b)\bigr)\, da} \cdot
		\int \frac{f}{t + \Lambda + b}\, da
		\cdot \prodin \frac{\Lambda + b(X_i)}{t + \Lambda +b(X_i)}\Bigr]\, dt.
	\end{align*}
	The first term of the product in the integrand is bounded above by 1, since the logarithm in the exponent is negative.
	To bound the second and third terms we split the integral over $t$ in the ranges $[0, P_0 b)$ and $[P_0 b,\infty)$.
	For $t$ in the range $[0, P_0 b)$ we use the bound $\int f/(t+\Lambda+b)\,da\le af/(\Lambda+\ub)$
	on the second term, where $af=\int f\,da$. Furthermore, in this range we bound the third term by 1. This shows that the integral over the range $[0, P_0 b)$ contributes
	no more than $af \, P_0b \, \E_\Lambda\left[ (\Lambda+\ub)^{-1}  \right] $.
	For $t$ in the range $[P_0 b,\infty)$, we bound the second term by $af/(t+\Lambda+\ub)$,
	while for the third term we use Jensen's inequality on the logarithm to see that
	\begin{align*}
		\prodin \frac{\l + b(X_i)}{t + \l + b(X_j)}
		&= e^{n\int\log\big(\frac{\l + b}{t + \l + b}\bigr)\,d\PP_n} \le \Bigl(\PP_n\frac{\l + b}{t + \l + b}\Bigr)^n
		\le \Bigl(\frac{\l + \PP_nb}{t + \l + \ub}\Bigr)^n.
	\end{align*}
	Substituting these bounds in the integral, we find
	\begin{align*}
		af\,\E_\Lambda\Bigl[  (\Lambda + \PP_n b)^n \int_{P_0 b}^\infty \frac{1}{(t + \Lambda + \ub)^{n+1}}\, dt  \Bigr]
		&=  \frac{af}{n} \E_\Lambda\Bigl(\frac{\Lambda + \PP_n b}{P_0 b + \Lambda + \ub}\Bigr)^n.
	\end{align*}
	Since $\ub>0$ and $\PP_nb- P_0b\ra 0$,  almost surely (or in probability),
	the quotient on the right is bounded above by 1 eventually, almost surely.
	Thus combining the results for the two ranges,  we find that as $n\ra\infty$, almost surely,
	\begin{align*}
		&\E_\Lambda\E_W\E_\Psi \left[\frac{\Psi f}{\Psi(\mX) + \sumin W_i} \right]
		\leq af \Bigl[\E_\Lambda \frac{P_0 b}{\Lambda + \ub} + \frac{1}{n}\Bigr]\\
		&\qquad\qquad\qquad\qquad\qquad
		\le af \Bigl[\frac{P_0b}{\ub}\Pi\Bigl(\Lambda<\frac{cn}{\log n}\given X^{(n)}\Bigr)+\frac{P_0b\, \log n}{cn}+\frac 1n\Bigr].
	\end{align*}
	We substitute $f=F+1_\mX M_n$ in the right side and multiply by $\sqrt n$ to obtain a bound on \eqref{eq.thmfirst}
	of the form
	\begin{align*}
		&\bigl(aF + M_n a(\mX)\bigr)\Bigl[\frac {P_0 b}{\ub}    \sqrt n \, \Pi\Bigl(\Lambda < \frac{c  n}{\log n}\given X^{(n)}\Bigr)
		+ \frac{P_0 b\, \log n+1}{c\sqrt n}\Bigr].
	\end{align*}
	Here $M_n=O(n^{1/r})$, almost surely, for some $r>2$,  by Lemma~\ref{lem.maximum} and
	the assumption that $P_0F^r<\infty$. This shows that the second term tends to zero.
	The first term tends to zero for sufficiently small $c>0$ in view of Lemma~\ref{lem.concentration}.
	
	We are left with showing \eqref{eq.thmsecond}. Given $\Lambda$ and $X^{(n)}$, the variables $\tilde{W}_i = W_i(\Lambda + b(X_i))/(\Lambda + \ub)$
	are i.i.d.\ exponential
	variables with parameter $\Lambda+\ub$. Since this is a scale parameter,
	the normalised variables $\tilde W_i/\sumin \tilde W_i$ are equal in distribution to the normalised variables
	$\xi_i/\sumin\xi_i$, for $\xi_1,\ldots, \xi_n$ standard exponential variables, still conditionally
	on $\Lambda$ and $X^{(n)}$. This shows that the empirical process
	$\PP_n^{\tilde W}$ with weights $\tilde{W}$, is equal in
	distribution to the Bayesian bootstrap process (see Example~3.7.9 in \cite{Vandervaart2023}), conditionally given
	$\Lambda$, but then also unconditionally (still given $X^{(n)}$), as its distribution
	does not depend on $\Lambda$. By \cite{PraestgaardandWellner(1993)}
	(or Theorem~3.6.13 in \cite{Vandervaart1996}), the Bayesian bootstrap process converges
	to the Brownian bridge process $\BB_{P_0}$. Therefore, in view of the triangle inequality
	it suffices to bound the difference in \eqref{eq.thmsecond} when replacing $\PP_n^W$ by $\PP_n^{\tilde W}$.
	This difference can be bounded by, with  $M_n=\max_{1\le i\le n} F(X_i)$ as before,
	\begin{align*}
		\E_\Lambda \E_W \sqrt{n}\|\PP_n^W-\PP_n^{\tilde W}\|_\F
		&\le  2M_n\sqrt n \,\E_\Lambda\E_W\frac{\sumin |W_i - \tilde{W}_i |}{\sumin W_i}  \\
		&= 2M_n\sqrt n  \max_{1 \leq i \leq n}  b(X_i)\,\E_\Lambda\left[ \left(\Lambda + \ub\right)^{-1} \right],
	\end{align*}
	because $|W_i-\tilde W_i|= W_i\bigl(b(X_i)-\ub\bigr)$.
	We again split the expectation over $\Lambda$ in two parts to bound this by
	\begin{align*}
		&2M_n  \max_{1 \leq i \leq n}  b(X_i)\Bigl[\frac{\sqrt{n}}\ub
		\Pi\Bigl(\Lambda < \frac{cn}{\log n} \given X^{(n)}\Bigr) +  \frac{\log n}{c\sqrt{n}}\Bigr].
	\end{align*}
	Here $M_n=O(n^{1/r})$  and $\max_{1\le i\le n}b(X_i)=O(n^{1/q})$ almost surely, by Lemma~\ref{lem.maximum}
	and the integrability assumptions on $F$ and $b$, where $1/q+1/r<1/2$.
	it follows that the leading term is $O(n^{1/2-\epsilon})$, for some $\epsilon>0$.
	Together with Lemma~\ref{lem.concentration} this shows that the expression tends
	to zero, for sufficiently small $c>0$.
\end{proof}

\begin{theorem}
	[Posterior mean]
	For $P_0^\infty$-almost every sequence $X_1,X_2,\ldots$,
	\begin{align*}
		\sup_{A\in\X}\bigl|\E \bigl(P_n(A)\given X_1,\ldots,X_n\bigr)-\PP_n(A)\bigr|=O(1/n), \qquad \text{a.s.}
	\end{align*}
\end{theorem}

\begin{lemma}
	\label{thm.integralexpressions}
	Let $\Psi$ be a completely random measure without fixed atoms
	and continuous intensity measure given by
	$\nu^c(dx,ds)=s^{-1}e^{-s b(x)}\,ds\,dx$ and let $W_{1},\ldots, W_{n}$
	be independent exponential variables with means $b(X_i)$, independent of $\Psi$.
	Then, for any measurable function $f : \mX \to [0,\infty)$,
	\begin{align*}
		&\E_W\E_\Psi \left[\frac{\Psi f}{\Psi(\mX) + \sumin W_i} \right] =\int_0^\infty \Bigl[e^{\int \log\left(\frac{b}{t+b}\right)\, da}
		\cdot \int \frac{f}{t + b}\, da
		\cdot \prod_{i=1}^n \frac{b(X_i)}{t + b(X_i)}\Bigr]\, dt.
	\end{align*}
\end{lemma}

\begin{lemma}[Maximum of iid random variables]
	\label{lem.maximum}
	If $X_1,\ldots,X_n$ are i.i.d.\ random variables with $\E| X_i|^r < \infty$ for some $r > 0$, then $\max_{1 \leq i \leq n} \left|X_i \right| n^{-1/r} \stackrel{\text{as}}{\longrightarrow} 0$.
	If $X_{1,n},\ldots, \allowbreak X_{n,n}$ are i.i.d.\ random variables for every $n$ such that the variables $| X_{1,n}|^r $
	are uniformly integrable, then the assertion is true in probability.
\end{lemma}
\begin{lemma}
	\label{Lemmapsi}
	For any nonnegative measurable function $b: \mX\to(0,\infty)$, and any $\l>0$,
	we have $\int\!\int_0^\infty (1-e^{-\l s})s^{-1}e^{-sb}\, ds\, da= \int \log(1+\l/b)\, da$.
\end{lemma}
\begin{lemma}[Concentration of the mixing distribution]
	\label{lem.concentration}
	For $\Lambda$ conditionally distributed given $X^{(n)}=(X_1,\ldots,X_n)$ according to
	density given in \eqref{eq.mixpost}, for a measurable function $b: \mX\to [\ub,\infty)$
	with $P_0b<\infty$, we have for $P_0^\infty$-almost every sequence $X_1,X_2,\ldots$,
	for any $d \in \RR^+$ and $0 < c < \ub/\bigl(a(\mX) + 1 + d)\bigr)$,
	\begin{align*}
		n^d\,\Pi\left(\Lambda < \frac{cn}{\log n}  \given X^{(n)}\right) \ra 0.
	\end{align*}
	The convergence is
	uniform in functions $b: \mX\to[\ub,\infty)$ such that
	the sequence $\PP_nb$ remains uniformly bounded, almost surely. If the functions are
	uniformly bounded in probability, then the convergence is true uniformly in probability.
\end{lemma}

\begin{supplement}
\stitle{Correctly specified prior on $\a$, efficiency theory and proofs}
\sdescription{In the supplementary material, we provide additional simulation results when the prior on the sensitivity parameter is correctly specified. It also contains a derivation of the efficiency theory and some proofs not contained in the main manuscript.}
\end{supplement}

\bibliographystyle{ba}
\bibliography{Mybib}

\begin{thebibliography}{41}
\newcommand{\enquote}[1]{``#1''}
\expandafter\ifx\csname natexlab\endcsname\relax\def\natexlab#1{#1}\fi
\expandafter\ifx\csname url\endcsname\relax
  \def\url#1{{\tt #1}}\fi
\expandafter\ifx\csname urlprefix\endcsname\relax\def\urlprefix{URL }\fi
\ifx\endbibitem\undefined \let\endbibitem\relax\fi

\bibitem[{Balzer et~al.(2020)Balzer, Ayieko, Kwarisiima, Chamie, Charlebois,
  Schwab, van~der Laan, Kamya, Havlir, and Petersen}]{Balzeretal2020}
Balzer, L.~B., Ayieko, J., Kwarisiima, D., Chamie, G., Charlebois, E.~D.,
  Schwab, J., van~der Laan, M.~J., Kamya, M.~R., Havlir, D.~V., and Petersen,
  M.~L. (2020).
\newblock \enquote{Far from MCAR: Obtaining Population-level Estimates of HIV
  Viral Suppression.}
\newblock {\em Epidemiology\/}, 31(5).
\endbibitem

\bibitem[{Cox(1972)}]{Cox(1972)}
Cox, D.~R. (1972).
\newblock \enquote{Regression models and life-tables.}
\newblock {\em J. Roy. Statist. Soc. Ser. B\/}, 34: 187--220.
\newblock With discussion by F. Downton, R. Peto, D. Bartholomew, D. Lindley,
  P. Glassborow, D. Barton, S. Howard, B. Benjamin, J. Gart, L. Meshalkin, A.
  Kagan, M. Zelen, R. Barlow, J. Kalbfleisch, R. Prentice and N. Breslow, and a
  reply by D. R. Cox.
\endbibitem

\bibitem[{{D}elft {H}igh {P}erformance {C}omputing~{C}entre
  ({DHPC})(2022)}]{DHPC2022}
{D}elft {H}igh {P}erformance {C}omputing~{C}entre ({DHPC}) (2022).
\newblock \enquote{{D}elft{B}lue {S}upercomputer ({P}hase 1).}
\newblock \url{https://www.tudelft.nl/dhpc/ark:/44463/DelftBluePhase1}.
\endbibitem

\bibitem[{Dorie et~al.(2016)Dorie, Harada, Carnegie, and Hill}]{Dorieetal2016}
Dorie, V., Harada, M., Carnegie, N.~B., and Hill, J. (2016).
\newblock \enquote{A flexible, interpretable framework for assessing
  sensitivity to unmeasured confounding.}
\newblock {\em Statistics in Medicine\/}, 35(20): 3453--3470.
\newline\urlprefix\url{https://onlinelibrary.wiley.com/doi/abs/10.1002/sim.6973}
\endbibitem

\bibitem[{Dudley(2014)}]{Dudley}
Dudley, R.~M. (2014).
\newblock {\em Uniform central limit theorems\/}, volume 142 of {\em Cambridge
  Studies in Advanced Mathematics\/}.
\newblock Cambridge University Press, New York, second edition.
\endbibitem

\bibitem[{Eggen et~al.(2023)Eggen, Van~der Pas, and Van~der Vaart}]{Supplement}
Eggen, B., Van~der Pas, S.~L., and Van~der Vaart, A.~W. (2023).
\newblock \enquote{Supplement to "Bayesian sensitivity analysis for a missing
  data model".}
\endbibitem

\bibitem[{Ferguson(1973)}]{Ferguson1973}
Ferguson, T.~S. (1973).
\newblock \enquote{{A Bayesian Analysis of Some Nonparametric Problems}.}
\newblock {\em The Annals of Statistics\/}, 1(2): 209 -- 230.
\newline\urlprefix\url{https://doi.org/10.1214/aos/1176342360}
\endbibitem

\bibitem[{Gelman et~al.(1997)Gelman, Gilks, and Roberts}]{Gelman1997}
Gelman, A., Gilks, W.~R., and Roberts, G.~O. (1997).
\newblock \enquote{{Weak convergence and optimal scaling of random walk
  Metropolis algorithms}.}
\newblock {\em The Annals of Applied Probability\/}, 7(1): 110 -- 120.
\newline\urlprefix\url{https://doi.org/10.1214/aoap/1034625254}
\endbibitem

\bibitem[{Ghosal and Van~der Vaart(2017)}]{Ghosal2017}
Ghosal, S. and Van~der Vaart, A. (2017).
\newblock {\em Fundamentals of Nonparametric Bayesian Inference\/}.
\newblock Cambridge University Press.
\endbibitem

\bibitem[{Gustafson and McCandless(2018)}]{GustafsonMcCandlessLawrence2018}
Gustafson, P. and McCandless, L.~C. (2018).
\newblock \enquote{When is a sensitivity parameter exactly that?}
\newblock {\em Statist. Sci.\/}, 33(1): 86--95.
\newline\urlprefix\url{https://doi.org/10.1214/17-STS632}
\endbibitem

\bibitem[{Gustafson et~al.(2010)Gustafson, McCandless, Levy, and
  Richardson}]{Gustafson2010}
Gustafson, P., McCandless, L.~C., Levy, A.~R., and Richardson, S. (2010).
\newblock \enquote{Simplified {B}ayesian sensitivity analysis for mismeasured
  and unobserved confounders.}
\newblock {\em Biometrics\/}, 66(4): 1129--1137.
\newline\urlprefix\url{https://doi.org/10.1111/j.1541-0420.2009.01377.x}
\endbibitem

\bibitem[{Hammer et~al.(1996)Hammer, Katzenstein, Hughes, Gundacker, Schooley,
  Haubrich, Henry, Lederman, Phair, Niu, Hirsch, and Merigan}]{Hammeretal96}
Hammer, S.~M., Katzenstein, D.~A., Hughes, M.~D., Gundacker, H., Schooley,
  R.~T., Haubrich, R.~H., Henry, W.~K., Lederman, M.~M., Phair, J.~P., Niu, M.,
  Hirsch, M.~S., and Merigan, T.~C. (1996).
\newblock \enquote{A Trial Comparing Nucleoside Monotherapy with Combination
  Therapy in HIV-Infected Adults with CD4 Cell Counts from 200 to 500 per Cubic
  Millimeter.}
\newblock {\em New England Journal of Medicine\/}, 335(15): 1081--1090.
\newblock PMID: 8813038.
\newline\urlprefix\url{https://doi.org/10.1056/NEJM199610103351501}
\endbibitem

\bibitem[{Hastings(1970)}]{Hastings1970}
Hastings, W.~K. (1970).
\newblock \enquote{Monte Carlo Sampling Methods Using Markov Chains and Their
  Applications.}
\newblock {\em Biometrika\/}, 57(1): 97--109.
\newline\urlprefix\url{http://www.jstor.org/stable/2334940}
\endbibitem

\bibitem[{Hernán and Robins(2022)}]{RobinsHernan2020}
Hernán, M. and Robins, J. (2022).
\newblock \enquote{Causal Inference: What If.}
\endbibitem

\bibitem[{Jacob et~al.(2017)Jacob, Murray, Holmes, and Robert}]{Jacob2017}
Jacob, P.~E., Murray, L.~M., Holmes, C.~C., and Robert, C.~P. (2017).
\newblock \enquote{Better together? Statistical learning in models made of
  modules.}
\newblock ArXiv:1708.08719.
\endbibitem

\bibitem[{James(2005)}]{James2005}
James, L. (2005).
\newblock \enquote{Bayesian {P}oisson process partition calculus with an
  application to {B}ayesian {L}\'evy moving averages.}
\newblock {\em Ann. Statist.\/}, 33(4): 1771--1799.
\newline\urlprefix\url{http://dx.doi.org.prox.lib.ncsu.edu/10.1214/009053605000000336}
\endbibitem

\bibitem[{James et~al.(2009)James, Lijoi, and
  Pr{\"u}nster}]{JamesLijoiandPrunster(2009)}
James, L., Lijoi, A., and Pr{\"u}nster, I. (2009).
\newblock \enquote{Posterior analysis for normalized random measures with
  independent increments.}
\newblock {\em Scand. J. Stat.\/}, 36(1): 76--97.
\newline\urlprefix\url{http://dx.doi.org.prox.lib.ncsu.edu/10.1111/j.1467-9469.2008.00609.x}
\endbibitem

\bibitem[{Kim(2006)}]{Kim(2006)}
Kim, Y. (2006).
\newblock \enquote{The {B}ernstein-von {M}ises theorem for the proportional
  hazard model.}
\newblock {\em Ann. Statist.\/}, 34(4): 1678--1700.
\newline\urlprefix\url{http://dx.doi.org.prox.lib.ncsu.edu/10.1214/009053606000000533}
\endbibitem

\bibitem[{Kim and Lee(2003)}]{KimandLee(2003b)}
Kim, Y. and Lee, J. (2003).
\newblock \enquote{Bayesian bootstrap for proportional hazards models.}
\newblock {\em Ann. Statist.\/}, 31(6): 1905--1922.
\newline\urlprefix\url{http://dx.doi.org.prox.lib.ncsu.edu/10.1214/aos/1074290331}
\endbibitem

\bibitem[{Kingman(1967)}]{Kingman(1967)}
Kingman, J. (1967).
\newblock \enquote{Completely random measures.}
\newblock {\em Pacific J. Math.\/}, 21: 59--78.
\endbibitem

\bibitem[{Kingman(1975)}]{Kingman(1975)}
--- (1975).
\newblock \enquote{Random discrete distribution.}
\newblock {\em J. Roy. Statist. Soc. Ser. B\/}, 37: 1--22.
\newblock With a discussion by S. J. Taylor, A. G. Hawkes, A. M. Walker, D. R.
  Cox, A. F. M. Smith, B. M. Hill, P. J. Burville, T. Leonard and a reply by
  the author.
\endbibitem

\bibitem[{Little and Rubin(2019)}]{LittleRubin}
Little, R. J.~A. and Rubin, D.~B. (2019).
\newblock {\em Statistical analysis with missing data\/}.
\newblock Wiley Series in Probability and Statistics. Wiley-Interscience [John
  Wiley \& Sons], Hoboken, NJ, third edition.
\endbibitem

\bibitem[{Liu et~al.(2013)Liu, Kuramoto, and Stuart}]{Liu2013}
Liu, W., Kuramoto, S., and Stuart, E. (2013).
\newblock \enquote{An introduction to sensitivity analysis for unobserved
  confounding in nonexperimental prevention research.}
\newblock {\em Prevention Science\/}, 14(6): 570--580.
\endbibitem

\bibitem[{Lo(1983)}]{Lo1983}
Lo, A.~Y. (1983).
\newblock \enquote{Weak Convergence for Dirichlet Processes.}
\newblock {\em Sankhyā: The Indian Journal of Statistics, Series A
  (1961-2002)\/}, 45(1): 105--111.
\newline\urlprefix\url{http://www.jstor.org/stable/25050418}
\endbibitem

\bibitem[{Lo(1986)}]{Lo1986}
--- (1986).
\newblock \enquote{A Remark on the Limiting Posterior Distribution of the
  Multiparameter Dirichlet Process.}
\newblock {\em Sankhyā: The Indian Journal of Statistics, Series A
  (1961-2002)\/}, 48(2): 247--249.
\newline\urlprefix\url{http://www.jstor.org/stable/25050593}
\endbibitem

\bibitem[{MacEachern(1999)}]{MacEachern(1999)}
MacEachern, S. (1999).
\newblock \enquote{Dependent nonparametric processes.}
\newblock In {\em ASA {P}roceedings of the {S}ection on {B}ayesian
  {S}tatistical {S}cience\/}, 50--55. American Statistical Association, pp.
  50--55, Alexandria, VA.
\endbibitem

\bibitem[{McCandless and Gustafson(2017)}]{McCandlessGustafson2017}
McCandless, L.~C. and Gustafson, P. (2017).
\newblock \enquote{A comparison of {B}ayesian and {M}onte {C}arlo sensitivity
  analysis for unmeasured confounding.}
\newblock {\em Stat. Med.\/}, 36(18): 2887--2901.
\newline\urlprefix\url{https://doi.org/10.1002/sim.7298}
\endbibitem

\bibitem[{McCandless et~al.(2007)McCandless, Gustafson, and
  Levy}]{McCandlessetal2007}
McCandless, L.~C., Gustafson, P., and Levy, A. (2007).
\newblock \enquote{Bayesian sensitivity analysis for unmeasured confounding in
  observational studies.}
\newblock {\em Stat. Med.\/}, 26(11): 2331--2347.
\newline\urlprefix\url{https://doi.org/10.1002/sim.2711}
\endbibitem

\bibitem[{McCandless et~al.(2012)McCandless, Gustafson, Levy, and
  Richardson}]{McCandlessetal2012}
McCandless, L.~C., Gustafson, P., Levy, A.~R., and Richardson, S. (2012).
\newblock \enquote{Hierarchical priors for bias parameters in {B}ayesian
  sensitivity analysis for unmeasured confounding.}
\newblock {\em Stat. Med.\/}, 31(4): 383--396.
\newline\urlprefix\url{https://doi.org/10.1002/sim.4453}
\endbibitem

\bibitem[{Metropolis et~al.(1953)Metropolis, Rosenbluth, Rosenbluth, Teller,
  and Teller}]{Metropolis1953}
Metropolis, N., Rosenbluth, A.~W., Rosenbluth, M.~N., Teller, A.~H., and
  Teller, E. (1953).
\newblock \enquote{Equation of State Calculations by Fast Computing Machines.}
\newblock {\em The Journal of Chemical Physics\/}, 21(6): 1087--1092.
\newline\urlprefix\url{https://doi.org/10.1063/1.1699114}
\endbibitem

\bibitem[{Moss and Rousseau(2022)}]{Moss2022}
Moss, D. and Rousseau, J. (2022).
\newblock \enquote{Efficient {B}ayesian estimation and use of cut posterior in
  semiparametric hidden {M}arkov models.}
\newblock ArXiv:2203.06081.
\endbibitem

\bibitem[{Pr{\ae}stgaard and Wellner(1993)}]{PraestgaardandWellner(1993)}
Pr{\ae}stgaard, J. and Wellner, J. (1993).
\newblock \enquote{Exchangeably weighted bootstraps of the general empirical
  process.}
\newblock {\em Ann. Probab.\/}, 21(4): 2053--2086.
\newline\urlprefix\url{http://links.jstor.org.prox.lib.ncsu.edu/sici?sici=0091-1798(199310)21:4<2053:EWBOTG>2.0.CO;2-W&origin=MSN}
\endbibitem

\bibitem[{Quintana et~al.(2020)Quintana, Mueller, Jara, and
  MacEachern}]{Quintana2020}
Quintana, F.~A., Mueller, P., Jara, A., and MacEachern, S.~N. (2020).
\newblock \enquote{The Dependent Dirichlet Process and Related Models.}
\newline\urlprefix\url{https://arxiv.org/abs/2007.06129}
\endbibitem

\bibitem[{Ray and van~der Vaart(2020)}]{RayvdV}
Ray, K. and van~der Vaart, A. (2020).
\newblock \enquote{Semiparametric {B}ayesian causal inference.}
\newblock {\em Ann. Statist.\/}, 48(5): 2999--3020.
\newline\urlprefix\url{https://doi-org.tudelft.idm.oclc.org/10.1214/19-AOS1919}
\endbibitem

\bibitem[{Robins et~al.(2000)Robins, Rotnitzky, and
  Scharfstein}]{RobinsRotnitzkyScharfstein}
Robins, J.~M., Rotnitzky, A., and Scharfstein, D.~O. (2000).
\newblock \enquote{Sensitivity analysis for selection bias and unmeasured
  confounding in missing data and causal inference models.}
\newblock In {\em Statistical models in epidemiology, the environment, and
  clinical trials ({M}inneapolis, {MN}, 1997)\/}, volume 116 of {\em IMA Vol.
  Math. Appl.\/}, 1--94. Springer, New York.
\newline\urlprefix\url{https://doi.org/10.1007/978-1-4612-1284-3_1}
\endbibitem

\bibitem[{Scharfstein et~al.(2003)Scharfstein, Daniels, and
  Robins}]{Scharfstein2003}
Scharfstein, D.~O., Daniels, M.~J., and Robins, J. (2003).
\newblock \enquote{Incorporating prior beliefs about selection bias into the
  analysis of randomized trials with missing outcomes.}
\newblock {\em Biostatistics\/}, 4(4): 495--512.
\endbibitem

\bibitem[{Sethuraman(1994)}]{Sethuraman1994}
Sethuraman, J. (1994).
\newblock \enquote{A constructive definition of Dirichlet priors.}
\newblock {\em Statistica Sinica\/}, 4(2): 639--650.
\newline\urlprefix\url{http://www.jstor.org/stable/24305538}
\endbibitem

\bibitem[{van~der Vaart(1991)}]{vdVHadamard}
van~der Vaart, A. (1991).
\newblock \enquote{Efficiency and {H}adamard differentiability.}
\newblock {\em Scand. J. Statist.\/}, 18(1): 63--75.
\endbibitem

\bibitem[{Van~der Vaart(1998)}]{Vandervaart1998}
Van~der Vaart, A. (1998).
\newblock {\em Asymptotic Statistics\/}.
\newblock Cambridge University Press.
\endbibitem

\bibitem[{Van~der Vaart and Wellner(1996)}]{Vandervaart1996}
Van~der Vaart, A. and Wellner, J. (1996).
\newblock {\em Weak Convergence and Empirical Processes\/}.
\newblock Springer, New York, NY.
\endbibitem

\bibitem[{Van~der Vaart and Wellner(2023)}]{Vandervaart2023}
--- (2023).
\newblock {\em Weak Convergence and Empirical Processes, 2nd edition\/}.
\newblock Springer, New York, NY.
\endbibitem

\end{thebibliography}


\begin{thebibliography}{2}
\providecommand{\natexlab}[1]{#1}
\providecommand{\url}[1]{\texttt{#1}}
\expandafter\ifx\csname urlstyle\endcsname\relax
  \providecommand{\doi}[1]{doi: #1}\else
  \providecommand{\doi}{doi: \begingroup \urlstyle{rm}\Url}\fi

\bibitem[Ghosal and Van~der Vaart(2017)]{Ghosal2017}
S.~Ghosal and A.W. Van~der Vaart.
\newblock \emph{Fundamentals of Nonparametric Bayesian Inference}.
\newblock Cambridge University Press, 2017.

\bibitem[van~der Vaart(1991)]{vanderVaart(1991)}
A.~van~der Vaart.
\newblock On differentiable functionals.
\newblock \emph{Ann. Statist.}, 19\penalty0 (1):\penalty0 178--204, 1991.
\newblock ISSN 0090-5364.
\newblock \doi{10.1214/aos/1176347976}.
\newblock URL \url{http://dx.doi.org/10.1214/aos/1176347976}.

\end{thebibliography}


\end{document}


\maketitle

\section{Correctly specified \texorpdfstring{$\a$}{a} simulations}
In this section we discuss the results of the simulation study when the prior on $\alpha$ is correctly specified. 

As discussed in the main paper, one of the differences between the two parametrisations is the posterior of $\a$ being dissimilar from the prior in the $P, \eta$ parametrisation. When the prior on $\a$ is misspecified, this effect can be seen in the simulations. This setting is natural, as it is not assumed that the model is correct. However, it is still relevant to understand the behaviour of the posteriors when $\a$ is correctly specified. We use the same setting as in Section~6 of the main paper, the only difference being that the mean of the prior on $\a$ is taken to be $\a_0 = 2$. The same results as in the main paper can be found in Figure~\ref{fig:posteriors} and the coverage for the different priors can be found in Table~\ref{tab:alphaprior1} and Table~\ref{tab:alphaprior2}.
\begin{table}[!ht]
	\centering
	\begin{tabular}{r|c|c|c}
		& $n = 100$ & $n = 1000$ & $n = 10000$ \\
		& Coverage | Length & Coverage | Length & Coverage | Length \\
		\hline  $H, \a$ &  \hspace*{10pt}0.560 |  0.946 & \hspace*{10pt}0.328 | 0.700 & \hspace*{10pt}0.026 | 0.662 \\
		$P, \eta, \a$ & \hspace*{10pt}0.531 | 0.932 & \hspace*{10pt}0.304 | 0.687 & \hspace*{10pt}0.131 | 0.635 
	\end{tabular}
	\caption{Coverage of the $90 \%$-credible intervals of the posterior distribution of the functional of interest for the different parametrisations with $\alpha \sim N(1,0.25)$.}
	\label{tab:alphaprior1}
\end{table}
\begin{table}[!ht]
	\centering
	\begin{tabular}{r|c|c|c}
		& $n = 100$ & $n = 1000$ & $n = 10000$ \\
		& Coverage | Length & Coverage | Length & Coverage | Length \\
		\hline  $H, \alpha$ & \hspace*{10pt}0.830 |  1.024 & \hspace*{10pt}0.991 | 0.735 & \hspace*{10pt}1.000 | 0.668 \\
		$P, \eta, \alpha$ & \hspace*{10pt}0.816 | 1.011 & \hspace*{10pt}0.986 | 0.724 & \hspace*{10pt}1.000 | 0.650
	\end{tabular}
	\caption{Coverage of the $90 \%$-credible intervals of the posterior distribution of the functional of interest for the different parametrisations with $\alpha \sim N(2,0.25)$.}
	\label{tab:alphaprior2}
\end{table}

\begin{table}[!ht]
	\centering
	\begin{tabular}{|m{0.3\textwidth}|m{0.3\textwidth}|m{0.3\textwidth}|}
		\hline
		\begin{mycenter}[2pt] $\mathbf{n = 100}$ \end{mycenter} & \begin{mycenter}[2pt]
			$\mathbf{n=1000}$	\end{mycenter} &  \begin{mycenter}[2pt]
			$\mathbf{n=10000}$	\end{mycenter}  \\[-15pt] \hline \vspace*{-0.3in}\hspace*{-0.26in} \vspace*{-0.3in} \includegraphics[scale = 0.457]{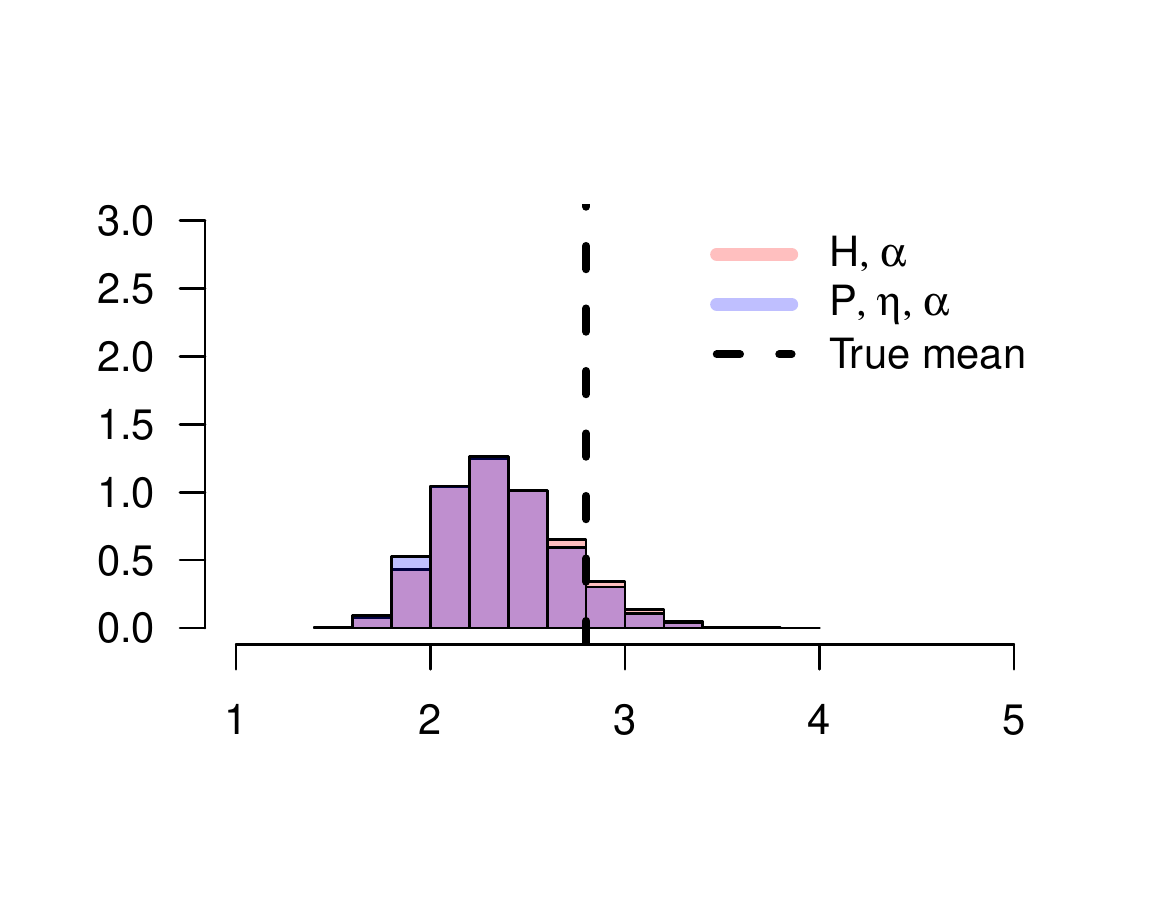}&\vspace*{-0.3in}\hspace*{-0.26in} \vspace*{-0.3in} \includegraphics[scale = 0.457]{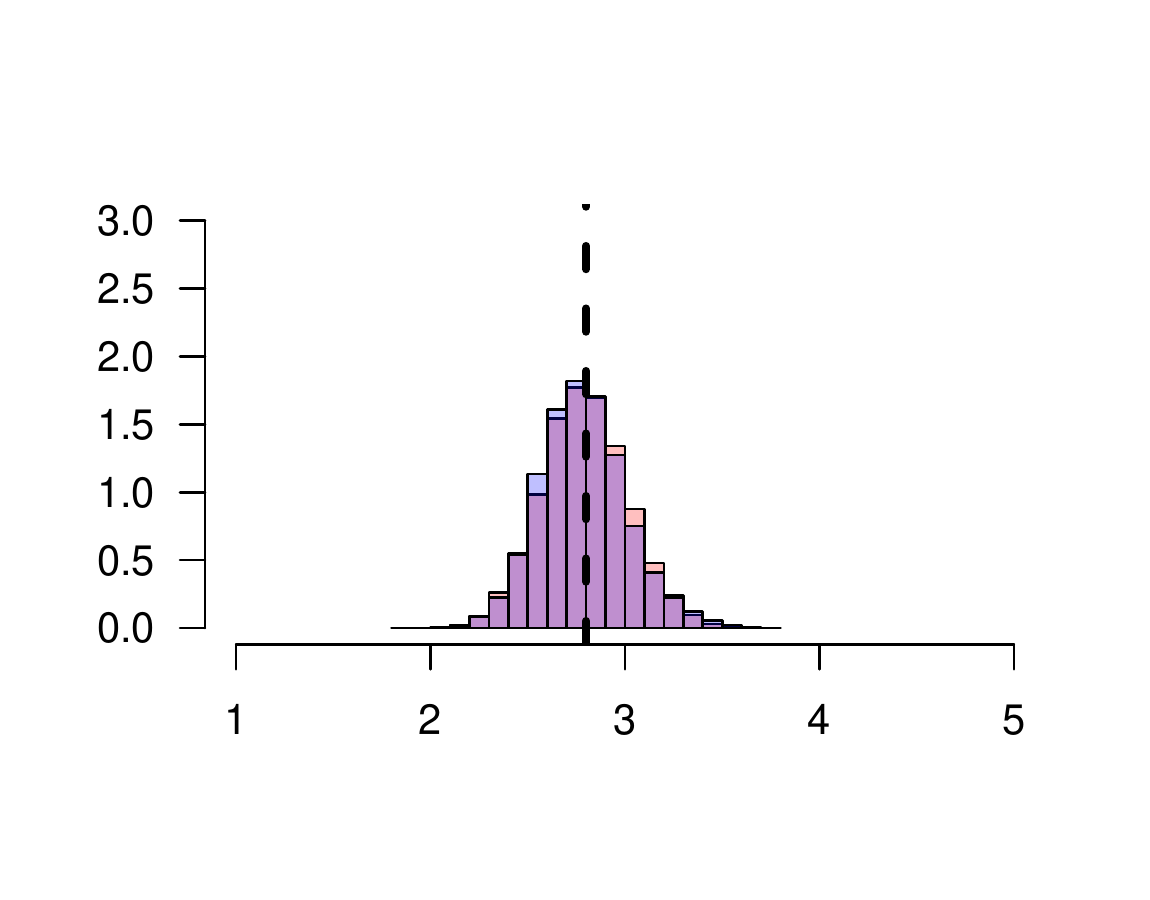}&\vspace*{-0.3in}\hspace*{-0.26in}\vspace*{-0.3in} \includegraphics[scale = 0.457]{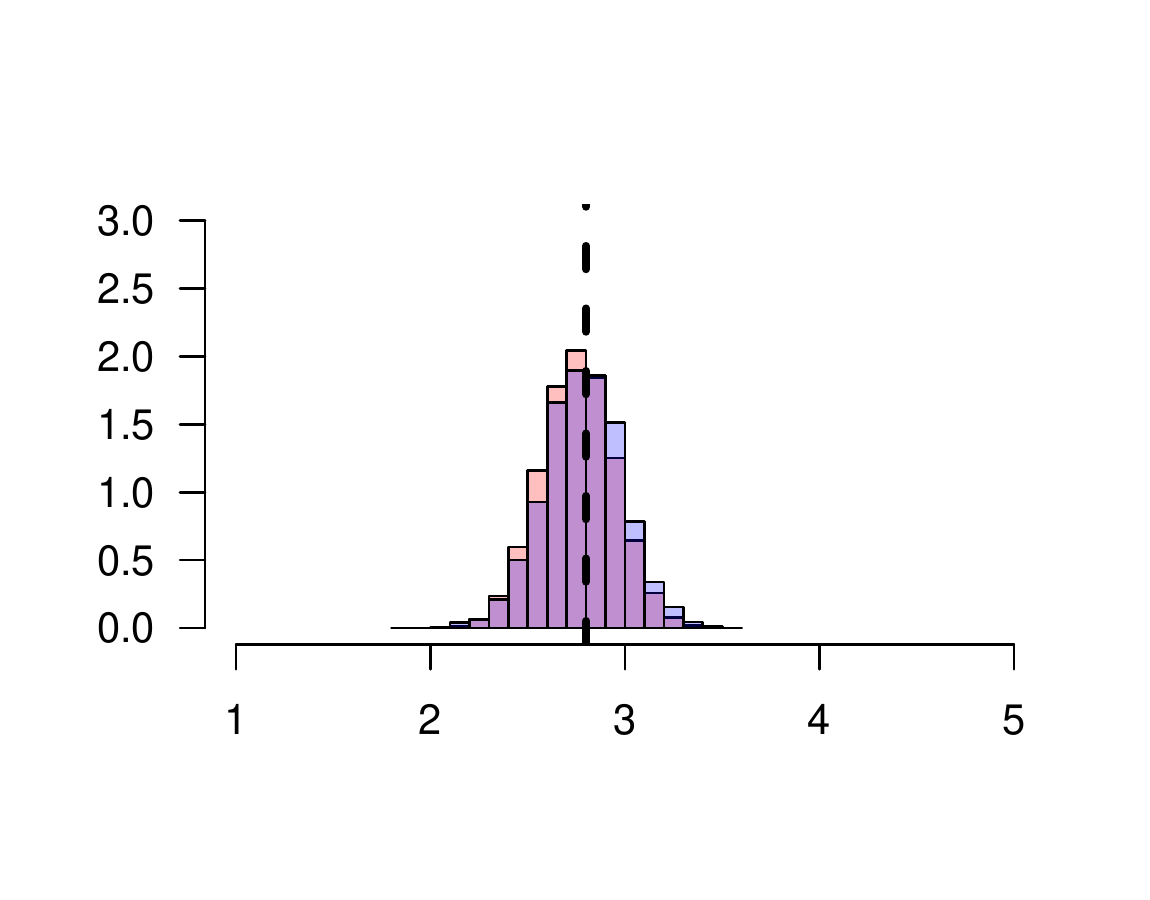}\\
		\hline \vspace*{-0.3in}\hspace*{-0.26in} \vspace*{-0.3in} \includegraphics[scale = 0.447]{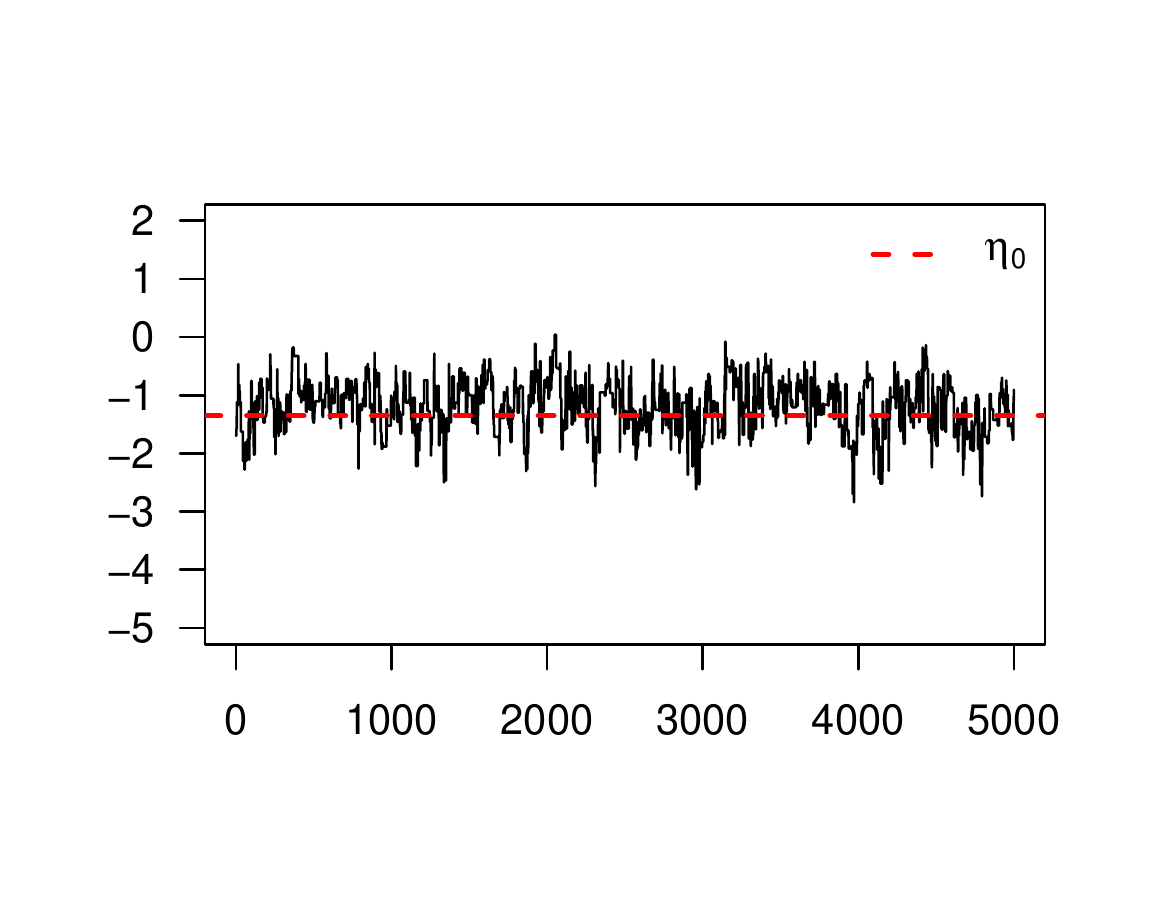}&\vspace*{-0.3in}\hspace*{-0.26in} \vspace*{-0.3in} \includegraphics[scale = 0.447]{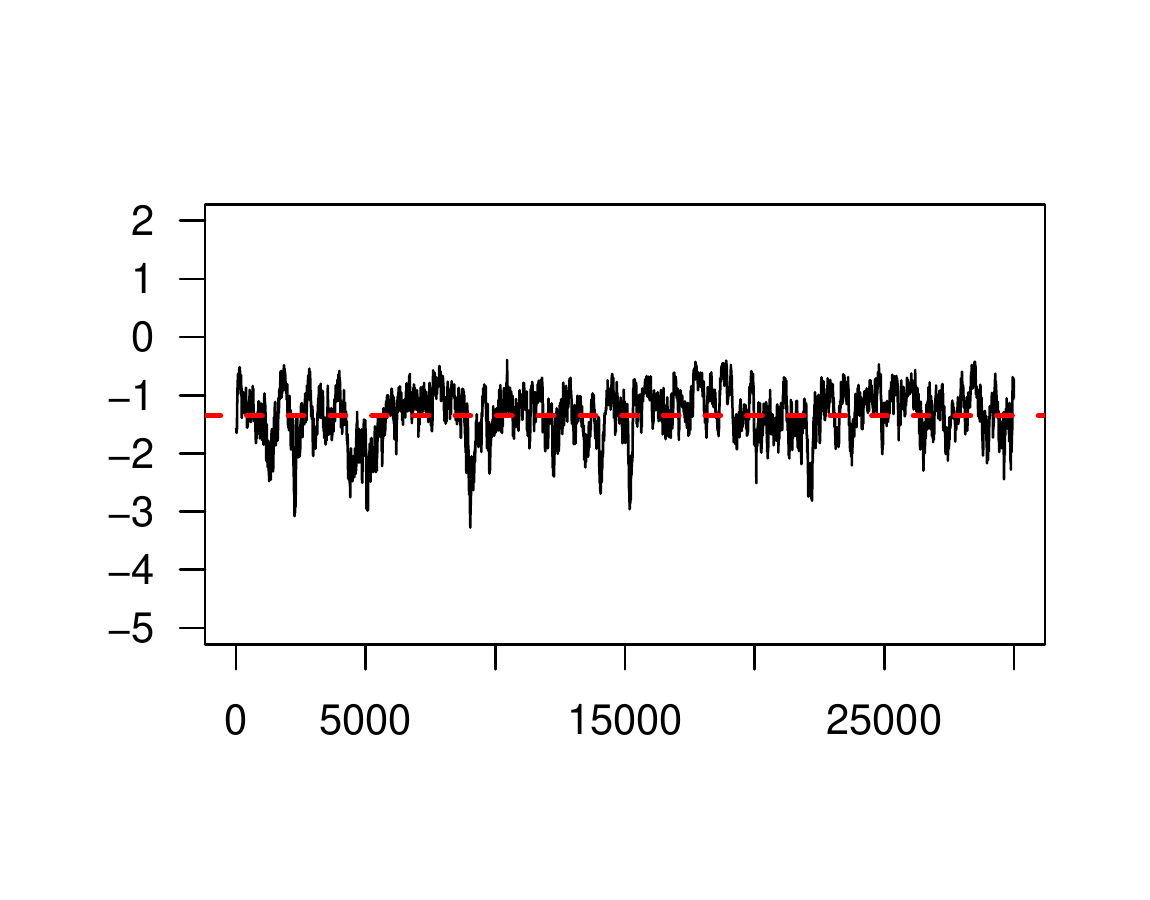}&\vspace*{-0.3in}\hspace*{-0.26in}\vspace*{-0.3in} \includegraphics[scale = 0.447]{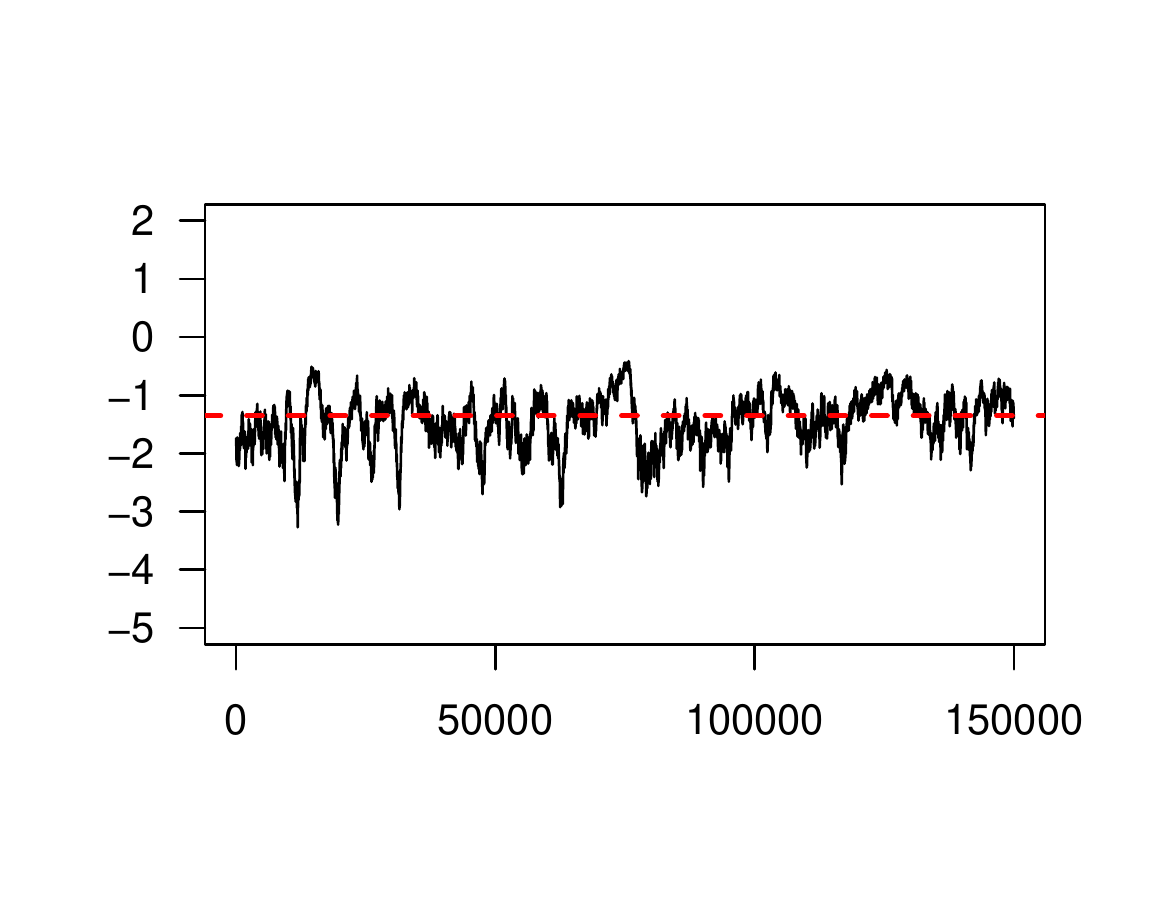}\\
		\hline \vspace*{-0.3in}\hspace*{-0.26in} \vspace*{-0.3in} \includegraphics[scale = 0.447]{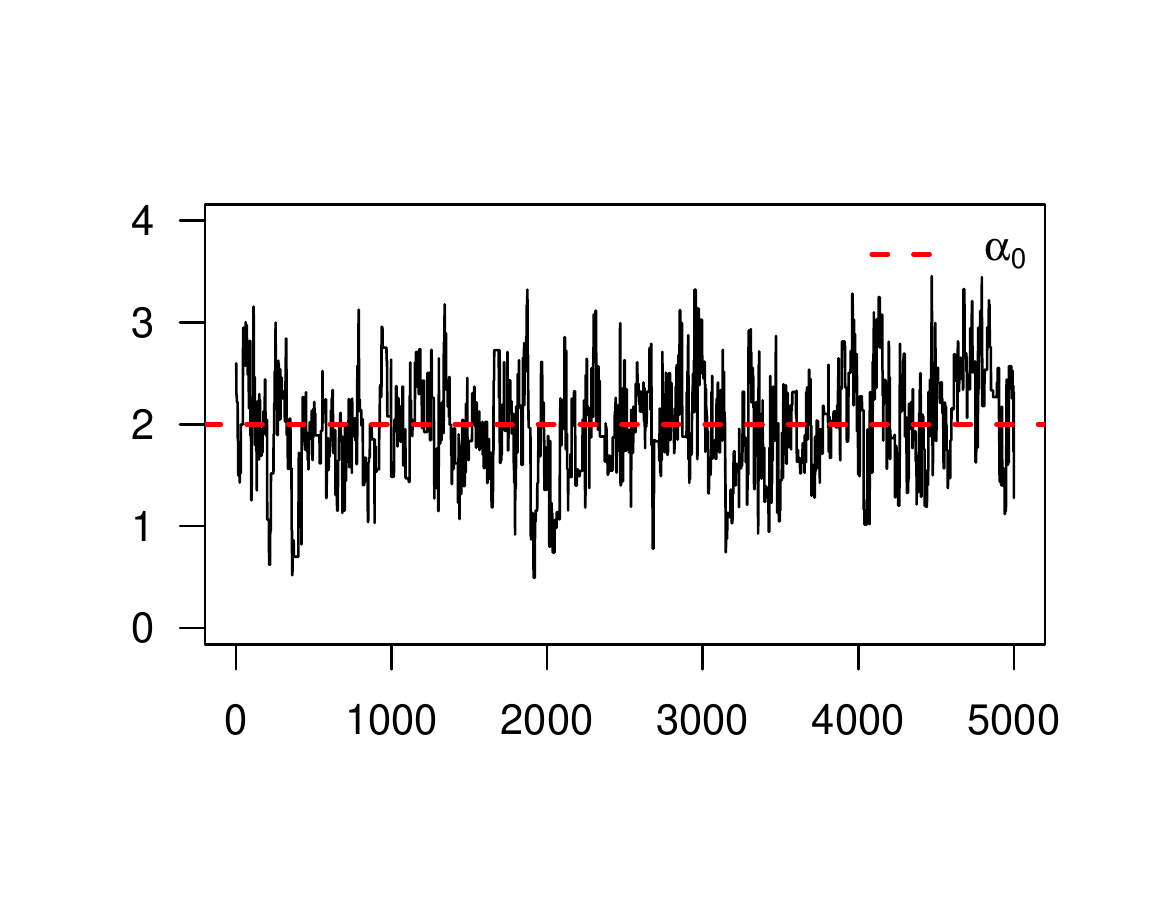}&\vspace*{-0.3in}\hspace*{-0.26in} \vspace*{-0.3in} \includegraphics[scale = 0.447]{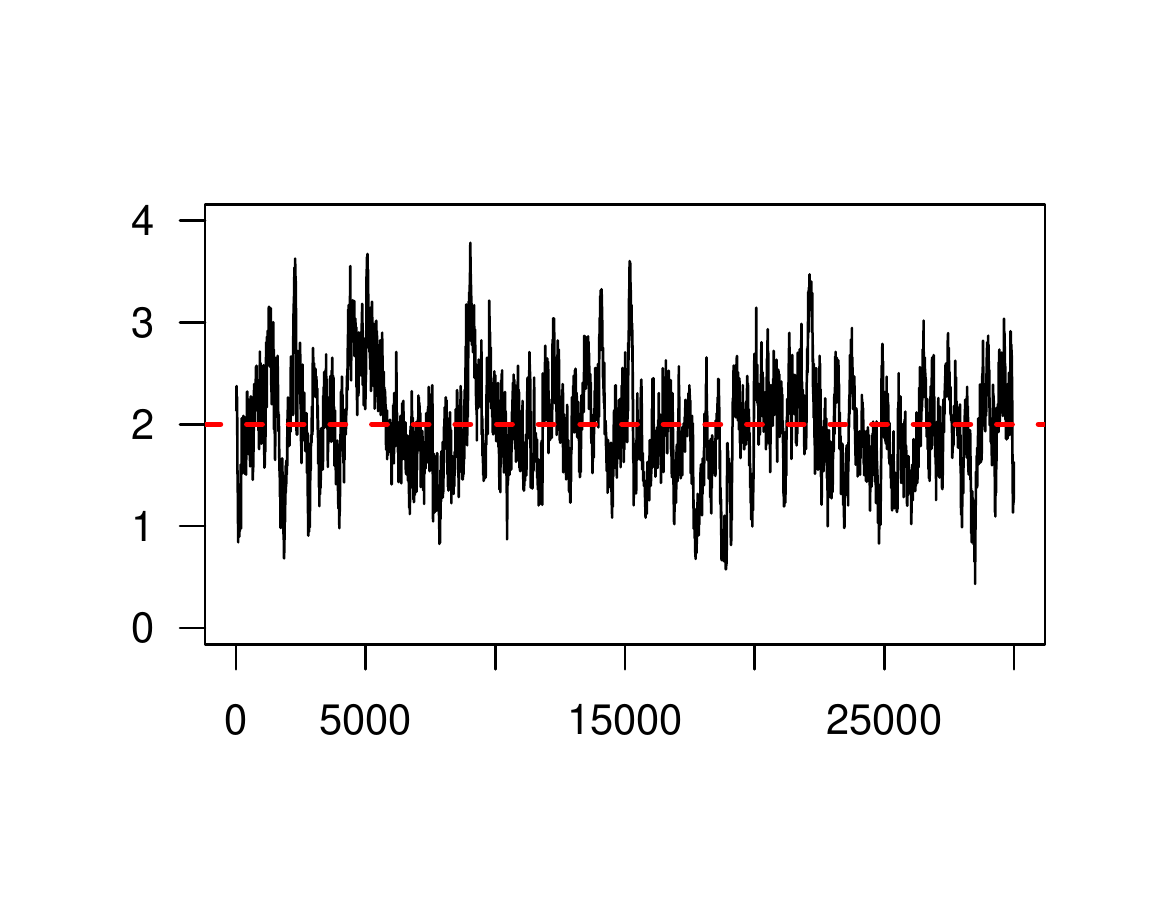}&\vspace*{-0.3in}\hspace*{-0.26in}\vspace*{-0.3in} \includegraphics[scale = 0.447]{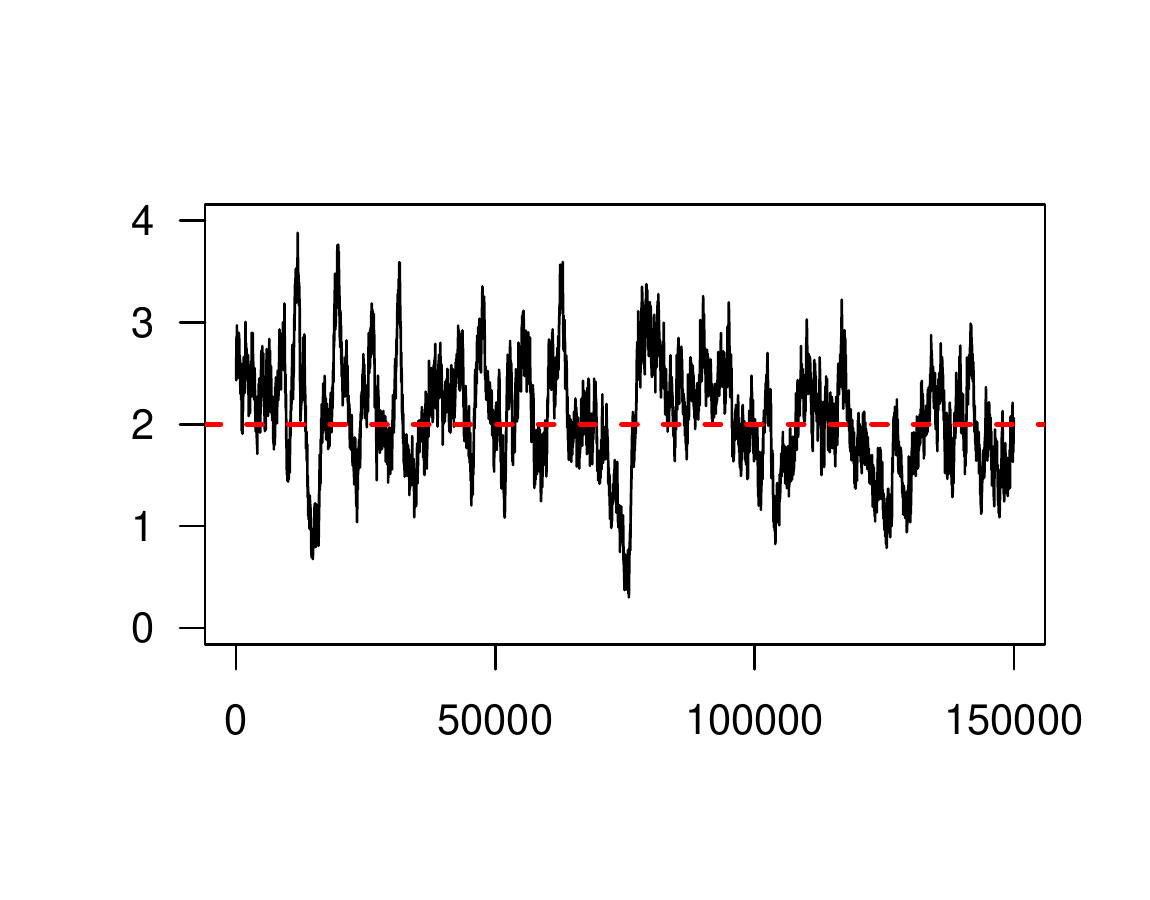}\\
		\hline \vspace*{-0.3in}\hspace*{-0.26in} \vspace*{-0.3in} \includegraphics[scale = 0.457]{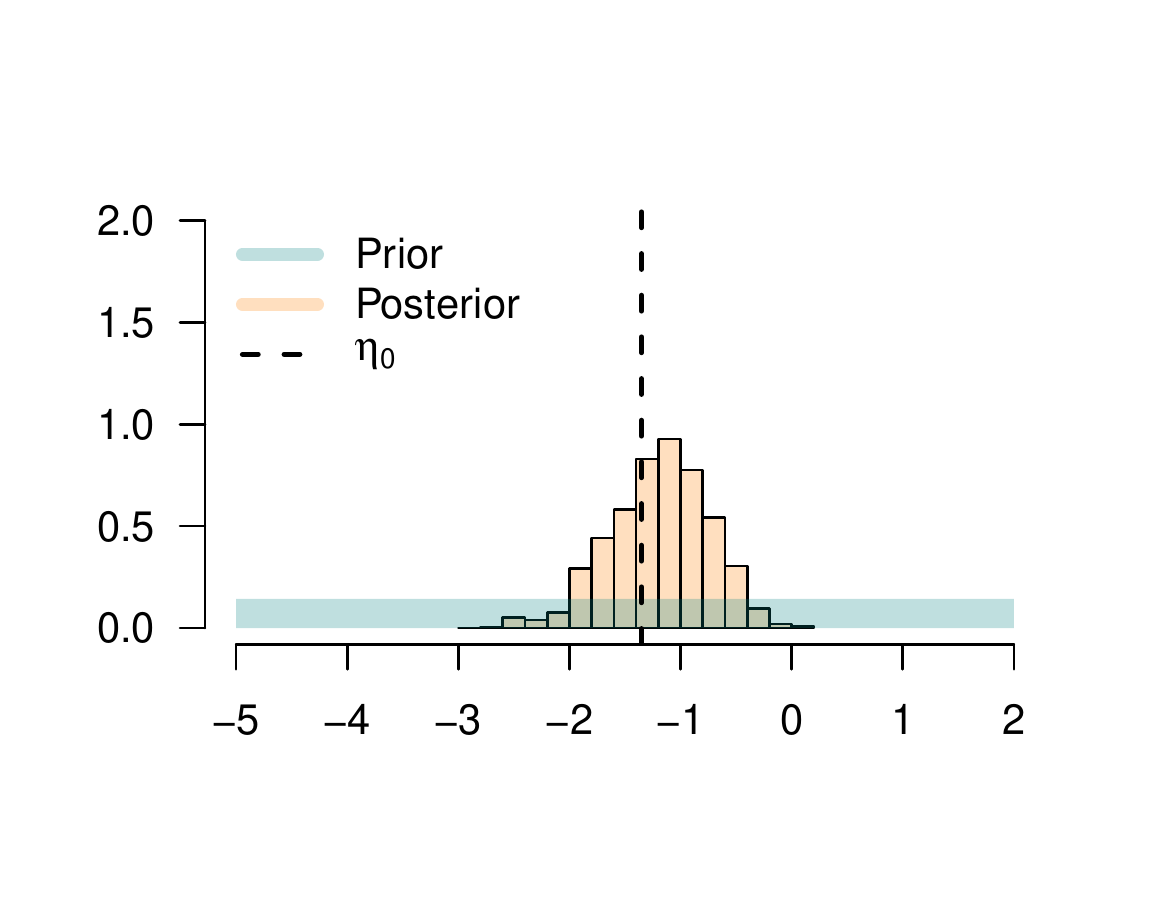}&\vspace*{-0.3in}\hspace*{-0.26in} \vspace*{-0.3in} \includegraphics[scale = 0.457]{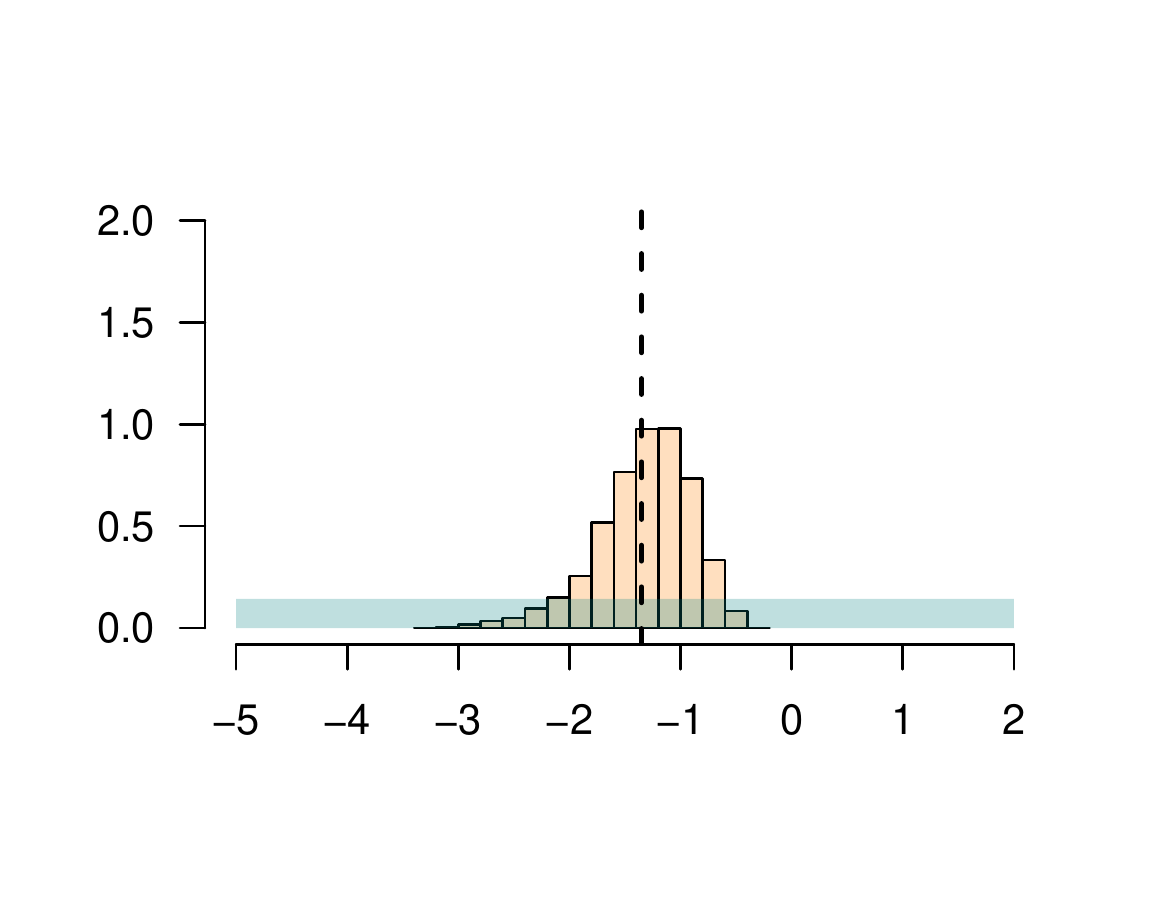}&\vspace*{-0.3in}\hspace*{-0.26in}\vspace*{-0.3in} \includegraphics[scale = 0.457]{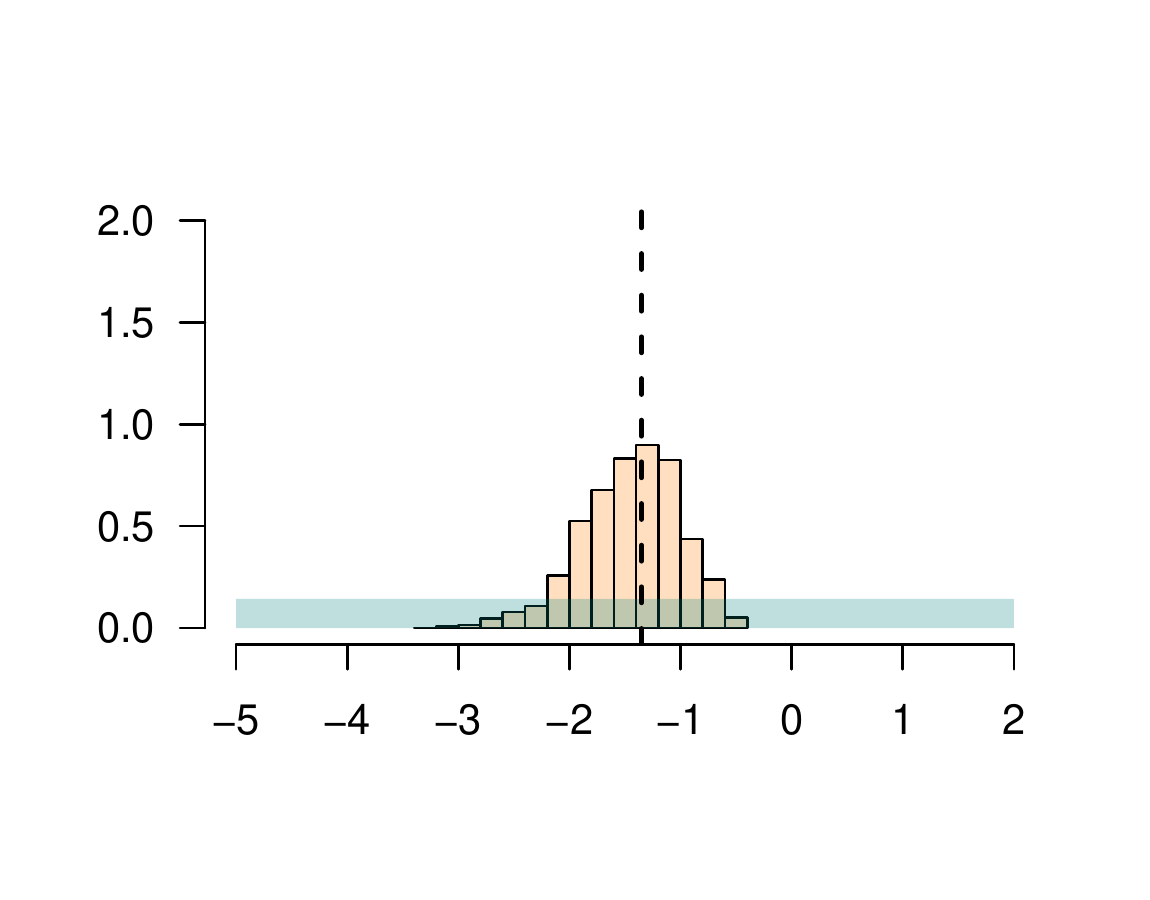}\\
		\hline \vspace*{-0.3in}\hspace*{-0.26in} \vspace*{-0.3in} \includegraphics[scale = 0.457]{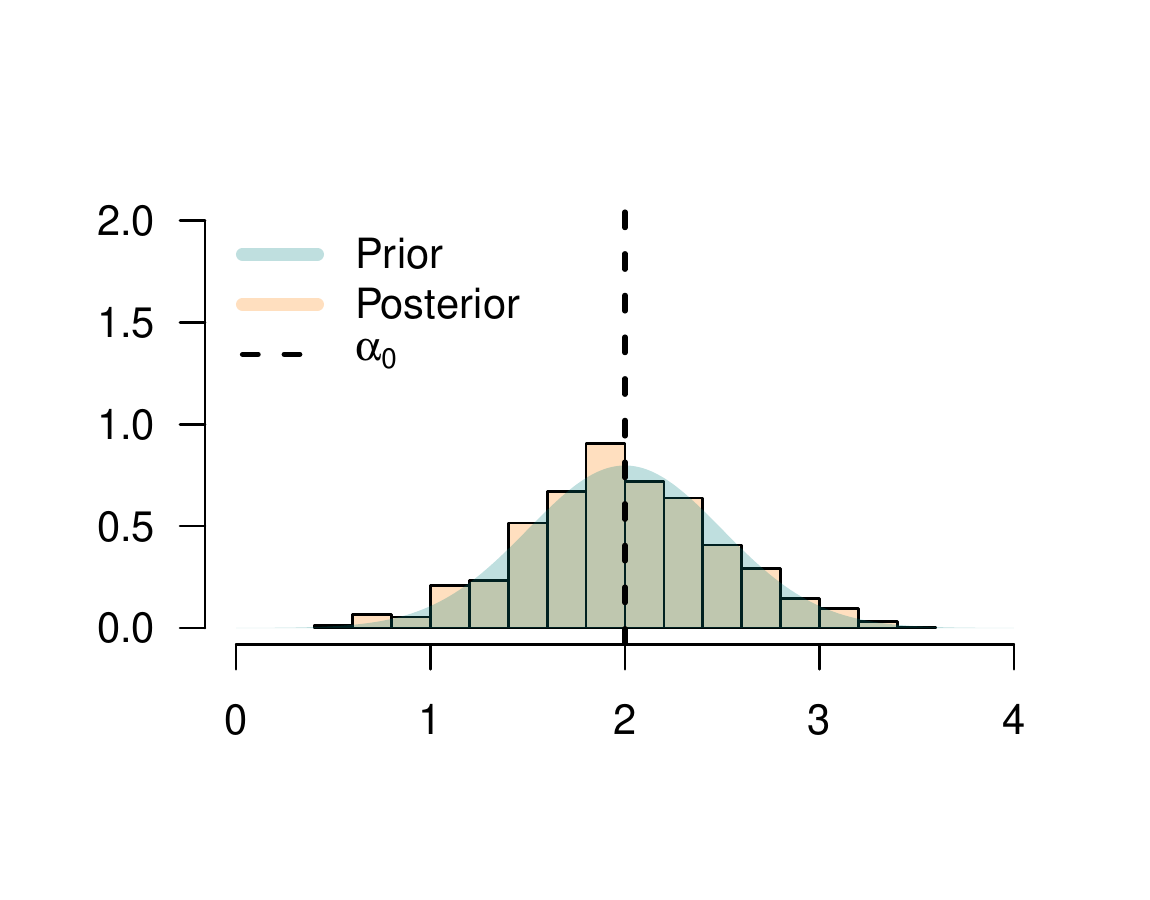}&\vspace*{-0.3in}\hspace*{-0.26in} \vspace*{-0.3in} \includegraphics[scale = 0.457]{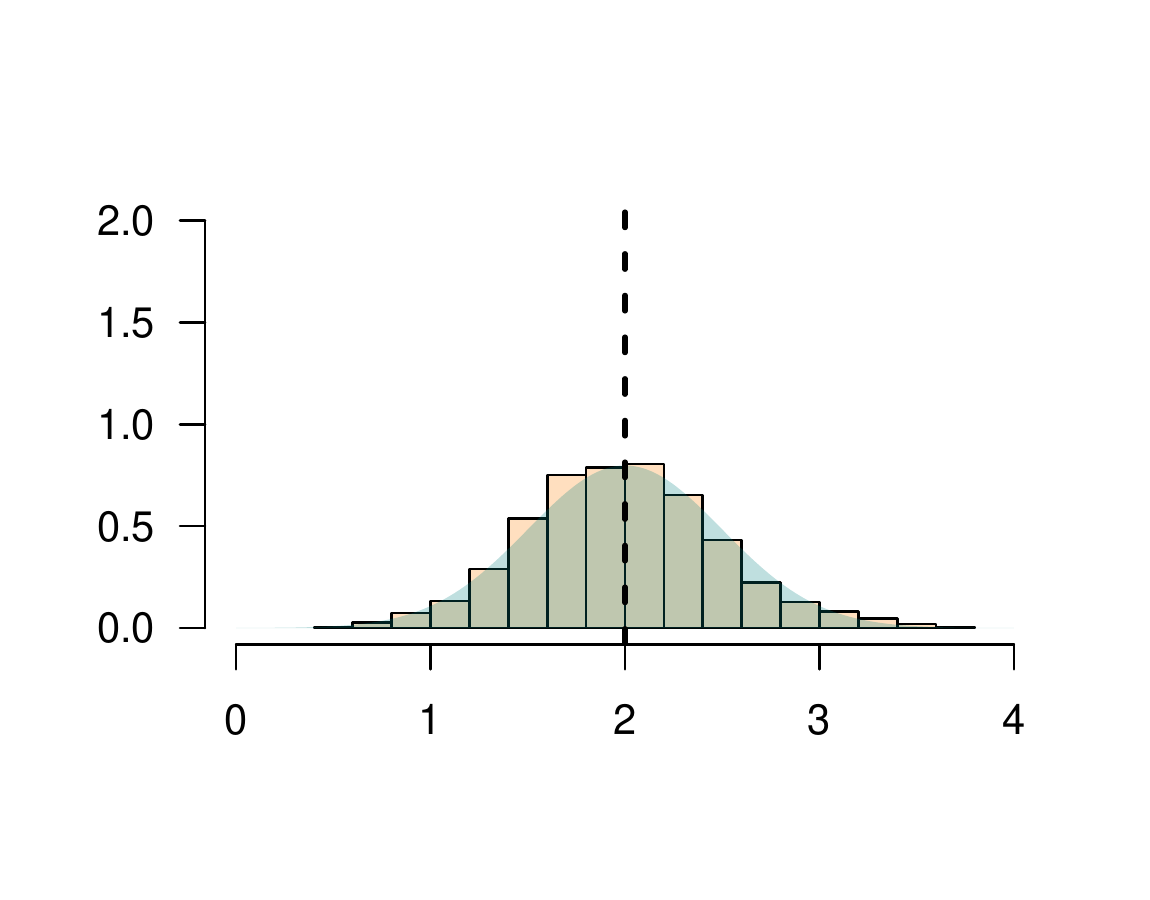}&\vspace*{-0.3in}\hspace*{-0.26in}\vspace*{-0.3in} \includegraphics[scale = 0.457]{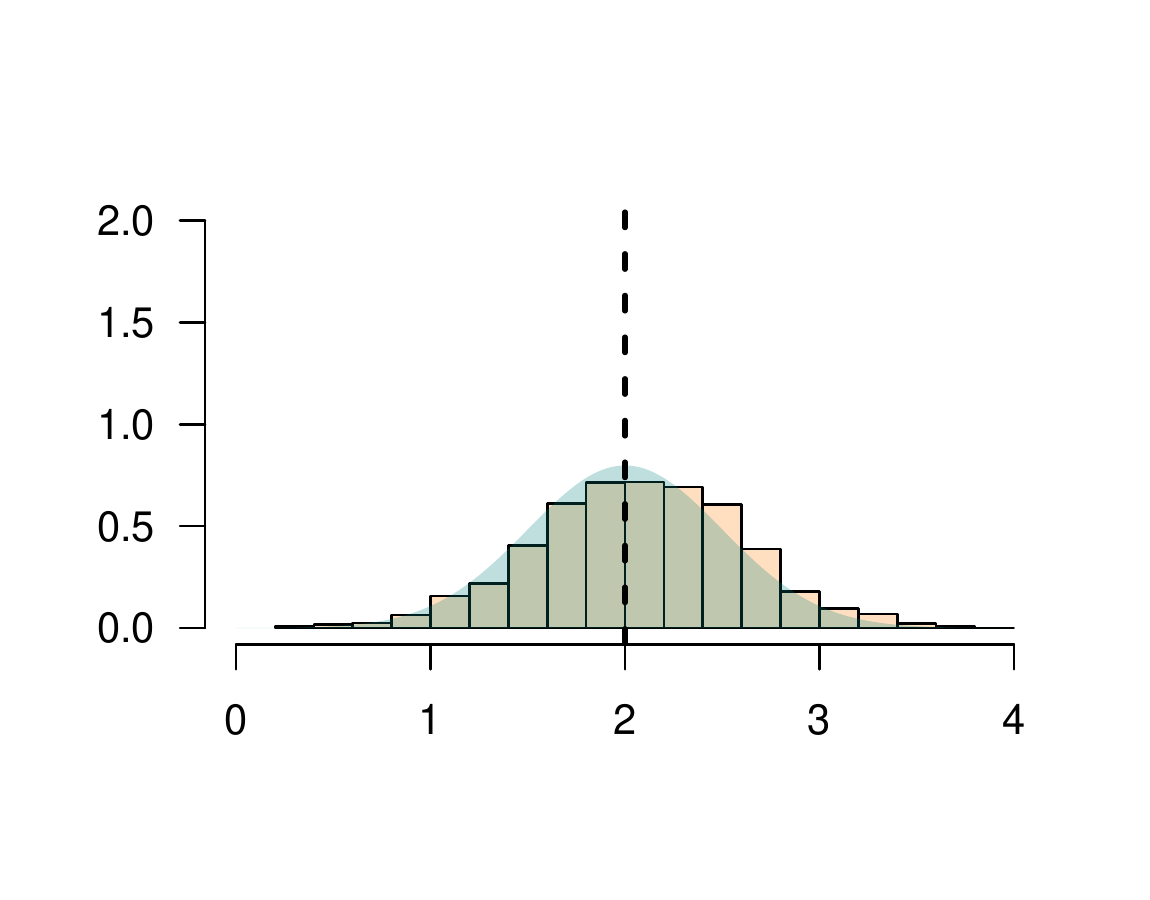}\\\hline
	\end{tabular}
	\vspace*{10pt}
	\captionof{figure}{\emph{Row 1}: Histogram of the posterior distribution of the functional of interest for the different parametrisations. \emph{Row 2}: The posterior path of $\eta$. \emph{Row 3}: The posterior path of $\a$. \emph{Row 4}: The prior and posterior of $\eta$. \emph{Row 5}: The prior and posterior of $\alpha$.}
	\label{fig:posteriors}
\end{table}

Both parametrisations seem to perform similarly, with only negligible differences in coverage and length. The posterior of $\a$ in the $P, \eta$ parametrisation seems to behave like the prior, suggesting the data has minimal influence on it. These results are important, as they confirm the practical usefulness of the model. When an expert has proper knowledge of the difference between the missing and observed population, it should not matter which parametrisation is being used. The choice should be based on the knowledge of the other parameters.

\section{Efficiency theory}
\label{SectionEfficiency}
In this section we derive the information for estimating the functional of interest when
the sensitivity function is known, 
in the nonparametric model in which the distribution of the observations is completely 
unknown.

We can parameterise the model by the pair $(p,P_1)$ and are then interested
in estimating the function $\chi(p,P_1)=(1-p)P_1(ge^q)/P_1e^q+pP_1g$, for a given
measurable function $g$. We suppose that $q$ is known, so that the functional
is identifiable, and are interested in obtaining the minimal attainable asymptotic variance
of a sequence of estimators (in the sense of the local minimax or convolution theorem,
as explained in e.g.\ \cite{vanderVaart(1991)}).

The likelihood for a single observation $(X,R)$ can be written as
$$(p,P_1)\mapsto (1-p)^{1-R} p^R p_1(X)^R.$$
A path $(p+th,P_{1,t})$ given by $dP_{1,t}\propto e^{t\phi}\,dP_1$,  gives a score function
at $t=0$ given by
$$h\frac{R-p}{p(1-p)}+R(\phi(X)-P_1\phi)=:(B_{p,P_1}^ph+B_{p,P_1}^{P_1}\phi)(R,X).$$
The tangent space of the model is by the definition the set of all scores, when $h$ varies
over $\RR$ and $\phi$ varies of $L_2(P_1)$.
Since $dP_{0,t}\propto e^q\,dP_{1,t}\propto e^qe^\phi\,dP_1\propto e^{t\phi}\,dP_0$, 
the path $P_{0,t}$ has score function $\phi-P_0\phi$. The derivative of the functional
of interest is 
\begin{align*}
	&\frac d{dt}_{|t=0}\int g\,dP_t=\frac d{dt}_{|t=0}\Bigl[(1-p_t)\int g\,dP_{0,t}+p_t\int g\,dP_{1,t}\Bigr]\\
	&\qquad=
	h\int g\,d(P_1-P_0)+(1-p)\int g(\phi-P_0\phi)\,dP_0+p\int g(\phi-P_1\phi)\,dP_1\\
	&\qquad=\E_{p,P_1}\biggl[\Bigl(\tilde h\frac{R-p}{p(1-p)}+R(\tilde\phi(X)-P_1\tilde\phi)\Bigr)
	\Bigl(B_{p,P_1}^ph+B_{p,P_1}^{P_1}\phi\Bigr)(R,X)\biggr],
\end{align*}
for the `least favourable directions' given by
\begin{align*}
	\tilde h&=p(1-p)\, (P_1-P_0)g,\\
	\tilde\phi&=\frac{1-p}p(g-P_0g)\frac {e^q}{P_1e^q}+g-P_1g=(g-P_0g)e^{\eta+q}+g-P_1g.
\end{align*}
Thus the efficient influence function is equal to 
$$B_{p,P_1}^p\tilde h+B_{p,P_1}^{P_1}\tilde \phi=(R-p)(P_0g-P_1g)+R\bigl((g-P_0g)e^{\eta+q}+(g-P_1g)\bigr).$$
The variance of this function is the least possible asymptotic variance, and is given by
\begin{align*}
	&p(1-p)(P_0g-P_1g)^2+p\,P_1\bigl((g-P_0g)e^{\eta+q}+(g-P_1g)\bigr)^2\\
	&\qquad\qquad=p(1-p)(P_0g-P_1g)^2+p\,P_1\bigl(g(1+e^{\eta+q})-P_1(g(1+e^{\eta+q}))\bigr)^2.
\end{align*}

\section{Proofs}

\begin{proof}[Proof of Theorem 8.2 (Posterior mean)]
	For $k\in\mathbb{N}$, define $\tau_k(\l \given x) = \int s^{k-1} e^{-s\left(\l + b(x)\right)} \,ds 
	= \Gamma(k) / \bigl(\l + b(x)\bigr)^k $. Then $\prod_{j=1}^{N_{j,n}}\tau_{N_{j,n}}(\l\given \tilde X_j)
	=\prod_{j=1}^{N_{j,n}}\Gamma(N_{j,n})/\prodin \bigl(\l+b(X_i)\bigr)$.
	Because the posterior mean is equal to the
	predictive distribution, Corollary~14.59 in \cite{Ghosal2017}, gives that 
	\begin{align*}
		\E\bigl(P_n(A) \given X^{(n)}\bigr) 
		= \sum_{j=1}^{K_n} p_j(X^{(n)}) \delta_{\tilde{X}_j}(A) + \int_A p_{K_n+1}(X^{(n)} \given x)\,d\overline{a}(x),
	\end{align*}
	where $\overline{a}$ is the normalisation of $a$ and the $p_j$ are the prediction probability function (PPF) of
	the exchangeable partitions generated by $P_n$, given by, for $j=1,\ldots, K_n$,
	\begin{align*}
		p_j(X^{(n)}) 
		&= \frac{N_{j,n}\int_0^\infty \l^{-1}e^{-\psi(\l)} \frac{\l }{\l + b(\tilde{X}_j)}\prod_{i=1}^n \frac{\l}{\l + b(X_i)}\,d\l}
		{n\int_0^\infty \l^{-1}e^{-\psi(\l)} \prod_{i=1}^n \frac{\l}{\l + b(X_i)}\,d\l}\\
		p_{K_n + 1}(X^{(n)} \given x) 
		&= \frac{\int_0^\infty\l^n e^{-\psi(\l)} \frac{1}{\l+b(x)} \prod_{i=1}^n \frac{\l}{\l + b(X_i)}\,d\l}
		{n\int_0^\infty \l^{-1}e^{-\psi(\l)} \prod_{i=1}^n \frac{\l}{\l + b(X_i)}\,d\l}.
	\end{align*}
	Define $p_1',\ldots, p_n'$ such that $N_{j,n}\,p_i'(X^{(n)})=p_j(X^{(n)})$ if $X_i=\tilde X_j$,
	and \\$p'_{n+1}(X^{(n)}\given x)=p_{K_n+1}(X^{(n)}\given x)$. (Hence 
	$p_i'(X^{(n)})$ is equal to the quotient in the preceding display that defines $p_j(X^{(n)})$ 
	without the multiplicity $N_{j,n}$, for $j$ the index of the partitioning set to 
	which  $X_i$ belongs.) Using this notation, we can simplify the formula for
	the posterior mean to 
	\begin{align*}
		\E\bigl(P_n(A) \given X^{(n)}\bigr) 
		= \sum_{i=1}^{n} p'_i(X^{(n)}) \delta_{X_i}(A) + \int_A p'_{n+1}(X^{(n)} \given x)\,d\overline{a}(x).
	\end{align*}
	As $\l / \bigl(\l + b(x)\bigr) \leq 1$, for all $x \in \mX$, it follows that $p_i'(X^{(n)}) \leq1/n$, for
	$i=1,2\ldots, n$, and $p'_{n+1}(X^{(n)} \given x ) \leq1/n$, for all $x \in \mX$. 
	Using this, we obtain
	\begin{align*}
		&\sup_{A \in \X}\left|\E\bigl(P_n(A) \given X^{(n)}\bigr) - \PP_n(A) \right| \\
		&\quad\leq \sup_{A \in \X} \sum_{i=1}^n \left| p'_i(X^{(n)}) - \frac{1}{n} \right| \delta_{X_i}(A) 
		+ \sup_{A \in \X}\left|\int_A p'_{n+1}(X^{(n)} \given x)\,d\overline{a}(x) \right| \\
		&\quad= 1 - \sum_{i=1}^n p'_i(X^{(n)}) + \int p'_{n+1}(X^{(n)} \given x)\,d\overline{a}(x) 
		= 2\int p'_{n+1}(X^{(n)} \given x)\,d\overline{a}(x) 
		\le \frac{2}{n},
	\end{align*}
	where we used that $\overline{a}$ is a probability measure.
\end{proof}	

\begin{proof}[Proof of Lemma 8.1]
	We first show the equality for $f = \mathbbm{1}_A$ with $A \in \X$. 
	Note that $1/y = \int_0^\infty e^{-t y}\, dt$ for any $y \in \RR^+$. This can be used to write
	\begin{align*}
		&\E_W\E_\Psi \left[\frac{\Psi(A)}{\Psi(\mX) + \sumin W_i} \right] 
		= \E_W\E_\Psi \left[\Psi(A)\int_0^\infty e^{-t\left(\Psi(\mX) + \sumin W_i\right)}\, dt \right] \\
		&\qquad= \int_0^\infty \E_W\E_\Psi \left[\Psi(A)e^{-t\Psi(A)} \cdot e^{-t\Psi(A^c)} \cdot e^{-t\sumin W_i}\right] \, dt,
	\end{align*}
	where we used Fubini's theorem in the last equality. Since 
	$\Psi$ is a Poisson process and the sets $A$ and $A^c$ are disjoint, the variables
	$\Psi(A)$ and $\Psi(A^c)$ are independent. They are also independent of the $W_i$ 
	and these weights are also independent of each other. Therefore the preceding display can be written as
	\begin{align}
		\label{eq:3parts}
		\int_0^\infty \E_\Psi \left[\Psi(A)e^{-t\Psi(A)}\right] 
		\cdot \E_\Psi \left[e^{-t\Psi(A^c)} \right] 
		\cdot \prodin \E_W\left[e^{-tW_i}\right]\,  dt
	\end{align}
	Each term can be calculated explicitly. We start with the second term in the above display. By Proposition J.6 in \cite[p.~598]{Ghosal2017}, 
	we have
	\begin{align*}
		\E_\Psi \left[e^{-t\Psi(A^c)} \right] 
		&= e^{-\int_{A^c}\int_0^\infty \left(1-e^{-ts}\right) s^{-1} e^{-s b}\, ds\, da} 
		= e^{-\int_{A^c} \log\left(\frac{t + b}{b}\right)\, da},
	\end{align*}
	where the last equality follows from Lemma~8.3 in the main paper, for any positive 
	constant $t$. In view of this, the first term in equation \eqref{eq:3parts} can be written
	\begin{align*}
		\E_\Psi \left[\Psi(A)e^{-t\Psi(A)}\right] &= -\frac{d}{dt} \E_\Psi \left[e^{-t\Psi(A)}\right] 
		= e^{-\int_{A} \log\left(\frac{t + b}{b}\right)\, da} \cdot \int_A \frac{1}{t + b}\, da.
	\end{align*}
	The third term in equation \eqref{eq:3parts} is
	\begin{align*}
		\prodin \E_W\left[e^{-tW_i}\right] 
		&= \prodin \int_0^\infty e^{-tw}\,b(X_i) e^{-w b(X_i)}\,dw= \prodin  \frac{ b(X_i)}{t + b(X_i)}.
	\end{align*}
	Combining these identities finishes the proof of the lemma for indicator functions. 
	
	When $f = \sumin a_i \mathbbm{1}_{A_i}$ for disjoint sets $A_1,\ldots,A_k$, the lemma is again true, 
	by the independence of the random variables $\Psi(A_i)$. 
	Any nonnegative measurable function $f$ can be approximated by such linear combinations $f_k$ as above with $f_k \leq f_{k+1}$ almost everywhere for each $k$ and $\limsup_{k \to \infty} f_k = f$. Then by the monotone convergence theorem, our claim is also true for any nonnegative measurable function $f$.
\end{proof}

\begin{proof}[Proof of Lemma 8.2]
	For any $x > 0$ we have
	\begin{align*}
		\max_{1 \leq i \leq n} \frac{\left|X_i \right|^r}{n} 
		\leq \frac{y^r}{n} + \frac{1}{n} \sumin \left| X_i \right|^r \mathbbm{1}_{\{|X_i| > y\}}.
	\end{align*}
	As $n \ra \infty$ the first term tends to zero, for fixed $y$, while the second term tends to 
	$\E\left|X_i \right|^r  \mathbbm{1}_{\{|X_i| > y\}}$,
	by the law of large numbers, and can be made arbitrarily small by choice of $y$.
\end{proof}

\begin{proof}[Proof of Lemma 8.3]
	Write $\psi(\l)$ for the left side of the equation. The derivative of the function $\l\mapsto \psi(\l)$
	is equal to
	\begin{align*}
		\psi'(\l) = \int\!\!\int_0^\infty s   e^{-s\l}s^{-1} e^{-sb}\,ds\,da = \int \frac{1}{\l + b}\,da.
	\end{align*}
	Since $\psi(0)=0$, the fundamental theorem of calculus gives
	$\psi(\l) = \int_0^\l \psi'(t)\,dt 
	= \int \log(1+\l/b\bigr)\, da$, by Fubini's theorem.
\end{proof}

\begin{proof}[Proof of Lemma 8.4 (Concentration of the mixing distribution)]
	We start by deriving upper and lower bounds on the posterior density given in equation~8.2 in the main paper.
	Denote the norming constant of this density by  $C_n$. 
	
	For the upper bound, we use that $\psi(\l)\ge 0$, $b(\tilde X_j)\ge \ub$,
	and $\l+\ub\ge \ub$ to find that 	
	\begin{align*}
		\pi(\l \given X^{(n)})
		&\leq \frac{1}{C_n\ub} \Bigl(\frac{\l}{\l+\ub}\Bigr)^{n-1} 
		\leq \frac{1}{C_n\ub} e^{- \ub /\delta},
	\end{align*}
	for $\l+\ub<\delta(n-1) $, where we used that $1-x\le e^{-x}$, for every $x$.
	
	For the lower bound we use that $\psi(\l) \leq a(\mX)\, \log(1+\l/\ub)$, in view of
	Lemma~8.3 in the main paper, and that
	$\l/(\l+b(\tilde X_j))\ge e^{-b(\tilde X_j)/\l}$ to find that
	\begin{align*}
		\pi(\l \given X^{(n)})
		&\geq \frac{1}{C_n\l} e^{-a(\mX)\log(1+\l/\ub)} e^{-\sum_{j=1}^{K_n} N_{j,K_n}b(\tilde X_j)/\l} \\
		&= \frac{1}{C_n\l}\Bigl(1+\frac\l\ub\Bigr)^{-a(\mX)}  e^{-n\PP_nb/\l} \\
		&\geq \frac{(\ub/2)^{a(\mX)}}{C_n\l^{1+a(\mX)}}e^{-\PP_nb/\epsilon} ,
	\end{align*}
	for $\l>n\epsilon\vee \ub$ and any $\epsilon>0$, so that $1+\l/\ub\le 2\l/\ub$.
	
	By combining the upper and lower bounds, we obtain, for $n\epsilon>\ub$,
	\begin{align*}
		\frac{\Pi(\Lambda < \delta n-\d-\ub \given X^{(n)}) }
		{\Pi(\Lambda >\epsilon n\given X^{(n)}) }
		&\le \frac{\int_0^{\delta n}\ub^{-1} e^{-\ub/\delta}\,d\l}
		{\int_{\epsilon n}^\infty \l^{-1-a(\mX)}(\ub/2)^{a(\mX)}e^{-\PP_nb/\epsilon}\,d\l} \\
		&= \delta n^{1+a(\mX)} e^{-\ub/\delta}\,e^{\PP_nb/\epsilon}a(\mX)\epsilon^{a(\mX)}\ub^{-a(\mX)-1}2^{a(\mX)}.
	\end{align*}
	Since the denominator on the left side is smaller than 1, the probability 
	$\Pi(\Lambda < \delta n-\d-\ub \given X^{(n)}) $ is also bounded by the right side. 
	Here $\PP_n b\ra P_0b$, by the strong law of large numbers, and
	$e^{-\ub /\delta}=n^{-\ub/c}$, for $\delta=c/\log n$.
	Then $n^d$ times the right side tends to zero if $1+a(\mX)+d\le \ub/c$.
	
	This is true uniformly in $b$ provided $\PP_n b$ remains bounded.
\end{proof}

\bibliographystyle{plainnat} 
\bibliography{Mybib}